\documentclass[12pt,psamsfonts,reqno]{amsart}

\usepackage{amssymb,amsfonts}
\usepackage[all,arc]{xy}
\usepackage{enumitem}
\setlist[enumerate]{topsep=2pt, itemsep=2pt}

\usepackage{mathrsfs}

\usepackage[top=1in, bottom=1.25in, left=1.25in, right=1.25in]{geometry}
\usepackage{amsfonts}
\usepackage{amsthm}
\usepackage{amsmath}
\allowdisplaybreaks

\usepackage{parskip}
\usepackage{mathrsfs}
\usepackage{graphicx}
\usepackage{amssymb}
\usepackage{xtab}
\usepackage{mathtools}

\usepackage{hyperref}
\usepackage{cleveref}
\usepackage{tikz} % https://q.uiver.app/ 
\usepackage{tikz-cd}
\usepackage{float}
\usetikzlibrary{positioning}
\usetikzlibrary{shapes.geometric,arrows}
\usepackage{mnotes}

\usepackage[none]{hyphenat}
\newtheorem{thm}{Theorem}[section]
\newtheorem{cor}[thm]{Corollary}
\newtheorem{prop}[thm]{Proposition}
\newtheorem{lem}[thm]{Lemma}

\theoremstyle{definition}
\newtheorem{defn}[thm]{Definition}

\newtheorem{exmp}[thm]{Example}

\newtheorem{condition}[thm]{Condition}
\newtheorem{Assum}[thm]{Assumption}

\theoremstyle{remark}
\newtheorem{rem}[thm]{Remark}


\numberwithin{equation}{section}

\newcommand{\cx}{\mathbb{C}}

\newcommand{\pp}{\mathbb{P}^{p-1}}

\newcommand{\z}{\mathbb{Z}}
\newcommand{\q}{\mathbb{Q}}

\newcommand{\fv}{\mathbb{F}_v}

\newcommand{\ga}{\mathbb{G}_a^n}
\newcommand{\ad}{\mathbb{A}_F}
\newcommand{\ov}{\mathfrak{o}_v}
\newcommand{\cf}{\mathcal{C}_F}
\newcommand{\cfs}{\mathcal{C}_{F,S}}
\newcommand{\vnorm}{\lVert \cdot \rVert}

\DeclareMathOperator{\rank}{rank}
\DeclareMathOperator{\red}{red}
\DeclareMathOperator{\redv}{red_v}
\DeclareMathOperator{\Dim}{dim}
\DeclareMathOperator{\pic}{Pic}

\DeclareMathOperator{\gal}{Gal}
\DeclareMathOperator{\eff}{eff}
\DeclareMathOperator{\Div}{div}

\DeclareMathOperator{\cohom}{H}
\DeclareMathOperator{\spec}{Spec}
\DeclareMathOperator{\Tr}{Tr}
\DeclareMathOperator{\Aut}{Aut}
\DeclareMathOperator{\re}{Re}
\DeclareMathOperator{\vol}{vol}
\DeclareMathOperator{\residue}{Res}
\DeclareMathOperator{\WeilRes}{Res}

\DeclareMathOperator{\Br}{Br}
\DeclareMathOperator{\inv}{inv_v}
\DeclareMathOperator{\lcm}{lcm}
\DeclareMathOperator{\av}{av}
\DeclareMathOperator{\md}{ mod }
\DeclareMathOperator{\ind}{Ind}
\DeclareMathOperator{\im}{im}
\DeclareMathOperator{\Ho}{H}

\newcommand*{\myproofname}{Proof of \Cref{thm:ShaTrivial}}

\theoremstyle{definition}

\DeclareFontFamily{U}{wncy}{}
\DeclareFontShape{U}{wncy}{m}{n}{<->wncyr10}{}
\DeclareSymbolFont{mcy}{U}{wncy}{m}{n}
\DeclareMathSymbol{\Sh}{\mathord}{mcy}{"58} 
\DeclareMathSymbol{\Be}{\mathord}{mcy}{"42}

\title[Manin's conjecture for compactifications of forms of $\ga$]{Manin's Conjecture for Equivariant compactifications of forms of $\ga$}
\author[A. Alfaraj]{Abdulmuhsin Alfaraj}
\address{Abdulmuhsin Alfaraj\\
	Department of Mathematical Sciences \\
	University of Bath \\
	Claverton Down \\
	Bath \\
	BA2 7AY \\
	UK.}
\urladdr{}

\begin{document}

\begin{abstract}
	We prove the Batyrev-Manin conjecture for smooth equivariant compactifications of forms of $\ga$ over a global function field, assuming some conditions on the boundary divisor. To verify that the leading constant agrees with Peyre's predicition we also show that a commutative unipotent group admitting a smooth equivariant compactification satisfies the Hasse principle for algebraic groups and weak approximation. We study in detail the case of $\mathbb{P}^{p-1}$, where $p$ is the characteristic of $F$, viewed as a compactification of appropriate $F$-wound groups to illustrate new phenomena appearing in the function field setting. 
\end{abstract}
\maketitle
\tableofcontents
\setcounter{equation}{0}
\section{Introduction}
Let $X$ be a Fano variety over a number field $F$. In 1989, Manin and his collaborators (see \cite{BatManin} and \cite{FMT}) proposed a conjecture that predicts the asymptotic behaviour of the number of rational points of bounded height, with respect to an ample line bundle, in terms of geometric invariants of $X$. Moreover, Peyre in \cite{Peyre1} gave a conjectural interpretation of the leading constant appearing in the asymptotic formula for the anticanoncial height. Manin's conjecture has an analogous formulation over global function fields (see  \cite{Peyre}), which is stated in terms of the analytic properties of the height zeta function, as opposed to predicting an asymptotic formula; since the height zeta function has an imaginary period in this setting, obtaining a meaningful asymptotic formula can be subtle.  So far, the conjecture was proven for a number of classes of Fano varieties over global fields, where different techniques were employed depending on the studied class (see e.g. \cite{BrowningDP}, \cite{CLSurvey}). 

In \cite{CLYT}, Chambert-Loir and Tschinkel covered the case of smooth equivariant compactifications of vector groups over number fields by using harmonic analysis tools on the locally compact group $\ga(\ad)$. In this paper, using similar techniques, we prove analogous results for smooth equivariant compactifications of unipotent groups that are forms of $\ga$ over a global function field $F$, assuming appropriate conditions on the boundary divisor. We also show that the leading constant agrees with the analogue prediction of Peyre in the case of functions fields (see \cite[\S 3.2]{Peyre}); a key part in the proof is showing that $G$ satisfies the Hasse principle for algebraic groups which we also prove.

The interest in the global function field case stems from the existence of non-trivial forms of $\ga$, and particularly $p$-torsion commutative $F$-wound groups (see \Cref{defn:Wound}).  Any non-trivial form of $\ga$ splits over a finite purely inseparable extension. This kind of behaviour causes the appearance of new phenomena in this setting. For example, if the Picard group of the form is non-trivial, then the boundary divisor with respect to any smooth compactification is not geometrically reduced (see \Cref{Cor:PicG-D-g.reduced}). This results in  complications with some of the techniques employed by Chambert-Loir and Tschinkel in \cite{CLYT}, particularly, in computing the so-called geometric Denef formula \cite[Theorem 9.1]{CLYT}. Dealing with these divisors in general presents significant difficulties. In the case of $F$-wound groups, the boundary divisor will moreover contain no local points at all places, which reflects the (analytic) compactness of the set of local points of such a group. 

To illustrate this purely inseparable situation, we will look at the case of $X=\mathbb{P}^{p-1}$ viewed as a compactification of an $F$-wound group over $\mathbb{F}_q(t)$, where $q$ is a power of $p$ (see \S\ref{section:Proj-example}). In this example, the adelic metrics associated to these divisors have unique behaviour that does not appear in characteristic zero. For example, the local $v$-adic norm of the rational section of such a divisor will attain only finitely many values; hence, the local height along the corresponding component will have finitely many values as well (see \Cref{prop:P^p-1-valuations-insep}.)

One particular consequence of the results in this paper is the following version of \cite[Theorem 0.1]{CLYT}.
\begin{thm}
	Let $X$ be a smooth equivariant compactification of $\ga$ over a global function field $F$. Assume that $X\setminus \ga$ has strict normal crossings with smooth irreducible components. Then, for any anti-canonical height function $H_{\omega_X^{-1}}$, the height zeta function $Z_{\omega_X^{-1}}(s)$ has a pole of largest real value at $s=1$ of order equal to the rank of $\pic(X)$. Moreover, $\lim_{s\rightarrow 1}(s-1)^{\rank(\pic(X))} Z_{\omega_X^{-1}}(s)$ agrees with Peyre's prediction.
\end{thm} 
Note that \cite[Theorem 0.1]{CLYT} does not assume the boundary is a strict normal crossings divisor, but they reduce to this case via resolution of singularities, which is still open in positive characteristic.

\subsection{Main Theorem} \label{subsection:assumptions} 
We now prepare to state our main result. We begin with some notation.
Let $F$ be a global function field of characteristic $p$ with constant field $\mathbb{F}_q$, $\Omega_F$ its set of places, $F_v$ its completion at a place $v$ with ring of integers $\ov$, and $\fv$ the corresponding residue field with cardinality equal to $q_v$. Let $G$ be a form of $\ga$ over $F$. Let $X$ be a smooth equivariant compactification of $G$ over $F$, i.e., a smooth projective variety over $F$ equipped with an action of $G$ having a dense orbit isomorphic to the image of $G$. Denote the boundary divisor by $D:=X\setminus G$ and let $\{ D_\alpha \}_{\alpha\in \mathcal{A}}$ be the set of (reduced) irreducible components of $D$ over $F$. Set $$\mathcal{B}:=\{\beta \in \mathcal{A}\>:\> D_\beta \text{ is not geometrically reduced}\}.$$ 
For all $\alpha\in \mathcal{A}$, we equip $\mathcal{O}(D_\alpha)$ with a smooth adelic metric $\{\lVert \cdot \rVert_v\}_{v\in\Omega_F}$ (see \Cref{defn:adelic-met}) and we fix a canonical section which we denote by $\mathsf{s}_\alpha$. Let $\mathbb{K}\subset G(\ad)$ be the maximal compact open subgroup fixing all metrized line bundles $(\mathcal{O}(D_\alpha), \lVert \cdot \rVert_v)$, $\alpha \in \mathcal{A}$. 

We denote by $\cf$ a smooth projective irreducible curve over $\mathbb{F}_q$ with function field isomorphic to $F$. We fix an appropriate (see \Cref{defn:adelic-met}) $\cfs$-model $\mathcal{X}$ of $X$, where $\cfs$ is the complement of a finite set of places $S$ in $\cf$. Let $\mathcal{D}$, $\mathcal{G}$ and $\mathcal{D}_\alpha$ be the closures of $D$, $G$, and $D_\alpha$ in $\mathcal{X}$, respectively.

\begin{Assum}\label{assum}
	We assume that $D$ satisfies the following:
	\begin{enumerate}[label=(\roman*),itemsep=1pt,parsep=1pt,topsep=4pt]
		\item $\cup_{\alpha\in \mathcal{A}\setminus\mathcal{B}} D_\alpha$ is a strict normal crossings divisor (see \cite[Tag 0CBN]{StacksEx}), where each irreducible component is smooth.
		\item For all $\beta \in \mathcal{B}$ (i.e. $D_\beta$ is not geometrically reduced):
		\begin{enumerate} [label=(\arabic*),itemsep=0pt,parsep=1pt,topsep=4pt]
			\item $D_\beta$ is geometrically irreducible;
			\item $D_\beta (F_v)=\emptyset$ for all $v\in \Omega_F$;
			\item  For all $ v\in \Omega_F\setminus S$ and for all $ x\in X(F_v)$ the value $\lVert \mathsf{s}_\beta \rVert_v (x)$ depends only on ${x \md v \in {\mathcal{X}}(\fv)}$ so that $\lVert \mathsf{s}_\beta \rVert_v (\cdot)$ is well defined on $\mathcal{X}(\fv)$ (hence $\lVert \mathsf{s}_\beta \rVert_v (\cdot)$ attains finitely many values);
			
			\item For all $v\in \Omega_F\setminus S$, 
			\begin{equation*}
				\#\{\bar{x}\in {\mathcal{D}}_\beta (\fv): \> \lVert \mathsf{s}_\beta \rVert_v(\bar{x})=q_v^{-1} \}=q_v^{\dim(X)-1}+O(q_v^{\dim(X)-3/2})
			\end{equation*}
			where the implicit constant only depends on $X$.
		\end{enumerate}
		\item If $\mathcal{B}\neq \emptyset$, then
		$$e_X:=\text{sup}_{\beta\in \mathcal{B}}\{m\in \z_{>0}:\>  \lVert \mathsf{s}_\beta \rVert_v(x)= q_v^{-m}, \text{ for some } v\in\Omega_F \text{ and } x\in X(F_v)\}$$ 
		is finite. If $\beta= \emptyset$, we set $e_X=1$.
	\end{enumerate}
\end{Assum}
These assumptions on the boundary divisor will allow us to compute the local Fourier transforms and obtain a version of Denef's formula \cite[Theorem 9.1]{CLYT} (for more details, see \S\ref{subsection:insep-assump}). 
\begin{rem} \hfill
	\begin{itemize}[itemsep=2pt,parsep=2pt,topsep=4pt]
		\item  The divisor $D$ is geometrically reduced if and only if $\pic(G)=0$ (see \Cref{Cor:PicG-D-g.reduced}). Thus if $\pic(G)$ is trivial (e.g. when $G=\ga$), we only assume $(i)$ in \Cref{assum}.
		\item If $G$ is $F$-wound then $D(F_v)=\emptyset$ for all places $v$ (see \Cref{prop:BoundaryEmpty}).
	\end{itemize}
\end{rem}

The Picard group of $X$ is free or rank $r:=|\mathcal{A}|$ such that $\text{Pic}(X)_\mathbb{Q} \cong \bigoplus_{\alpha\in \mathcal{A}}\mathbb{Q} [D_\alpha]$. We note that $\{D_\alpha\}_{\alpha\in \mathcal{A}}$ generate a subgroup of $\pic(X)$ of index equal to $|\pic(G)|$ (see \S\ref{section:PicX} for more details.) Let $\omega_X^{-1}$ be the anti-canonical line bundle and $\rho \in \bigoplus_{\alpha\in \mathcal{A}}\mathbb{Q} [D_\alpha]$ be the corresponding vector. By \Cref{prop:anticanIsEffective}, we have that ${\rho \in \bigoplus_{\alpha\in \mathcal{A}}\z_{>0} [D_\alpha]}$.

Let $\lambda=(\lambda_\alpha)\in \bigoplus_{\alpha\in \mathcal{A}}\mathbb{Q} [D_\alpha] $ be a class contained in the interior of the cone of effective divisors corresponding to a line bundle $\mathcal{L}_\lambda$. 
Denote by $H_{\lambda}$ the associated exponential height function defined on $X(F)$ (see \eqref{eqn:defnHeight} for the definition). Set ${a_\lambda:=\max\{\rho_\alpha/\lambda_\alpha\}}$ and $b_\lambda$ to be equal to the cardinality of 
${\mathcal{A}_\lambda:= \{\alpha \in \mathcal{A} \> : \> \rho_\alpha =a_\lambda \lambda_\alpha\}.}$  If $\lambda\in \bigoplus_{\alpha\in \mathcal{A}}\z [D_\alpha]$, we define
$d_\lambda:=  \gcd(\{\lambda_\alpha\}_{\alpha\in \mathcal{A}_\lambda})$ and $ g_\lambda:=\gcd(\{\lambda_\alpha \}_{\alpha\in \mathcal{A}})$ ($d_\lambda$ and $g_\lambda$ are defined for a general $\lambda$ in \Cref{cor:lambda-pole-order}). For example, when $\lambda=\rho$ we see that $\mathcal{A}=\mathcal{A}_\rho$; hence $a_\rho=1$, $b_\rho=r$, and $g_\rho=d_\rho$.
Finally, we define the domain
$\Omega:=\{ s\in \cx \> : \>   0\leq\text{Im}(s)< {2\pi i}/{\log(q)} \}.$

We now state our main theorem, which we prove using harmonic analysis techniques on the locally compact group $G(\ad)$. We note that \Cref{condition} appearing in \Cref{prop:HeightZetaPoles} asks that all characters of $G(\ad)$ with Fourier transforms possibly contributing to the main term have no poles along the components of $D$ that are not geometrically reduced.

\begin{thm}\label{prop:HeightZetaPoles}
	Let $X$, $G$ and $D$ be as above. Let $\mathcal{L}_\lambda$ be a big line bundle on $X$. Assume that \Cref{condition} holds. The series $$Z_{\lambda}(s)=\sum_{x\in G(F)} H_{\lambda}^{-s}(x)$$
	converges absolutely and uniformly for $\re(s)>a_\lambda$, and has a meromorphic continuation to $\re(s)>a_\lambda-\delta$ for some $\delta>0$. When $s$ is restricted to $\Omega$, the poles of order $b_\lambda$ and with largest real value are $$s_j = a_\lambda + j \frac{2\pi i}{d_\lambda \log(q)},$$
	for $j\in J_\lambda:=\{0,\ldots, \lceil  d_\lambda \rceil -1\}$. In particular, $c_\lambda:= \lim_{s\rightarrow a_\lambda}(s-a_\lambda)^{b_\lambda} Z_{\lambda}(s)$ is non-zero.
\end{thm}

\begin{rem}\hfill
	\begin{enumerate}[label=(\roman*)]
		\item If $\pic(G)$ is trivial (e.g. $G=\ga$), or $G(\ad)=G(F)+\mathbb{K}$, then \Cref{condition} holds.
		\item We note that  $Z_\lambda(s)$ can have other poles along $\re(s)=a_\lambda$ of order less than $b_\lambda$ (see \Cref{cor:lambda-pole-order}).
		\item $Z_\lambda(s)$ has imaginary period $\frac{2\pi i}{g_\lambda \log(q)}$ (see \Cref{lem:Period}). Thus, if $\lambda\in \bigoplus_{\alpha\in \mathcal{A}}\z [D_\alpha]$ and $d_\lambda\mid g_\lambda$, then  for all $j\in J_\lambda$ we have $\lim_{s\rightarrow s_j}(s-s_j)^{b_\lambda} Z_{\lambda}(s)=c_\lambda$.
	\end{enumerate}
\end{rem}

For a positive integer $M$, we define $$N(\mathcal{L}_\lambda, M):= \# \{x\in G(F)\> : \> H_{\mathcal{L}_\lambda}(x)=q^M \}.$$
Due to the imaginary periodic nature of $H_\lambda^{-s}(\cdot)$, we will study the asymptotic of the following weighted average counting function
$$N_{\av} (\mathcal{L}_\lambda, M):=\frac{1}{d_\lambda} \sum_{j=0}^{d_\lambda-1}  q^{-a_\lambda j} N(\mathcal{L}_\lambda, M+j) .$$ This averaged count resolves issues arising from the differences in the periods of the poles and the period of $Z_\lambda(s)$ by averaging over the  poles of largest order lying in the imaginary period $2\pi i/\log(q)$.

By applying an appropriate Tauberian theorem (\Cref{thm:Tauberian}) to \Cref{prop:HeightZetaPoles}, we obtain the following.

\begin{cor}\label{thm:Manin}
	Take the assumptions in \Cref{prop:HeightZetaPoles}. If $\lambda=(\lambda_\alpha) \in \bigoplus_{\alpha\in \mathcal{A}}\z [D_\alpha]$, then 
	$$ N_{\av}(\mathcal{L}_\lambda, M) \sim c_\lambda \frac{ \log(q)^{b_\lambda}}{(b_\lambda-1)!} q^{a_\lambda M} M^{b_\lambda-1}, \quad \text{ as } M\rightarrow \infty.$$
	Suppose further that $d_\lambda\mid g_\lambda$. Then, 
	\begin{enumerate}[itemsep=2pt,parsep=2pt,topsep=4pt]
		\item $N(\mathcal{L}_\lambda,M)=0$  for all $M\in\z_{>0}$ such that $d_\alpha \nmid M$.
		\item $N(\mathcal{L}_\lambda, d_\lambda M) \sim d_\lambda c_\lambda \frac{ \log(q)^{b_\lambda}}{(b_\lambda-1)!} q^{a_\lambda d_\lambda M} (d_\lambda M)^{b_\lambda-1}$ as $M\rightarrow \infty.$
	\end{enumerate}
\end{cor}

\begin{rem}\hfill
	\begin{enumerate} [label=(\roman*),itemsep=2pt,parsep=2pt,topsep=4pt]
		\item Observe that the asymptotic of the averaged counting function  gives a more natural leading term. This asymptotic is analogous to Peyre's prediction in the number field case.
		\item If $d_\lambda\mid g_\lambda$ and $g_\lambda,d_\lambda\in \z$, then by the definition of $H_\lambda$, one sees that $H_\lambda(x)$ is always a power of $q$ with exponent divisible by $g_\lambda$ (hence by $d_\lambda$). Thus, if $d_\lambda \nmid M$, then $N(\mathcal{L}_\lambda, M)=0$.
	\end{enumerate}
\end{rem}

\subsection{The leading constant} 
Peyre in \cite[\S 3.2]{Peyre} gives a conjectural interpretation for $c_\rho$, the leading constant appearing in \Cref{prop:HeightZetaPoles} for the anti-canonical height. The prediction by Peyre is that $c_\rho$ should equal
$$ \alpha^*(X) \beta(X) \tau_X(X(\ad)^{\Br}),$$
where $\beta(X):=|\Ho^1(F,\pic(X^s))|$, $\tau_X(X(\ad)^{\Br})$ is the measure of Brauer Manin set with respect to the Tamagawa measure on $X$ (\Cref{defn:Peyre's-constant}), and 
$$\alpha^*(X):= \int_{\Lambda_{\text{eff}}(X)^\vee } e^{-\langle \omega_X^{-1}, y \rangle} dy$$
is the (modified) effective cone constant.
For the class of varieties we are studying, $\Br(X)=\Br(F)$ (see \Cref{prop:picX}). So, $\beta(X)=1$ and ${\tau_X(X(\ad)^{\Br})=\tau_X(X(\ad))}$

\begin{thm}\label{thm:Leading-constant}
	Take the assumptions and notation in \Cref{prop:HeightZetaPoles}. Let $\rho\in \pic(X)$ be the class of the anticanonical line bundle. Then
		$$c_\rho =  \alpha^*(X) \tau_X(X(\ad))$$
		which agrees with Peyre's prediction. 
\end{thm}

A key part of the proof of \Cref{thm:Leading-constant} is to show that the Tamagawa number of $G$ is equal to $|\pic(G)|$. This will amount to showing that the Tate-Shafarevich group of $G$ is trivial, which will follow from the following theorem. We note that, in general, the Tate-Shafarevich group of a connected unipotent group over $F$ need not be trivial (see \cite[Theorem 1.8]{ZRos1}).

\begin{thm}\label{thm:ShaTrivial}
	Let $G$ be a connected commutative unipotent $F$-group that admits a smooth equivariant compactification. Then $G$ satisfies the Hasse principle for algebraic groups and weak approximation. 
\end{thm}

We prove the result in \S \ref{sec:ShTrivial} by showing that for any $G$-torsor $Y$ over $F$, we have that $\Be_\omega(Y)$ is trivial (see \eqref{eqn:Be(Y)Defn} for the definition). The result then follows from the work of Đonlagić \cite[Theorem 4.5]{Azur} where he shows that the Brauer-Manin obstruction given by $\Be_\omega(Y)$ is the only one to weak approximation on $Y$.

\subsection{Purely inseparable examples}

The following example illustrates the complications that arise when studying compactifications of non-trivial forms of $\ga$, particularly, dealing with the inseparable behaviour of the boundary divisor.

\begin{cor}\label{thm:resGm} 
	Suppose that $F=\mathbb{F}_q(t)$ and $F^{1/p}=\mathbb{F}_q(t^{1/p})$, where $q$ is a power of $p$. Let $G=\WeilRes_{F^{1/p}/F} \mathbb{G}_m / \mathbb{G}_m$. This is a non-trivial twist of $\mathbb{G}_a^{p-1}$ that splits over $F^{1/p}$. Then $X:=\mathbb{P}^{p-1}$ has the structure of an equivariant compactification of $G$ such that the boundary divisor $D$, given by $\sum_{i=0}^{p-1} t^i x_i^p=0$, is geometrically irreducible but not geometrically reduced. Moreover, there exists a choice of adelic metrics such that $D$ satisfies the conditions of \Cref{prop:HeightZetaPoles}. In particular, we have that
	$$N(\omega_X^{-1}, M)  \sim \frac{1}{p}  \left( q^{p-1} \cdot \residue_{s=1} \zeta_{F}(s) \cdot \prod_{v} C_v  \right) \log(q) \cdot q^{M}, \quad \text{ as } M\rightarrow \infty,$$
	
	where $C_\infty= (1-q^{-1}) p q^{-(p-1)}$ and
	$$C_v =\left(1-\frac{1}{q_v}\right) \left(1+ \frac{1}{q_v} + \frac{1}{q_v^2} +\cdots + \frac{1}{q_v^{p-1}} \right)  \text{ for } v\neq \infty.$$
\end{cor}
The main step in the proof of \Cref{thm:resGm} is checking that the boundary divisor satisfies  \Cref{assum}, which is the content of \Cref{prop:P^p-1-valuations-insep}. 

\begin{rem}
	One can prove an analogous result for $\WeilRes_{F^{1/{p^k}}/F} \mathbb{G}_m / \mathbb{G}_m \hookrightarrow \mathbb{P}^{p^k-1}$, where $F^{1/{p^k}}:= F(t^{1/{p^k}})$ for any $k\in \z_{>0}$, by a relatively straight forward generalisation of the proof.
\end{rem}

\begin{rem}
	Notice that $N_{\av}(\omega_X^{-1}, M)=N(\omega_X^{-1}, M)$ in \Cref{thm:resGm}, which follows since $d_\rho=1$.
\end{rem}

The following is a basic class of examples where the boundary divisor contains both a geometrically reduced component and non-geometrically reduced component.

\begin{exmp}
	Let $G'=G\times \ga$ where $G$ is the $F$-wound group in \Cref{thm:resGm} and $F=\mathbb{F}_q(t)$ for some $q=p^t$ and $t\in \z_{>0}$. Then $X:=\mathbb{P}^{p-1}\times \mathbb{P}^n$ can be viewed as an equivariant compactification of $G'$ where $G$ acts on $\mathbb{P}^{p-1}$ as in \Cref{thm:resGm}, and $\ga$ acts on $\mathbb{P}^n$ via translations. Let $D_\beta:= \mathbb{P}^{p-1}\setminus G$ and $D_\alpha= \mathbb{P}^{n}\setminus \ga$. Then $D:=X\setminus G'$ consists of two irreducible components, one isomorphic to $D_\beta\times \mathbb{P}^n$, which is not geometrically reduced, and the other component is isomorphic to $\mathbb{P}^{p-1}\times D_\alpha$, which is geometrically reduced. By \Cref{thm:resGm}, it is not difficult to see that $X$ satisfies the conditions of \Cref{prop:HeightZetaPoles}.
\end{exmp}

\subsection{The outline of the paper}
We start in \S\ref{sec:formsOfGa} by recalling relevant properties of $F$-forms of $\ga$ and study the existence of smooth compactifications along with the behaviour at the boundary. Following that, in \S\ref{section:PicX}, we study the Picard group and the Brauer group of any smooth compactification $X$ of $G$. In \S\ref{sec:ShTrivial}, we prove \Cref{thm:ShaTrivial}. In \S\ref{sec:Metr-Heights-Measures} we define adelic metrics, the height function, and relevant measures on $G(\ad)$ and $X(\ad)$; we also prove relevant properties and results. In \S\ref{sec:PoissonSum}, we cover some harmonic analysis on $G(\ad)$ and show a correspondence between additive characters of $G(\ad)$ and elements of the Brauer group of $G$. After that, we compute the Fourier transforms for the trivial character in \S\ref{sec:FT-TrivialChar} and the non-trivial characters in \S\ref{sec:FT-NonTrivialChar}. In \S\ref{sec:Asump-LeadCons}, we apply the Poisson summation formula to prove \Cref{prop:HeightZetaPoles} and \Cref{thm:Leading-constant}, where we subsequently apply an appropriate Tauberain theorem to conclude \Cref{thm:Manin}. We end the paper by proving \Cref{thm:resGm} in \S\ref{section:Proj-example}, which provides an illustration of the new behaviour appearing in this setting. 

\vspace{10 pt}

\noindent\textbf{Acknowledgements.} I would like to thank my supervisor Daniel Loughran for suggesting this problem and for his endless support. I would like to thank Antoine Chambert-Loir for useful discussions while hosting my visit to Institut de mathématiques de Jussieu-Paris Rive Gauche. I would also like to thank Azur Đonlagić for the proof of \Cref{lem:BrY-alg} and for useful discussions. Finally, I would like to thank Michel Brion, Gebhard Martin, Zev Rosengarten, and Jiazhi He for useful discussions.

\section{Non-trivial forms of $\ga$ in positive characterstic}\label{sec:formsOfGa}

In this section, we start by stating some properties of smooth connected $p$-torsion commutative algebraic groups over a field of positive characterstic $F$; these are precisely the $F$-forms of $\ga$ (see \cite[Lemma 1.7.1]{KMT}). We will also be stating special results when $F$ is a global function field.

\subsection{$F$-wound groups and forms of $\ga$}
Over a field $F$ of characterstic $0$, every commutative unipotent algebraic group over $F$ is isomorphic to $\ga$ for some $n\in \z_{>0}$. When $\text{char}(F)=p>0$, this is no longer the case. Nonetheless, if $G$ is a $p$-torsion commutative unipotent algebraic group over $F$, then it is a direct product of $\ga$, for some $n\in \z_{>0}$, and a so-called $F$-wound group (see \cite[V.5]{Ost}).

Let $F$ be a field of positive characteristic $p$.

\begin{defn}[{{\cite[Definition 4.6]{KMT}}}]\label{defn:Wound}
	A unipotent algebraic group $G$ over $F$ is defined to be $F$-wound if every $F$-morphism of algebraic groups $\mathbb{G}_a\rightarrow G$ is constant.	
\end{defn}

If $F$ is infinite, then every commutative $p$-torsion $F$-wound group of dimension $n$ is isomorphic to a hypersurface of $\mathbb{G}_a^{n+1}$, which is given by what is called a separable $p$-polynomial \cite[V.6.3]{Ost}.
\begin{defn}[{{\cite[ V.6.1]{Ost}}}]
	A polynomial $P\in F[x_1,...,x_{n+1}]$ is said to be a separable $p$-polynomial if
	$$P=\sum_{i,j}^{m,r} c_{i,j} x_i^{p^j}$$ 
	and at least one $c_{i,0}\neq 0$, i.e., at least one monomial of degree $1$ appears in $P$. We define the principal part of $P$, denoted by $P_{\text{princ}}$, to be the sum of highest degree monomials in each $x_i$ if it appears, and $0$ otherwise.   
\end{defn}

\begin{exmp}[{{\cite[V.3.4]{Ost}}}]
	Let $F$ be a global function field of characteristic $p$. Fix $t\in F\setminus F^p$. The subgroup $G$ of $\ga \times \ga$ given by the equation
	$$x^p-ty^p=x$$
	is $F$-wound; it splits over the inseparable extension $F(t^{1/p})$. If $p>2$, then $G(F)$ is finite (see \Cref{thm:wound-finitepts-lowdim} below). If $p=2$, then any connected unipotent group over $F$ of dimension $1$ with $G(F)$ infinite is either isomorphic to $\mathbb{G}_a$ or $G$ (\cite[Proposition VI.3.4]{Ost}). 
\end{exmp}

\begin{prop}\label{cor:ex.non-van.diff}
	Let $G$ be an $F$-form of $\ga$. Then $G$ is a subgroup of $\mathbb{G}_a^{n+1}$ that is defined by a separable $p$-polynomial. Moreover, there exists a non-vanishing top differential $\omega \in \omega_G(G)$ that is translation invariant.
\end{prop}
\begin{proof}
	By {{\cite[V.5]{Ost}}}, $G=\mathbb{G}_a^m \times W$ where $W$ is $F$-wound and $m\leq n$. By \cite[V.6.3]{Ost}, we have that $W$ is a subgroup of $\mathbb{G}_a^{n-m+1}$ defined by a separable $p$-polynomial $P\in F[x_1,...,x_{n-m+1}]$. The image of $P$ via the injection $F[x_1,...,x_{n-m+1}]\rightarrow F[x_1,...,x_{n+1}]$ defines $G$ in $\mathbb{G}_a^{n+1}$.
	
	Now, as $P$ is a separable $p$-polynomial, without loss of generality, we can assume that the coefficient of $x_1$ in $P$ is non-zero. Let $M$ be the $F[x_1,\ldots,x_{n+1}]$-submodule of $\Omega_{\mathbb{G}_a}^{n+1}(\mathbb{G}_a^{n+1})=\bigoplus_{i=1}^{n+1} F[x_1,\ldots, x_{n+1}] dx_i$ generated by $(-c_{1,0} dx_1- \sum_{j=2}^{n-m+1} c_{j,0} dx_j) $ with $c_{1,0} \neq 0$. Then, $\Omega_G(G)\cong \left( \bigoplus_{i=1}^{n+1} \mathcal{O}_G(G) dx_i \right) /M \cong \bigoplus_{i=2}^{n+1} \mathcal{O}_G(G) dx_i$  which implies that $\Omega_G(G)$ is a free $\mathcal{O}_G(G)$-module of rank $n$ and basis given by the classes of $\{dx_i\}_{i=2}^{n+1}$. Thus, one easily sees that (the class of) $\omega:= dx_2 \wedge \cdot\cdot\cdot \wedge dx_{n+1}$ is a translation invariant regular top differential that does not vanish on $G$.
\end{proof}

\subsection{$F$-wound groups over a global function field}
For the rest of this section, we assume that $F$ is a global field of characteristic $p$, and $F_v$ is the completion of $F$ at a place $v$. Note that by \cite[V.7]{Ost} if $G$ is a commutative $F$-wound group, then $G_v:=G \times_F F_v$ is $F_v$-wound. By \cite[VI.2.1]{Ost}, $G$ is $F$-wound if and only if $G(F_v)$ is compact for all $v$. Also, if $G$ is $F$-wound, then it splits over a unique purely inseperable extension of minimal degree (see e.g. \cite[Proposition 7.4]{Rosen}).

\begin{thm}[{{\cite[VI.3.1]{Ost}}}]\label{thm:wound-finitepts-lowdim}
	Let $G$ be an $F$-wound unipotent group. If the dimension of $G$ is $<p-1$, then $G(F)$ is finite.
\end{thm}

This result is optimal in the sense that there exists an $F$-wound group $G$ of dimension $p-1$ and with infinitely many $F$-point.

\begin{exmp}[{\cite[VI.3.4]{Ost}}]\label{example:ResGm}
	Let $t\in F\setminus{F^p}$ and define $F^{1/p}:=F(t^{1/p})$. Let ${G:=\WeilRes_{F^{1/p}/F} \mathbb{G}_m / \mathbb{G}_m}$, the quotient of the Weil restriction of $\mathbb{G}_{m,F^{1/p}}$ with respect to the extension $F^{1/p}/F$ by $\mathbb{G}_{m,F}$. This group is $F$-wound of dimension $p-1$ and has infinitely many $F$-points. Indeed, by Hilbert's 90 theorem, $G(F)$ is in bijection with $(F^{1/p})^\times/F^\times$. 
\end{exmp}

\subsection{Smooth compactifications of $F$-wound groups}
We now attempt to investigate the existence of smooth compactifications of $F$-wound groups. Finding a smooth compactification of a given $F$-wound group does not appear to be easy at a first glance. Nonetheless, we have the following example.

\begin{prop}\label{prop:ResGm-SmComp}
	Let $G=\WeilRes_{F^{1/{p^n}}/F} \mathbb{G}_m / \mathbb{G}_m$. Then there is an open immersion ${\iota:\> G \longrightarrow \mathbb{P}^{{p^n}-1}}$ where the action of $G$ extends to $\mathbb{P}^{{p^n}-1}$ such that $G$ fixes the boundary divisor $D:= \mathbb{P}^{{p^n}-1}\setminus G$ given by the equation $ \sum_{i=0}^{{p^n}-1} t^i X_i^{p^n}= 0 .$
\end{prop}
\begin{proof}
	The $F$-variety $\WeilRes_{F^{1/{p^n}}/F} \mathbb{G}_m$ is isomorphic to the open subscheme of $\mathbb{G}_a^{p^n}\setminus \{0\}$ given by $\sum_{i=0}^{{p^n}-1} t^i y_i^{p^n} \neq 0$, where $y_0,...,y_{{p^n}-1}$ are the coordinates of $\mathbb{G}_a^{p^n}$, with the standard action of $\mathbb{G}_{m,F}$ on $\mathbb{G}_a^{p^n}$. By taking the quotient by the action of $\mathbb{G}_m$, this induces an open immersion
	$$\phi:\> G \rightarrow \mathbb{P}^{{p^n}-1},$$ where $\mathbb{P}^{{p^n}-1}\setminus \phi(G)$ is given by $ \sum_{i=0}^{{p^n}-1} t^i X_i^{p^n}= 0 .$
\end{proof}

A necessary condition for an $F$-form of $\ga$ to admit a smooth compactification is the following.

\begin{prop}\label{prop:smoothCompImpliesInfPts}
	Let $G$ be a commutative unipotent group. If there exists a smooth compactification $X$ of $G$ over $F$, then the maximal $F$-unirational subgroup of $G$ is non-trivial. In particular,  $G(F)$ is infinite.
\end{prop}
\begin{proof}
	This follows by \cite[10.3, Theorem 1]{Neron}.
\end{proof}

\begin{cor}
	Let $G$ be a commutative $F$-wound group of dimension less than $p-1$. Then there does not exist any smooth compactification of $G$.
\end{cor}
\begin{proof}
	By \Cref{thm:wound-finitepts-lowdim}, $G(F)$ is finite. The result now follows by \Cref{prop:smoothCompImpliesInfPts}.
\end{proof}

\begin{rem}
	Given a $p$-torsion commutative $F$-wound group $G$, attempting to construct a smooth compactification for it (if one exists) appears to be difficult. 
	One naive attempt is to use \Cref{cor:ex.non-van.diff}, which implies that $G$ is a subgroup of $\mathbb{G}_a^{n+1}$ defined by a separable $p$-polynomial $P$. The homogenization of $P$ in $\mathbb{P}^{n+1}$ defines a compactification $X$ of $G$. It is not difficult to check that, when $n>1$, this compactification is never smooth. The boundary divisor is everywhere singular and not geometrically reduced.
\end{rem}

We end this section with the following result which shows that if $G$ is $F$-wound and admits a smooth compatficiation then the boundary divisor contains no local points at all places.
\begin{prop}\label{prop:BoundaryEmpty}
	Let $G$ be an $F$-wound group with a smooth compactification $X$. Then $G(F_v)=X(F_v)$ for all places $v$.
\end{prop}
\begin{proof}
	By \cite[VI.2.1]{Ost}, $G(F_v)$ is compact for all places $v$. As $X(F_v)$ equipped with the analytic topology is Hausdorff, we deduce that $G(F_v)$ is a closed analytic subset of $X(F_v)$. By setting $D:=X\setminus G$, this implies that $D(F_v)$ is an (analytic) open subset of $X(F_v)$. Suppose that $D(F_v)$ is non-empty and fix $x\in D(F_v)$. Then there exists an analytic open subset $U\subset D(F_v)$ containing $x$ that is homeomorphic to a basic open subset $V$ of $\prod^n F_v$. As the measure of $V$ is positive with respect to the canonical measure and $D$ has dimension $n-1$, this gives a contradiction. 
\end{proof}
\begin{cor}
	Let $G$ be an $F$-wound group with a smooth compactification $X$ and a boundary divisor $D$. Then the smooth locus of $D$ is empty.
\end{cor}
\begin{proof}
	If the smooth locus of $D$ is non-empty, then by the Lang-Weil estimates and Hensel's lemma we deduce that $D(F_v)\neq \emptyset$ for almost all places $v$; this contradicts \Cref{prop:BoundaryEmpty}.
\end{proof}

\section{The Picard group of $X$} \label{section:PicX}
Let $F$ be a field of a positive characteristic $p$ and fix an algebraic closure $\bar{F}$ of $F$. Define $F^{1/p^m}:=\{a\in \bar{F} \> : \> a^{p^m}\in F\}$. In this section, we assume that $F$ has degree of imprefection $1$, i.e. $[F^{1/p}:F]=[F:F^p]=p$ (see e.g. \cite[\S 1]{BeckerMac}). In particular, this implies that $F^{1/p^m}$ is the unique purely inseparable extension of $F$ of degree $p^m$. Note that a global function field has degree of imprefection equal to $1$ (see \cite[Proposition 7.4]{Rosen}) and hence so does each of its local fields (\cite{Bastos}).

Let $G$ be a connected commutative unipotent $F$-group of dimension $n$. Then the base change of $G$ to the perfect closure of $F$ is isomorphic to affine space of dimension $n$ (see e.g. \cite[A.3]{KMT}). This isomorphism descends to some finite purely inseparable extension of $F$, which we can take to be minimal as $F$ has degree of imprefection $1$.

Let $Y$ be a $G$-torsor over $F$, $X$ be a smooth compactification of $Y$, and ${D:=X\setminus Y}$. Let $\{ D_\alpha \}_{\alpha\in \mathcal{A}}$ be the set of irreducible components of $D$ over $F$ and define ${\mathcal{B}:=\{\beta \in \mathcal{A}\>:\> D_\beta \text{ is not geometrically reduced}\} }$.
Fix an algebraic closure $\bar{F}$ of $F$ and let $F_s$ be the separable closure of $F$ in $\bar{F}$. Set $\Gamma_F:=\gal(F_s/F)$. Denote by $F'$ the smallest purely inseparable over which $G$ is isomorphic to affine space and such that $D_\alpha \times_F F'$ is supported on a geometrically reduced divisor for all $\alpha\in \mathcal{A}$. Let $p':=[F':F]$ and $\Gamma_{F'}:=\gal(F'_s/F')$. For a variety $V$ over $F$, we write $V_{L}:=V\times_F L$ for a field extension $L$ of $F$.

As $F'/F$ is purely inseparable, the base change morphism $X_{F'}\rightarrow X$ is a homeomorphism. Thus, for all $\alpha\in \mathcal{A}$, the reduced scheme structure of $D_{\alpha,F'}$, which we denote by $D_{\alpha,F'}^{\red}$, is identified with a prime Weil divisor on $X_{F'}$.  By \cite[Proposition 7.1.38]{Liu}, we have the following equality of Weil divisors on $X_{F'}$
$$ D_{\alpha,F'} = p_\alpha D_{\alpha,F'}^{\red},$$
where $p_\alpha= p'/[F'(D_{\alpha,F'}^{\red}):F(D_\alpha)]$. Note that $p_\alpha$ is the degree of the purely inseparable closure of $F$ in $F(D_\alpha)$, which is the smallest extension of $F$ over which $D_\alpha$ is supported on a geometrically reduced divisor.
If $\alpha \in \mathcal{A}\setminus \mathcal{B}$, then $D_\alpha$ is geometrically reduced so that $[F'(D_{\alpha,F'}^{\red}):F(D_\alpha)]=p'$; this implies that $p_\alpha=1$, i.e., that $D_{\alpha,F'}$ is a prime divisor on $X_{F'}$.
If $\beta\in \mathcal{B}$, then $D_\beta$ is not geometrically reduced, which implies that $p_\beta>1$.

\begin{lem}[Rosenlicht's lemma]\label{lemma:Rosenlicht}
	Let $G$ be a connected commutative unipotent group over $F$. Let $Y$ be a $G$-torsor over $F$ and $X$ a smooth compactification of $Y$.
	If $f\in F(X)$ has no zeroes or poles on $Y$, then $f\in F^*$.
\end{lem}
\begin{proof}
	First, suppose that $Y=G$.
	Let $0_G$ be the identity of $G$. By \cite[Theorem 3]{Rosenlicht}, $f/f(0_G)$ induces a homomorphism $G\rightarrow \mathbb{G}_m$. As $G$ has no non-trivial characters, $f=f(0_G)\in F^*$.
	
	For the general case, let $K/F$ be a finite separable extension such that $Y_K \cong G_K$. We have an embedding $F(X)\hookrightarrow K(X_{K})$. By \cite[Proposition 1.38]{Liu}, the image of $f$ in $ K(X_{K})$ has no zeroes or poles on $G_K$ so that $f\in K^*$ by applying the lemma to $G_K$. As $X$ is a geometrically integral $F$-variety, $F$ is algebraically closed in $F(X)$. Hence $f\in K^* \cap F(X) = F^*$, the intersection being taken in $ K(X_{K})$.
\end{proof}

\begin{lem}\label{lem:picSepIsom}
	We have the following isomorphisms 
	$$\pic(X_{F_s})\xrightarrow{\sim} \pic(X_{F'_s})  \xrightarrow{\sim}  \pic({X}_{\bar{F}}),$$
	which are induced by the pullback maps.
\end{lem}
\begin{proof}
	First, note that $Y_{\bar{F}}$ is a trivial $G_{\bar{F}}$-torsor so that $Y_{\bar{F}}\cong G_{\bar{F}}$. Since $\bar{F}$ is perfect, we have $Y_{\bar{F}}\cong G_{\bar{F}} \cong \mathbb{A}^n_{\bar{F}}$.
	Thus, $X_{\bar{F}}$ is a rational $\bar{F}$-variety as it is a smooth compactification of $Y_{\bar{F}}\cong \mathbb{A}^n_{\bar{F}}$. By flat base change, we have
	$$\Ho^1(X,\mathcal{O}_X)\otimes \bar{F}\cong \Ho^1(X_{\bar{F}},\mathcal{O}_{X_{\bar{F}}}) = 0,$$
	as $X_{F'}$ is a smooth rational $\bar{F}$-variety.
	Since $F'$ is a faithfully flat $F$-module, this implies that $\Ho^1(X,\mathcal{O}_X)= 0.$ By \cite[Corollary 5.1.3]{CTSK}, we deduce that
	$$\pic(X_{F_s})\xrightarrow{\sim} \pic(X_{\bar{F}}).$$
	As $X$ is projective, we have
	$ \pic (X_{F_s}) \hookrightarrow \pic(X_{F'_s})$ and $\pic (X_{F'_s})\hookrightarrow \pic(X_{\bar{F}}),$
	which completes the proof.
\end{proof}

Recall that $\mathcal{A}$ is an index set for the set of irreducible components of $D$ over $F$ and $\mathcal{B}\subset \mathcal{A}$ is the index subset for those that are not geometrically reduced. For $\alpha\in \mathcal{A}$, we denote the separable closure of $F$ in $F(D_\alpha)$ by $F_\alpha$. We write $\ind_{F_\alpha}^F (\cdot)$ for the induced module from $\Gamma_{F_\alpha}$ to $\Gamma_F$.

\begin{prop}\label{prop:picGalModule}
	We have the following isomorphisms of $\Gamma_F$-modules
	$$\pic(X^s)\cong \bigoplus_{\alpha \in \mathcal{A}} \ind_{{F_\alpha}}^{F} \mathbb{Z} \> \text{ and } \> \pic(Y^s)\cong \bigoplus_{\beta \in \mathcal{B}} \ind_{{F_\beta}}^{F} \mathbb{Z}/p_\beta \mathbb{Z}.$$
	In particular, 
	$\Ho^1(\Gamma_{F}, \pic(X_s))$ is trivial as $\pic(X^s)$ is torsion free.
\end{prop}
\begin{proof}
	First, we study the $\Gamma_F$ and $\Gamma_{F'}$ actions on the irreducible components of $D$ and $D^{\text{red}}_{F'}$, respectively. As $F$ has degree of imperfection $1$, we can write $F'=F(t^{1/p'})$ for some $t\notin F^{p}$. Any automorphism of $F_s$ extends uniquely to $F'_s$ by fixing $t^{1/p'}$.  This gives a canonical isomorphism $\Gamma_{F}\xrightarrow{\sim} \Gamma_{F'}$. Now fix $\alpha\in \mathcal{A}$. We denote the separable closure of $F'$ in $F'(D_{\alpha,F'}^{\red})$ by $F'_\alpha$. Write $\widetilde{F}_\alpha$ and $\widetilde{F}'_\alpha$ for the Galois closures of $F_\alpha$ and $F'_\alpha$, respectively. The algebraic closure of $F$ in $F(D_\alpha)$ is $F_\alpha (t^{1/p_\alpha})$ and  $F'(D_{\alpha,F'}^{\red})=F(D_\alpha)[(t^{1/p_\alpha})^{p_\alpha/p'}]$.  Thus, $F'_\alpha/ F_\alpha$ and $\widetilde{F}'_\alpha/ \widetilde{F}_\alpha$ are purely inseparable extensions of degree $p'$, i.e. $F'_\alpha=F_\alpha \otimes_{F} F'$ and ${\widetilde{F}'_\alpha=\widetilde{F}_\alpha \otimes_{F} F'}$. Hence, we have canonical isomorphisms $\Aut(F_\alpha/F)\xrightarrow{\sim} \Aut(F'_\alpha/F')$ and  ${\gal(\widetilde{F}_\alpha/F)\xrightarrow{\sim} \gal(\widetilde{F}'_\alpha/F')}$ using the fact that an automorphism of a field extends uniquely to any finite purely inseparable extension. This allows us to identify $\sigma \in \Aut(F_\alpha/F)$ (resp. $\sigma \in \gal(\widetilde{F}_\alpha/F)$) with its extension to $F'_\alpha$ (resp. $\widetilde{F}'_\alpha$), which we will denote by $\sigma$ as well. Then
	$$F(D_\alpha) \otimes_F F_\alpha \cong  \prod_{\sigma \in \Aut(F_\alpha/F)} F(D_\alpha)^\sigma,$$   
	$$F'(D_{\alpha,F'}^{\red}) \otimes_{F'} F'_\alpha \cong \prod_{\sigma \in \Aut(F'_\alpha/F')} F'(D_{\alpha,F'}^{\red})^\sigma $$
	where $F(D_\alpha)^{\sigma}\cong F(D_\alpha)$ and the superscript keeps track of the indexed coordinate (similarly for $F'(D_{\alpha,F'}^{\red})^\sigma$); the groups $\gal(\widetilde{F}_\alpha/F)$ and $\gal(\widetilde{F}'_\alpha/F')$ act on the right by permuting the indices via their actions on $\Aut(F_\alpha/F)$ and $\Aut(F'_\alpha/F')$, respectively. The embedding $F(D_\alpha) \hookrightarrow F'(D_{\alpha,F'}^{\red})$ and the isomorphism ${\gal(\widetilde{F}_\alpha/F)\xrightarrow{\sim} \gal(\widetilde{F}'_\alpha/F')}$, induce a compatible morphism of $\gal(\widetilde{F}_\alpha/F)$-modules
	$$\prod_{\sigma \in \Aut(F_\alpha/F)} F(D_\alpha)^\sigma\hookrightarrow \prod_{\sigma \in \Aut(F'_\alpha/F')} F'(D_{\alpha,F'}^{\red})^\sigma .$$
	This shows that we have a decomposition to geometrically irreducible components
	$$D_\alpha \times_F F_\alpha=\bigcup_{\sigma \in \Aut(F_\alpha/F)} D_\alpha^{(\sigma)} \quad \text{ and } \quad  D_{\alpha,F'}^{\red} \times_{F'} F'_\alpha=\bigcup_{\sigma \in \Aut(F'_\alpha/F')} {D_{\alpha,F'}^{\red}}^{(\sigma )}$$
	where each $D_\alpha^{(\sigma)}$ (resp. ${D_{\alpha,F'}^{\red}}^{(\sigma)}$) corresponds to $F(D_\alpha)^{\sigma}$ (resp. $F'({D_{\alpha,F'}^{\red})}^{(\sigma)}$), with a corresponding $\gal(\widetilde{F}_\alpha/F)$ action.
	
	We have a commutative diagram
	\begin{equation} \label{eqn:PicXSep}
	\begin{tikzcd}[row sep=1.7em, column sep=1.7em]
		0 & { \underset{\alpha} \bigoplus \underset{\sigma\in \Aut(F'_\alpha/F')}{\bigoplus} \mathbb{Z} [{D_{\alpha,F'}^{\red}}^{(\sigma)} \times_{F'_\alpha} F'_s]} & {\pic(X_{F'_s})} & {\pic(Y_{F'_s})} & 0 \\
		0 & {\underset{\alpha} \bigoplus \underset{\sigma\in \Aut(F_\alpha/F)}{\bigoplus} \mathbb{Z} [D_\alpha^{(\sigma)}\times_{F_\alpha} F_s]} & {\pic(X_{F_s})} & {\pic(Y_{F_s})} & 0 
		\arrow[from=2-3, to=2-4]
		\arrow[from=2-4, to=2-5]
		\arrow[from=2-2, to=2-3]
		\arrow[from=1-3, to=1-4]
		\arrow[from=1-4, to=1-5]
		\arrow[from=1-2, to=1-3]
		\arrow["\rotatebox{90}{$\sim$}", from=2-3, to=1-3]
		\arrow[from=2-4, to=1-4]
		\arrow["{D_\alpha^{(\sigma)}\mapsto p_\alpha {D_{\alpha,F'}^{\red}}^{(\sigma)}}"', from=2-2, to=1-2]
		\arrow[from=1-1, to=1-2]
		\arrow[from=2-1, to=2-2]
	\end{tikzcd} \end{equation}
	where the left exactness in both rows follows from \Cref{lemma:Rosenlicht} and the middle vertical isomorphism follow from \Cref{lem:picSepIsom}. As $Y_{F'_s}$ is the trivial $G_{F'_s}$-torsor, we deduce that $Y_{F'_s}\cong G_{F'_s} \cong \mathbb{A}^n_{F'_s}$ which implies that $\pic (Y_{F'_s}) =0$. Therefore, the canonical isomorphism $\Gamma_F\cong \Gamma_{F'}$ induces the following isomorphisms of $\Gamma_F$-modules
	$$\pic(X_{F'_s})\cong \underset{\alpha} \bigoplus \underset{\sigma\in \Aut(F'_\alpha/F')}{\bigoplus} \mathbb{Z} [{D_{\alpha,F'}^{\red}}^{(\sigma)} \times_{F'_\alpha} F'_s] \cong \underset{\alpha} \bigoplus \ind_{{F'_\alpha}}^{F'} \z \cong \underset{\alpha} \bigoplus \ind_{{F_\alpha}}^{F} \z$$
	Now, by the isomorphism $\pic(X_{F_s})\xrightarrow{\sim} \pic(X_{F'_s})$ we see that $\pic(X_{F_s})$ is freely generated by the line bundles $\{\frac{1}{p_\alpha} [D_\alpha^{(\sigma)}]\}_{\alpha,\sigma}$; moreover, it is easy for see that for $g\in \Gamma_F$ we have $g\cdot (\frac{1}{p_\alpha} [D_\alpha^{(\sigma)}])=\frac{1}{p_\alpha} [D_\alpha^{(g\cdot\sigma)}]$. We have shown in the first paragraph that the left vertical map in \eqref{eqn:PicXSep} is that of $\Gamma_F$-modules.
	Hence, by the commutativity of diagram \eqref{eqn:PicXSep} and the fact that  $[D_\alpha^{(\sigma)}]$ maps to $p_\alpha [{D_{\alpha,F'}^{\red}}^{(\sigma)}]$, we deduce that 
	\begin{equation}\label{eqn:picX}
		\pic(X_{F_s})\cong \underset{\alpha} \bigoplus \underset{\sigma\in \Aut(F_\alpha/F)}{\bigoplus} \mathbb{Z} \left( \frac{1}{p_\alpha}[{D_{\alpha}}^{(\sigma)} \times_{F_\alpha} F_s]\right) \cong  \bigoplus_{\alpha \in \mathcal{A}} \ind_{{F_\alpha}}^{F} \mathbb{Z}
	\end{equation}
	as $\Gamma_F$-modules. By applying the snake lemma to \eqref{eqn:PicXSep} we get
	\begin{equation}\label{eqn:picY}
		\pic(Y_{F_s})\cong   \underset{\alpha} \bigoplus \underset{\sigma\in \Aut(F_\alpha/F)}{\bigoplus} \frac{\mathbb{Z}}{p_\alpha \mathbb{Z}} \left( \frac{1}{p_\alpha}[{D_{\alpha}}^{(\sigma)} \times_{F_\alpha} F_s]\right) \cong \bigoplus_{\beta \in \mathcal{B}} \ind_{{F_\beta}}^{F}  \frac{\mathbb{Z}}{p_\beta \mathbb{Z}}
	\end{equation}
	as $\Gamma_F$-modules since $p_\alpha=1$ for all ${\alpha\in \mathcal{A}\setminus \mathcal{B}}$.
\end{proof}

Denote by $\Delta_{\eff}(X)$ the monoid of effective divisors of $X$. 
\begin{prop}\label{prop:picX}	
	Suppose that $Y= G$. Then the following statements hold.
	\begin{enumerate}
		\item The pullback map $\pic(X)\rightarrow \pic(X_{F'})$ is an isomorphism.
		\item There exists effective divisors $\{D_\alpha^{1/p_\alpha}\}_{\alpha\in \mathcal{A}}$ on $X$ satisfying $[D_\alpha]=p_\alpha [D_\alpha^{1/p_\alpha}]$ for all $\alpha\in \mathcal{A}$, inducing isomorphisms
		$$\bigoplus_{\alpha\in \mathcal{A}}\z [D_\alpha^{1/p_\alpha}]  \xrightarrow{\sim} \pic(X) \quad \text{ and } \quad \bigoplus_{\alpha\in \mathcal{A}}\mathbb{N} [D_\alpha^{1/p_\alpha}]  \xrightarrow{\sim} \Delta_{\eff}(X).$$
		
		\item We have an isomorphism
		$$\pic(G)\cong \bigoplus_{\beta \in \mathcal{B}} \z/p_\beta \z. $$
		In particular, $|\pic(G)|=\prod_{\beta\in \mathcal{B}} p_\beta$.
		
		\item $\Br(X)\cong \Br(F)$.
	\end{enumerate}
\end{prop}

	\begin{proof}	
	(1) Note that $F_s[X]^*=F_s[G]^*=F_s^*$ and $F'_s[X ']^*={F'_s}^*$. Hence, as $X(F)$, $X_{F'}(F'),$ and $G(F)$ are non-empty we have ${\pic(X_{F_s})^{\Gamma_F}\cong \pic(X),}$  ${\pic(X_{F'_s})^{\Gamma_{F'}}\cong \pic(X_{F'}) }$ and $\pic(G_{F_s})^{\Gamma_F}\cong \pic(G)$ (see \cite[Remark 5.4.3]{CTSK}). Thus, by applying $(\cdot)^{\Gamma_F}$ to  diagram \eqref{eqn:PicXSep}, and using the fact that both free modules on the left in \eqref{eqn:PicXSep} are $\Gamma_F$-permutation modules and that $\pic(X_{F_s})\xrightarrow{\sim} \pic(X_{F'_s})$ is an isomorphism of $\Gamma_F$-modules, we obtain the commutative diagram with exact rows:
	% https://q.uiver.app/#q=WzAsMTEsWzIsMSwiXFx0ZXh0e0NsfShYKSJdLFszLDEsIlxcdGV4dHtDbH0oRykiXSxbNCwxLCIwIl0sWzEsMSwiXFxiaWdvcGx1c1xcbWF0aGJme1p9IFtEX1xcYWxwaGFdIl0sWzEsMCwiXFxiaWdvcGx1c1xcbWF0aGJme1p9IFtEJ197MCxcXGFscGhhfV0iXSxbMiwwLCJcXHRleHR7Q2x9KFgnKSJdLFszLDAsIlxcdGV4dHtDbH0oXFxtYXRoYmZ7R31fe2EsRid9KT0wIl0sWzQsMCwiMCJdLFsyLDIsIjAiXSxbMCwwLCIwIl0sWzEsMiwiMCJdLFswLDFdLFsxLDJdLFszLDBdLFs1LDZdLFs2LDddLFs0LDVdLFswLDVdLFs4LDAsIlxcbWF0aGNhbHtPfV9YKFgpPUYiLDJdLFsxLDZdLFszLDQsIkRfXFxhbHBoYVxcbWFwc3RvIEQnX1xcYWxwaGEiLDJdLFs5LDRdLFsxMCwzXV0=
	\begin{equation}\label{diag1} \begin{tikzcd}
		0 & {\bigoplus_{\alpha\in \mathcal{A}}\mathbb{Z} [D_{\alpha,F'}^{\red}]} & {\pic(X_{F'})} & {\pic(\mathbb{A}^n_{F'})=0} & 0 \\
		0 & {\bigoplus_{\alpha\in \mathcal{A}}\mathbb{Z} [D_\alpha]} & {\pic(X)} & {\pic(G)} & 0 
		\arrow[from=2-3, to=2-4]
		\arrow[from=2-4, to=2-5]
		\arrow[from=2-2, to=2-3]
		\arrow[from=1-3, to=1-4]
		\arrow[from=1-4, to=1-5]
		\arrow[from=1-2, to=1-3]
		\arrow["\rotatebox{90}{$\sim$}", from=2-3, to=1-3]
		\arrow[from=2-4, to=1-4]
		\arrow["{D_\alpha\mapsto p_\alpha D_{\alpha,F'}^{\red}}"', from=2-2, to=1-2]
		\arrow[from=1-1, to=1-2]
		\arrow[from=2-1, to=2-2]
	\end{tikzcd}\end{equation}
	
	$(2)$ Denote the inverse image of $[D_{\alpha,F'}^{\red}]$ via the isomorphism ${\pic(X)\xrightarrow{\sim} \pic(X_{F'})}$ in \eqref{diag1} by $[D_\alpha^{1/p_\alpha}]$, so that $p_\alpha [D_\alpha^{1/p_\alpha}] = [D_\alpha]$. Thus, the classes of $\{D_\alpha^{1/p_\alpha}\}_{\alpha\in \mathcal{A}}$ is a $\z$-basis for $\pic(X)$.
	
	$(3)$ By applying the snake lemma to \eqref{diag1}, we deduce that $\pic(G)\cong \bigoplus_{\beta \in \mathcal{B}} \z/p_\beta \z $. 
	
	$(4)$ Since $\Ho^1(\Gamma_{F}, \pic(X_{F_s}))=0$ and $F_s[X]^*=F_s^*$, by Proposition 5.4.2 and Remark 5.4.3 in \cite{CTSK}, we deduce that $\Br_1(X)=\Br(F)$. As $\Ho^1(X,\mathcal{O}_X)= 0$, by \cite[Theorem 5.2.5]{CTSK}, the natural map $\Br(X_{F_s})\rightarrow \Br({X}_{\bar{F}})$ is injective. As $X$ is geometrically rational, $\Br(X_{F_s})=\Br(X_{\bar{F}})=0$. Therefore, $\Br(X)=\Br_1(X)=\Br(F)$. 
\end{proof}

\begin{cor}\label{Cor:PicG-D-g.reduced}
	Let $G$ be a connected commutative unipotent $F$-group that admits a smooth compactification  $X$. Then $\pic(G)$ is finite and is a power of $p$. Moreover,  $X\setminus G$ is geometrically reduced if and only if $\pic(G)$ is trivial.
\end{cor}

\begin{rem}
	If $G$ is a unirational $F$-form of $\ga$, then by \cite[Theorem 2.8]{Achet2} the group $\pic(G)$ is $p$-torsion. Thus if $G$ admits a smooth compactification $X$ then ${\pic(G)=\bigoplus_{i=1}^m \z/p\z}$, where $m$ is the number of irreducible components of $X\setminus G$ that are not geometrically reduced. In particular, this implies that $p_\beta=p$ $,\forall \beta\in \mathcal{B}$.
\end{rem}

Suppose now that $Y=G$ and $X$ is a smooth equivariant compactification of $G$. In the following, we follow the definition given in \cite[Definition 1.6]{Mumford}) for a $G$-linearization of a line bundle.

\begin{prop}\label{prop:G-linz line bundle}
	Every line bundle $\mathcal{L}$ on $X$ admits at most one $G$-linearization. 
	Every line bundle $\mathcal{L}$ in the image of
	$$\bigoplus_{\alpha \in \mathcal{A}} \frac{p'}{p_\alpha} \z  \> [D_\alpha] \rightarrow \pic(X)$$
	admits a $G$-linearization.
	If $\mathcal{L}$ is effective and admits a $G$-linearization, then $\cohom^0(X,\mathcal{L})$ has a unique line of $G$-invariant sections. In particular, let $D=\sum d_\alpha (p'/p_\alpha) D_\alpha$ for integers $d_\alpha\geq 0$. Then the canonical section $s_D$ of $\mathcal{O}_X(D)$ is $G$-invariant.
\end{prop}
\begin{proof}
	We follow a similar proof to \cite[Proposition 1.5]{CLYT}.
	The first statement follows by \cite[Proposition 1.4]{Mumford}, as $G$ has no non-trivial characters by \Cref{lemma:Rosenlicht}.  
	By \Cref{prop:picX}.3, the exponent of $\pic(G)$ divides $p'$. Then, $$\mathcal{O}(D_\alpha)^{p'/p_\alpha}\cong \mathcal{O}(D_\alpha^{1/p_\alpha})^{p'/p_\alpha p_\alpha}= \mathcal{O}(D_\alpha^{1/p_\alpha})^{p'}.$$ Thus, by \cite[Theorem 2.14]{Bri}, $\mathcal{O}(D_\alpha)^{p'/p_\alpha}$ admits a $G$-linearization.
	
	Suppose that $\mathcal{L}$ is effective and admits a $G$-linearization. Consider the induced action of $G$ on the projectivization $\mathbb{P}(\text{H}^0(X,\mathcal{L}))$. As $G$ is solvable, by Borel's fixed point theorem \cite[Theorem 10.6]{Bor1}, there exists a non-zero $\mathsf{s}\in \text{H}^0(X,\mathcal{L})$ such that $F\mathsf{s}$ is fixed by $G$. Since $G$ has no non-trivial characters, $\mathsf{s}$ itself is fixed by $G$, so that $\Div(\mathsf{s})$ is $G$-invariant. Hence, $\Div(\mathsf{s})$ does not intersect $G$, implying that it has to be a sum $\sum d_\alpha D_\alpha$. By \Cref{prop:picX}.2, we deduce that $\mathcal{L}\cong \mathcal{O}(\sum d_\alpha D_\alpha)$. Given any other  $\mathsf{s}_0\in  \text{H}^0(X,\mathcal{L})$ with $F\mathsf{s}_0$ $G$-invariant, the same argument shows that $\mathcal{O}(\Div(\mathsf{s}_0))\cong \mathcal{L}$. By \Cref{lemma:Rosenlicht} we deduce that $\mathsf{s}/\mathsf{s}_0\in F^*$. Therefore, $\cohom^0(X,\mathcal{L})$ has a unique line of $G$-invariant sections. 
\end{proof}

\begin{prop}\label{prop:anticanIsEffective}
	Suppose that $G$ is an $F$-form of $\ga$. There exist positive integers $\rho_\alpha \geq 1$ for all $\alpha \in \mathcal{A}$ such that 
	$$\omega_X^{-1}\simeq \mathcal{O}_X(\sum_{\alpha \in \mathcal{A}} \rho_\alpha D_\alpha)$$ 
	
\end{prop}
\begin{proof}
	By \Cref{cor:ex.non-van.diff} there exists a non-vanishing regular top differential $\omega$ on $G$. Hence $\Div(\omega)$ does not meet $G$, so that we can write $\Div(\omega)=\sum_{\alpha\in \mathcal{A}} \rho_\alpha D_\alpha$ for integers $\rho_\alpha\geq 0$. Thus, $\omega_X^{-1}\simeq \mathcal{O}_X(\sum \rho_\alpha D_\alpha)$. It remains to show that $\rho_\alpha \geq 1$, $\forall \alpha$. 
	
	Let $\pi:X_{F'}\rightarrow X$ be the pull back map. Then
	$$ \pi^* \omega_X=\pi^*(\text{det} \Omega_{X/F}) \cong \text{det}(\pi^* \Omega_{X/F})\cong \text{det}(\Omega_{X_{F'}/F'})\cong \omega_{X_{F'}}.$$
	By  \cite[Lemma 2.4]{CLYT}, we have $\omega_{X_{F'}}^{-1}\simeq \mathcal{O}_{X_{F'}}(\sum \rho'_\alpha D_{\alpha,F'}^{\red})$, for some integers $\rho'_\alpha\geq 2$. By \Cref{prop:picX}.1, the pullback map $\pi^*:\pic(X)\rightarrow \pic(X_{F'})$ is an isomorphism; this implies that $\omega_X^{-1}\simeq \mathcal{O}_{X}(\sum_{\alpha \in \mathcal{A}} \rho'_\alpha D_\alpha^{1/p_\alpha})$. Hence $$\mathcal{O}_{X}(\sum_{\alpha \in \mathcal{A}} \rho'_\alpha D_\alpha^{1/p_\alpha})\simeq \omega_X^{-1}\simeq \mathcal{O}_X(\sum \rho_\alpha D_\alpha) ,$$ implying that $\rho_\alpha=\rho'_\alpha/ p_\alpha$. As $\rho'_\alpha\geq 2$ and $p_\alpha\geq 1$ for all $\alpha$, we deduce that $\rho_\alpha \geq 1$ for all $\alpha$.
\end{proof}

By \Cref{prop:picX}, we have an isomorphism $\bigoplus_{\alpha\in \mathcal{A}}\cx [D_\alpha]  \xrightarrow{\sim} \pic(X)_\cx$ where we identify $\pic(X)$ with $\bigoplus_{\alpha\in \mathcal{A}} \frac{1}{p_\alpha}\z [D_\alpha] \hookrightarrow \bigoplus_{\alpha\in \mathcal{A}}\cx [D_\alpha] $. We write ${\rho=(\rho_\alpha)_\alpha \in \bigoplus_{\alpha} \cx [D_\alpha]}$ for the vector corresponding to $\omega_X^{-1}$ (note that $\rho_\alpha \in \z$ for all $\alpha$ by \Cref{prop:anticanIsEffective}.) We will follow this notation throughout the paper.

\section{The Hasse principle and weak approximation}\label{sec:ShTrivial}
Let $F$ be a global function field and $G$ be a connected commutative unipotent $F$-group. We say that $G$ satisfies the Hasse principle for algebraic groups if the Tate-Shafarevich group of $G$, $$\Sh(G):= \text{ker}\left( \Ho^1(F,G)\rightarrow \coprod_v \Ho^1(F_v,G) \right),$$  is trivial. Recall that elements of $\Sh(G)$ correspond to isomorphism classes of everywhere locally soluble $G$-torsors over $F$. We show that if $G$ admits a smooth equivariant compactification then any everywhere locally soluble $G$-torsor over $F$ satsifies weak approximation. In particular,  this will prove \Cref{thm:ShaTrivial}.

\subsection{Preliminaries}
For a variety $V$ over a global field $F$ we define
\begin{equation}\label{eqn:Be(Y)Defn}
	\Be_\omega(V):= \bigcup_{S\subset \Omega_F \text{ finite}} \ker\left( \frac{\Br(V)}{\im \Br(F)} \rightarrow \prod_{v\in \Omega_F \setminus S} \frac{\Br(Y_{F_v})}{\im \Br(F_v)}  \right)
\end{equation}

\begin{lem}\label{lem:BrY-alg}
	Let $V$ be a variety over a global field $F$.  Then $$\Be_\omega(V)\subset \Br_1(V)/\im \Br(F).$$ Moreover, if $F_s[V]^*=F_s^*$ then we have an embedding ${\Be_\omega(V) \hookrightarrow \Ho^1(F, \pic V_{F_s})}$.
\end{lem}
\begin{proof}
	Let $A \in \Br(V)$ be such that $A_{F_v}\in \im \Br(F_v)\subset \Br(V_{F_v})$ for some place $v$. There exists a finite separable extension $K$ of $F_v$, which is also a local field, such that $A_K=0\in \Br(V_K)$. Since $\Br(V_K)$ is the inductive limit of the groups $\Br(V_R)$ over all the finitely generated $F$-algebras $R$ contained in $K$, there exists a finitely generated $F$-algebra $R_0$ such that $A_{R_0}=0\in \Br(V_{R_0})$. Since $K/F$ is separable, the algebra $R_0$ is an $F$-smooth domain. Thus $\spec(R_0)\times F_s$ has an $F_s$-point $P$ such that ${A_{R_0\otimes_F F_s}(P)=0}$.  By functoriality, this shows that $A_{F_s}=0\in \Br(V_{F_s})$, i.e. ${A \in \Br_1(V)}$. Now, if $F_s[V]^*=F_s^*$, we deduce by \cite[Prop 5.4.2]{CTSK} that $\Br_1(V)/\im \Br(F)\hookrightarrow \Ho^1(F, \pic V_{F_s})$.
\end{proof}

\begin{lem}\label{lem:Y-comp}
		Let $k$ be any field and $G$ be a commutative affine algebraic $k$-group. Let $Y$ be a $G$-torsor over $k$. If $G$ admits a smooth equivariant compactification $X$, then $Y$ admits a smooth compactification.
\end{lem}
\begin{proof}
	Let $\varphi:\Gamma_k\rightarrow G(k_s)$ be the $1$-cocycle with class in $\Ho^1(k,G)$ corresponding to the isomorphism class of the $G$-torsor $Y$. Since $G$ acts on $X$, we have a map $G(k_s)\rightarrow \Aut X_{k_s} $ which induces a map $\Ho^1(k,G)\rightarrow \Ho^1(k,\Aut X_{k_s})$. We have an isomorphism $G(k_s)\xrightarrow{\sim} \Aut \mathbf{G}_{k_s}$, where $\mathbf{G}$ is the trivial $G$-torsor over $k$, which induces an isomorphism  $\Ho^1(k,G)\xrightarrow{\sim} \Ho^1(k,\Aut \mathbf{G}_{k_s})$. We abuse notation and denote the image of $\varphi$ in $\Ho^1(k,\Aut \mathbf{G}_{k_s})$ and $\Ho^1(k,\Aut X_{k_s})$ by $\varphi$. This gives an effective descent datum on $\mathbf{G}_{k_s}$ and on $X_{k_s}$ so that $\mathbf{G}_{k_s}$ descends to $Y$ and $X_{k_s}$ descends to a smooth proper $k$-variety $X^\varphi$. Now, since $X$ is an equivariant smooth compactification of $G$, the image of $G$ in $X$ is an open dense orbit; thus, the descent datum given by $\varphi$ is compatible so that the image of $G_{k_s}$ in $X_{k_s}$ is an open subscheme that is stable under this descent datum. Therefore, this gives an open immersion $Y\hookrightarrow X^\varphi$.
\end{proof}

Let $L/F$ be a finite separable extension of global fields. For $v\in\Omega_F$, we can write $F_v\otimes_F L\cong \prod_{w \mid v} L_w$, where $L_w$ is the local field of $L$ at $w$ lying above $v$.

\begin{lem}\label{lem:ShInd=0}
	Let $F$ be a global field and $L/F$ be a finite separable extension. Let $M$ be the induced $\Gamma_F$-module $\ind_F^L \z/m\z$, where $m\in \z_{>0}$ is a power of some prime. 
	\begin{enumerate}
		\item For $v\in \Omega_F$, the map $\Gamma_{F_v}\rightarrow \Gamma_F$ gives $M$ the structure of a $\Gamma_{F_v}$-module that is isomorphic to 
		$ \prod_{w \mid v} \ind_{{L_w}}^{F_v} \mathbb{Z}/m \mathbb{Z} $.
		\item For any finite set of places $S$ we have
		$$\Sh_S(M):= \text{ker}\left( \Ho^1(F,M)\rightarrow \prod_{v\notin S} \Ho^1(F_v,M) \right),$$  is trivial.
	\end{enumerate}
\end{lem}
\begin{proof}
	(1) follows from the $F_v$-isomorphism $F_v\otimes_F L\cong \prod_{w \mid v} L_w$.
	
	(2) By Shapiro's lemma and (1), we obtain the commutative diagram
	% https://q.uiver.app/#q=WzAsNixbMCwwLCJIXjEoRiwgXFxpbmRfe3tGX1xcYmV0YX19XntGfSBcXG1hdGhiYntafS9wX1xcYmV0YSBcXG1hdGhiYntafSkiXSxbMSwwLCIgXFxwcm9kX3YgXFxwcm9kX3tcXGJldGFfdlxcbWlkIFxcYmV0YSB9IEheMShGX3YsXFxpbmRfe3tGX3tcXGJldGFfdn19fV57Rl92fSBcXG1hdGhiYntafS9wX1xcYmV0YSBcXG1hdGhiYntafSkiXSxbMCwxLCJIXjEoRl9cXGJldGEsIFxcbWF0aGJie1p9L3BfXFxiZXRhIFxcbWF0aGJie1p9KSJdLFsxLDEsIiBcXHByb2RfdiBcXHByb2Rfe1xcYmV0YV92XFxtaWQgXFxiZXRhIH0gSF4xKEZfe1xcYmV0YV92fSwgXFxtYXRoYmJ7Wn0vcF9cXGJldGEgXFxtYXRoYmJ7Wn0pIl0sWzAsMiwiSF4xKEZfXFxiZXRhLCBcXG1hdGhiYntafS9wX1xcYmV0YSBcXG1hdGhiYntafSkiXSxbMSwyLCIgXFxwcm9kX3t3XFxpblxcT21lZ2Ffe0ZfXFxiZXRhfX0gSF4xKEZfe1xcYmV0YSx3fSwgXFxtYXRoYmJ7Wn0vcF9cXGJldGEgXFxtYXRoYmJ7Wn0pIl0sWzAsMl0sWzAsMV0sWzEsM10sWzIsM10sWzIsNCwiXCJcXHJvdGF0ZWJveHs5MH17PX1cIiJdLFszLDUsIlwiXFxyb3RhdGVib3h7OTB9ez19XCIiXSxbNCw1XV0=
	\[\begin{tikzcd}
		{\Ho^1(F, \ind_{{L}}^{F} \mathbb{Z}/m \mathbb{Z})} & { \prod_v \prod_{w\mid v } \Ho^1(F_v,\ind_{L_w}^{F_v} \mathbb{Z}/m \mathbb{Z})} \\
		{\Ho^1(L, \mathbb{Z}/m \mathbb{Z})} & { \prod_v \prod_{w \mid v } \Ho^1(L_w, \mathbb{Z}/m \mathbb{Z})} \\
		{\Ho^1(L, \mathbb{Z}/m \mathbb{Z})} & { \prod_{w} \Ho^1(L_w, \mathbb{Z}/m \mathbb{Z}).}
		\arrow[from=1-1, to=1-2]
		\arrow["\rotatebox{90}{$\sim$}",from=1-1, to=2-1]
		\arrow["\rotatebox{90}{$\sim$}", from=1-2, to=2-2]
		\arrow[from=2-1, to=2-2]
		\arrow["{\rotatebox{90}{=}}", from=2-1, to=3-1]
		\arrow["{\rotatebox{90}{=}}", from=2-2, to=3-2]
		\arrow[ from=3-1, to=3-2]
	\end{tikzcd}\]
	An element $\varphi$ in $\Ho^1(L, \mathbb{Z}/m \mathbb{Z})=\text{Hom}(\Gamma_{L}, \z/m \z)$ corresponds to a finite separable extension $K$ of $L$ such that $[K:L]\mid m$ and $\Aut(K/L) \cong \Gamma_{L}/ \ker(\varphi)$. If the image of $\varphi$ in $ \Ho^1(L_w, \mathbb{Z}/m \mathbb{Z})$ is trivial for all places $w$ not lying above $S$, then almost all places $w$ of $L$ are completely split in $K$. By Chebotarev's density theorem, we conclude that $K=L$ so that $\varphi$ is trivial as desired.
\end{proof}

\subsection{Proof of \Cref{thm:ShaTrivial}}

	Let $Y$ be a $G$-torsor over $F$ that is everywhere locally soluble. By \cite[Theorem 4.5]{Azur}, the Brauer-Manin obstruction given by $\Be_\omega(Y)$
	is the only obstruction to weak approximation on $Y$. Therefore, to prove that $Y$ satisfies weak approximation, it suffices to show that $\Be_\omega(Y)$ is trivial which we prove in the following proposition.

\begin{prop}
	Let $F$ be global function field. Let $G$ be a connected commutative unipotent $F$-group 
    that admits a smooth equivariant compactification. Let $Y$ be a $G$-torsor over $F$. Then $\Be_\omega(Y)$ is trivial.
\end{prop}
\begin{proof}
	By \Cref{lem:BrY-alg} it suffices to prove that the kernel of the map
	$$\Ho^1(F, \pic(Y_{F_s}))\rightarrow \prod_{v\in \Omega_F\setminus S} \Ho^1(F_v,\pic(Y_{F_v^s}))$$
	is trivial for all finite set of places $S$ in $\Omega_F$. By \Cref{lem:Y-comp}, there exists a smooth compactification of $Y$. Denote the boundary divisor by $D$. Following the notation in \S \ref{section:PicX}, let $\{ D_\alpha \}_{\alpha\in \mathcal{A}}$ be the set of irreducible components of $D$ over $F$ and let ${\mathcal{B}:=\{\beta \in \mathcal{A}\>:\> D_\beta \text{ is not geometrically reduced}\}}$. Let $F_\beta$ be the separable closure of $F$ in $F(D_\beta)$. By \Cref{prop:picGalModule}, we have an isomorphism of $\Gamma_F$-modules
	$$\pic(Y_{F_s})\cong \bigoplus_{\beta \in \mathcal{B}} \ind_{{F_\beta}}^{F} \mathbb{Z}/p_\beta \mathbb{Z},$$
	where $p_\beta$ is the degree of the inseparable closure of $F$ in $F(D_\beta)$ over $F$.
	
	For all places $v\in \Omega_F$, fix an embedding $F_s\subset F_v^s$. Write $D\times_F F_v=\cup_{\alpha_v \in \mathcal{A}_v} D_{\alpha_v}$, where each $D_{\alpha_v}$ is irreducible over $F_v$. Let $\mathcal{B}_v\subset \mathcal{A}_v$ be the subset indexing the non-geometrically reduced components. Let $F_{\beta_v}$ be the separable closure of $F_v$ in $F_v(D_{\beta_v})$ for $\beta_v\in \mathcal{B}_v$. If $D_{\beta_v}$ lies above $D_\beta$, then the degree of the inseparable closure of $F_v$ in $F_v(D_{\beta_v})$ over $F_v$ is also $p_\beta$ since $F_v/F$ is separable. Since $Y_{F_v}$ admits a smooth compactification and $F_v$ has degree of imperfection $1$, we may apply \Cref{prop:picGalModule} to obtain an isomorphism of $\Gamma_{F_v}$-modules
	$$\pic(Y_{F_v^s})\cong \bigoplus_{\beta \in \mathcal{B}} \bigoplus_{\beta_v \mid \beta} \ind_{{F_{\beta_v}}}^{F_v} \mathbb{Z}/p_\beta \mathbb{Z}.$$
	By \eqref{eqn:picX} (resp. \eqref{eqn:picY}) we see that $\pic(Y_{F_s})$ (resp. $\pic(Y_{F_v^s})$ ) is freely generated by the classes of the irreducible components of $D_{F_s}$ (resp. $D_{F_v^s}$) that are not geometrically reduced. Since the base change of an irreducible $F_s$-variety to any field containing $F_s$ remains irreducible, we deduce that the pullback map $\pic(Y_{F_s})\rightarrow \pic(Y_{F_v^s})$ is an isomorphism. Moreover, we can write  $\pic(Y_{F_s})\xrightarrow{\sim} \pic(Y_{F_v^s})$ as a direct sum of the maps $\ind_{{F_\beta}}^{F} \mathbb{Z}/p_\beta \mathbb{Z}\rightarrow \bigoplus_{\beta_v \mid \beta} \ind_{{F_{\beta_v}}}^{F_v} \mathbb{Z}/p_\beta \mathbb{Z}$ over $\beta\in \mathcal{B}$. Since $$F_\beta \otimes_F F_v \simeq \prod_{\{\beta_v : \> \beta_v \mid \beta\}} F_{{\beta_v}}=\prod_{\{w\in\Omega_{F_\beta} \>:\> w \mid v\}} F_{{\beta,w}},$$
	the result now follows by \Cref{lem:ShInd=0}.
\end{proof}

\subsection{An example with no smooth compactification}
We end this section by showing that $\Sh(G)$ is trivial for the following class of $F$-forms of $\ga$ without the assumption that a smooth equivariant compactifications exists.
\begin{prop}\label{prop:Sha-trivial-example}
	Assume that $F=\mathbb{F}_q(t)$. Let $G$ be an $F$-form of $\ga$ defined, as a subgroup of $\mathbb{G}_a^{n+1}$, by a separable $p$-polynomial
	$$P=\sum_{i,j}^{n,r} c_{i,j} x_i^{p^j}$$ 
	such that each $c_{i,j}\in \mathfrak{o}_F =\mathbb{F}_q[t]$. If $\deg (c_{i,j})< p^j $ for all $i,j$, then $\Sh(G)=0.$
\end{prop}
\begin{proof}
	We follow a similar proof to the one given  in \cite[Lemma 3.1]{ZRos1}. For a field $K$, we write $K^{n}$ for $\mathbb{G}_a^{n}(K)$. From the exact sequence
	$$0\rightarrow G\rightarrow \mathbb{G}_a^{n+1} \xrightarrow{P} \mathbb{G}_a \rightarrow 0,$$
	we deduce that $\Ho^1(F,G)\simeq F/P(F^{n+1})$ and $\Ho^1(F_v,G)\simeq F_v/P(F_v^{n+1})$ for all $v$. Therefore $$\Sh(G)\simeq \frac{\{ a\in F \mid a\in P(F_v^{n+1}) \text{ for all } v\}}{P(F^{n+1})}.$$  
	Assume $a\in F$ satisfies $a\in P(F_v^{n+1})$ for all $v$, so that we have $a=P(x_{0,v}, \ldots , x_{n,v})$ for some $x_{0,v}, \ldots , x_{n,v}\in F_v$, for all $v$. By strong approximation on $\ga$, there exists $x_0,x_1,\ldots, x_n \in F$ such that $x_0-x_{0,v}, \ldots, x_n-x_{n,v}\in \ov$ for all $v\neq \infty$ with $a\notin \ov$, and $x_0\in \ov$ for all other places $v\neq \infty$. Therefore, for all $v\neq \infty$, we have
	$$P(x_0,\ldots, x_n)-a=P(x_0-x_{0,v}, \ldots , x_n-x_{n,v})\in \ov,$$ since each $c_{i,j}\in \mathbb{F}_q[t]$ in $P$. Since $a$ and $a-P(x_0,\ldots,x_n)$ give the same class in $\Sh(G)$, we may assume that $a\in \ov$ for all $v\neq \infty$. We want to show that $a\in P(F^{n+1})$.
	
	Write $a=P(x_{0,\infty},\ldots, x_{n,\infty})$ for some $x_{0,\infty},\ldots, x_{n,\infty}\in F_{\infty}$. Note that we can write $x_{i,\infty}=x_i+y_{i,\infty}$ for some $x_i\in \mathbb{F}_q[t]$ and $y_{i,\infty} \in \mathfrak{m}_\infty$ for all $i$. Then, as $$P(x_0+y_{0,\infty}, \ldots, x_n+y_{n,\infty})=P(x_0,\ldots,x_n) + P(y_{0,\infty,\ldots,y_{n,\infty}})$$ it suffices to show that $P(y_{0,\infty},\ldots,y_{n,\infty})=a-P(x_0,\ldots,x_n)=0$. By the condition $c_{i,j}\in \mathbb{F}_q[t]$, we deduce that $a-P(x_0,\ldots,x_n)\in \ov$ for all $v\neq \infty$. And by the condition $\deg (c_{i,j})< p^j $ for all $i,j$, we see that $a-P(x_0,\ldots,x_n)= P(y_{0,\infty},\ldots,y_{n,\infty}) \in \mathfrak{m}_\infty$. This shows that ${a-P(x_0,\ldots,x_n)}=0$ as desired.
\end{proof}

\begin{exmp}
	Let $F=\mathbb{F}_p(t)$ where $p>2$. Let $G$ be the subgroup of $\mathbb{G}_a^2$ over $F$ given by the equation $y^p-tx^p=y$. Then 
	\begin{enumerate}[label=(\roman*)]
		\item $G(F)$ is finite (\cite[VI.3.1]{Ost});
		\item $\Ho^1(F,G)$ is infinite (\cite[Theorem 2.1]{RosNN});
		\item $\Sh(G)$ is trivial (\Cref{prop:Sha-trivial-example});
		\item and $\prod_v G(F_v)$ is infinite (\cite[Remark 5.7]{Azur}) so that $G$ fails weak approximation. In particular, $\Be_\omega(G)$ is not trivial by \cite[Theorem 4.5]{Azur}.
	\end{enumerate}
	Note that any regular compactification of $G$ is not smooth by \Cref{prop:smoothCompImpliesInfPts} as $G(F)$ is finite. In fact, it is not difficult to see that the obvious compactification of $G$ in $\mathbb{P}^2$ is regular (and hence unique) but not smooth. Thus, this example does not satisfy the conditions of \Cref{thm:ShaTrivial}.
\end{exmp}

\section{Metrizations, heights, and measures}\label{sec:Metr-Heights-Measures}

\subsection{Global function fields and measures}
Let $F$ be a global function field of characteristic $p$ with constant field $\mathbb{F}_q$ and let $\Omega_F$ be its set of places. Recall that all places $v$ of $F$ are non-archimedean and they correspond to the closed points of the curve $\cf$. Let $F_v$ be the completion of $F$ at $v$, $\mathfrak{o}_v$ be its ring of integers, $\mathfrak{m}_v$ its maximal ideal, $\fv$ its residue field with $q_v$ elements. We denote the ring of ad\`eles of $F$ by $\ad$. Denote $\mathbb{F}_q(t)$ by $F_0$. Then $F/F_0$ is a non-constant finite separable extension.

For all $v\in \Omega_F$, the field $F_v$ is a locally compact group which we equip with a Haar measure $\mu_v$ defined so that $\mu_v(\ov)=1$ for all $v\in \Omega_F$. This measure satisfies the property $\mu_v(aI)=|a|_v\mu_v(I)$ for any measurable subset $I\subset F_v$ and any $a\in F_v$, where $|a|_v=q_v^{-v(a)}$. We define the Haar measure $\mu:=\prod_v \mu_v$ on $\ad$, which satisfies
\begin{equation}\label{eqn:vol(A/F)}
	\mu(\ad/F)=q^{g_{F}-1}
\end{equation}
where $g_{F}$ is the genus of $\cf$ (see \cite[2.1.3]{Weil2}). This shows that $\mu$ is not self-dual when the genus of $\cf$ is different from $1$.

\subsection{Adelic metrics on line bundles} 
Let $X$ be a smooth equivariant compactification of an $F$-form $G$ over a global function field $F$.

\begin{defn}\label{defn:adelic-met} 
	A \emph{smooth adelic metric} on a linearized line bundle $\mathcal{L}$ on $X$ is a family of $v$-adic metrics $\vnorm_v$ on a $\mathcal{L}$ for all places $v$ of $F$ satisfying the following properties:
	\begin{enumerate}
		\item $\vnorm_v$ is locally constant.
		\item there exists open dense subset $\cfs\subset \cf$, given by taking away a finite set of places $S\subset \Omega_F$, a flat projective $\cfs$-scheme $\mathcal{X}$ extending $X$ together with an action of $\mathcal{G}_{\cfs}$ extending the action of $G$ on $X$ and a linearized line bundle $\mathfrak{L}$ on $\mathcal{X}$ extending the linearized line bundle $\mathcal{L}$ on $X$,  such that for any place $v$ lying over $\cfs$, the $v$-adic metric on $\mathcal{L}$ is given by the integral model. 
	\end{enumerate}
\end{defn}
For the definition of a $v$-adic metric on a line bundle see e.g. \cite[2.1.3]{CLYTig}; for the definition of a $v$-adic metric given by the integral model see e.g. \cite[2.14]{BGM20}.

We say that $\lVert \cdot \rVert_{v}$ is a \emph{$q$-metric} if for all  $x\in X(F_v)$ and $\ell \in \mathcal{L}(x)$, the value $\lVert  \ell \rVert_v $ is a power of $q$.
We will work with $q$-metrics throughout the paper. In what follows, whenever we work with an $\ov$-model $\mathcal{G}$ of $G$, we abuse notation and set $G(\ov)=\mathcal{G}(\ov)$.

\begin{prop}\label{prop:K-stabilizer}
	Let $\mathcal{L}$ be a linearized line bundle on $X$ with a smooth adelic metric.
	\begin{enumerate}[label=(\roman*)]
		\item For all places $v$, the stabilizer of  $(\mathcal{L},\> \vnorm_v)$, i.e., the set of $g\in G(F_v)$ which act isometrically on $(\mathcal{L},\> \vnorm_v)$, is a compact open subgroup of $G(F_v)$. If $v\notin S$, then $\mathcal{G}(\ov)$ stabilizes $(\mathcal{L},\> \vnorm_v)$. 
		\item The product of the stabilizers for $v\in S$ with $\prod_{v\notin S} \mathcal{G}(\ov)$ is a compact open subgroup $\mathbb{K}\subset G(\ad)$ that stabilizes the adelic metric.
	\end{enumerate}
\end{prop}
\begin{proof}
	The proof is analogous to Lemma $3.2$  and Proposition $3.3$ in \cite{CLYT}.
\end{proof}

\subsection{Heights}
Let $\{D_\alpha\}_{\alpha\in \mathcal{A}}$ be the set of irreducible components of $D:=X\setminus G$. We choose and fix smooth adelic $q$-metrics on the line bundles $\mathcal{O}(D_\alpha)$ for all $\alpha\in \mathcal{A}$. Note that we can equip the tensor product of two metrics by the product of the corresponding metrics. In this way, the adelic metrics on $\mathcal{O}(D_\alpha)$ induce adelic metrics on each line bundle in the image of $$\bigoplus_{\alpha \in \mathcal{A}} \> \z [D_\alpha] \rightarrow \pic(X).$$
Since $\{\mathcal{O}(D_\alpha) \}_{\alpha \in \mathcal{A}}$ is a $\cx$-basis for $\pic(X)_\cx$ (see \Cref{prop:picX}),
the choice of smooth adelic metrics on $\mathcal{O}(D_\alpha)$ allows us to define compatible systems of heights
\begin{equation}\label{eqn:defnHeight}
	H: G(F)\times \pic(X)_\mathbb{C} \rightarrow \mathbb{C}, \>\>\> ( x ;  (s_\alpha)_{\alpha} ) \mapsto \prod_v \prod_\alpha \lVert \mathsf{s}_\alpha(x) \rVert_v ^{-s_\alpha},
\end{equation}
where $\mathsf{s}_\alpha$ is a fixed canonical section of $\mathcal{O}(D_\alpha)$, and $ (s_\alpha)_{\alpha}$ are the coordinates in $\pic(X)_\cx$ with respect to this chosen basis. By the product formula, this definition is independent of the choice of the section of each  $\mathcal{O}(D_\alpha)$. For a vector $\lambda = (\lambda_\alpha) \in \bigoplus_\alpha \cx [D_\alpha]$, we denote $H(\cdot, \lambda)$ by $H_\lambda(\cdot)$ or $H_{\mathcal{L}_\lambda}(\cdot)$.
One can extend this system of heights in an obvious way to a pairing
$$ H = \prod_v H_v : G(\ad) \times \pic(X)_{\mathbb{C}}\rightarrow \mathbb{C}$$
by the mapping 
$$( \mathbf{x}; \mathbf{s} ) = \left(  (x_v) ;  (s_\alpha)  \right)\rightarrow \prod_v \prod_\alpha \lVert \mathsf{s}_\alpha(x_v) \rVert_v^{-s_\alpha}.$$
This definition does depend on the choice of the section of each $\mathcal{O}(D_\alpha)$.

Let $\mathbb{K}$ be the intersection of the maximal open compact subgroups of $G(\ad)$ fixing $(\mathcal{O}(D_\alpha),\> ||\cdot||_v)$, for $\alpha \in \mathcal{A}$.

\begin{prop}\label{prop:height-K-inv}
	The height pairing $H: G(\ad) \times \pic(X)_{\mathbb{C}}\rightarrow \mathbb{C}$ is $\mathbb{K}$-invariant in the first component and exponentially linear in the second component.
\end{prop}
\begin{proof}
	First, it clearly follows from the definition of $H$ that it is exponentially linear in the second component. 
	Now, by \Cref{prop:G-linz line bundle}, we see that $\mathcal{O}(D_\alpha)^{p'/p_\alpha}$ is effective and admits a $G$-linearization, which we equip with the tensor product metric.  By \Cref{prop:K-stabilizer} and the construction of $\mathbb{K}$, we have that
	$$\lVert \mathsf{s}_\alpha (x_v+y_v) \rVert_v ^{p'/p_\alpha}= \lVert \mathsf{s}_\alpha^{p'/p_\alpha} (x_v+y_v) \rVert_v= \lVert \mathsf{s}_\alpha^{p'/p_\alpha} (x_v) \rVert_v = \lVert \mathsf{s}_\alpha (x_v) \rVert_v ^{p'/p_\alpha},$$
	for all $x=(x_v)\in G(\ad)$, $y=(y_v)\in \mathbb{K}$, $v\in\Omega_v$ and $\alpha \in \mathcal{A}$; this implies that ${\lVert \mathsf{s}_\alpha (x_v+y_v) \rVert_v = \lVert \mathsf{s}_\alpha (x_v) \rVert_v}$.
\end{proof}

\begin{prop}\label{prop:Northcott}
	Suppose $\lambda=(\lambda_\alpha)\in \pic(X)_\cx$ is contained in the interior of the effective real cone.
	Then, for any real $B$, there are only finitely many $x\in G(F)$ such that $H(x;\lambda)\leq B$.
\end{prop}
\begin{proof}
	It suffices to prove the result over the finite extension $F'/F$. The result now follows by \cite[Proposition 4.3]{CLYT}.
\end{proof}

\subsection{The height zeta function and a domain of convergence}

To study the asymptotics for the number of points of bounded height, we introduce the \textit{height zeta function}
$$Z(\mathbf{s})=\sum_{x\in G(F)} H(x;\mathbf{s})^{-1}, \quad \mathbf{s}=(s_\alpha)\in \pic(X)_\cx.$$
For a line bundle $\mathcal{L}$ with coordinates $\lambda:=(\lambda_\alpha)\in \bigoplus_\alpha \cx [D_\alpha]$, the analytic properties of $Z_\lambda(s):= Z(\lambda s)$, for $s\in \cx$, provide information on the asymptotics for the number of points of bounded height with respect to the height $H_\lambda(x):=H(x; \lambda )$.

\begin{prop}\label{Prop:conv-zeta-tube}
	There exists a real number $t>0$ such that $Z(\mathbf{s})$ converges absolutely to a bounded holomorphic function in the tube domain $${ \{(w_\alpha)\in \pic(X)_\mathbb{R}: \> w_\alpha>t, \> \forall \alpha\in \mathcal{A}\} +i\pic(X)_\mathbb{R}}$$ in the complex vector space $\pic(X)_\mathbb{C}$.
\end{prop}
\begin{proof}
	The proof is analogous to \cite[Proposition 4.5]{CLYT}.
\end{proof}

Let $\mathcal{L}$ be a big line bundle on $X$ and let $\lambda=(\lambda_\alpha)\in \bigoplus_\alpha \cx [D_\alpha]$ be the corresponding vector in $\pic(X)_\cx$. By \Cref{prop:picX}, we can write $$\lambda_\alpha=\lambda'_\alpha/p'_\alpha \text{ where } \lambda'_\alpha,p'_\alpha\in \z_{\geq 0} \text{ are coprime} \text{ and } p'_\alpha=p^{k_\alpha} \text{ for some } k_\alpha\in \z_{>0} .$$ 
We define 
$$ \quad g_\lambda:=\frac{ \gcd(\{\lambda'_\alpha \}_{\alpha\in \mathcal{A}})}{\lcm(\{p'_\alpha \}_{\alpha\in \mathcal{A}})}.$$
The following result shows that $Z_\lambda(s)$ has an imaginary period.
\begin{lem}\label{lem:Period}
	Let $\mathcal{L}_\lambda$ be a big line bundle corresponding to $\lambda = (\lambda_\alpha )\in \bigoplus_\alpha \mathbb{Q} [D_\alpha]$. Then $Z_\lambda(s)$ is periodic with respect to the imaginary axis with period dividing
	$$\frac{2\pi i}{g_\lambda \log(q)}.$$
\end{lem} 
\begin{proof}
	Since $Z_\lambda(s)=\sum_{x\in G(F)} H_\lambda(x)^{-s}$, it suffices to prove the result for 
	$$H_\lambda(x)^{-s}= \prod_v \prod_\alpha \lVert \mathsf{s}_\alpha \rVert_v (x)^{\lambda_\alpha s},$$
	for all $x\in G(F)$.
	Note that for all places $v$, as the adelic metric is $q$-valued we deduce that $\lVert \mathsf{s}_\alpha \rVert_v (x)^{\lambda_\alpha s}$ has an imaginary period dividing $\frac{2\pi i}{\lambda_\alpha \log(q)}$. Hence, by the definition of $g_\lambda$, we see that $\prod_\alpha \lVert \mathsf{s}_\alpha \rVert_v (x)^{\lambda_\alpha s}$ has period dividing $\frac{2\pi i}{g_\lambda \log(q)}$ for all places $v$. 
\end{proof}
\begin{rem}\label{rem:PeriodZ}
	Let $\lambda = (\lambda_\alpha )\in \bigoplus_\alpha \mathbb{Z} [D_\alpha]$. Then $g_\lambda$ is always an integer, which implies that $Z_\lambda(s)$ has imaginary period dividing $2\pi i/\log(q)$.
	If $\lambda = (\lambda_\alpha )\notin \bigoplus_\alpha \mathbb{Z} [D_\alpha]$, then $g_\lambda$ is not necessarily an integer. For example, if $\lambda_\beta=1/p$ for all $\beta\in \mathcal{B}$ and $\lambda_\alpha=1$ otherwise, then $g_\lambda=1/p$. Hence
	$$q^{\frac{1}{p} (s+\frac{2\pi i}{\log(q)})}=q^{\frac{s}{p}} e^ {\frac{2\pi i}{p}}\neq q^{\frac{s}{p} }$$
	which implies that $2\pi i/\log(q)$ is not an imaginary period for $Z_\lambda(s)$.	
\end{rem}

\subsection{Tamagawa and Haar measures}\label{TamHaarMeasure}
Let $X_v$ be the base change of $X$ to $F_v$. The $v$-adic $q$-metrics on $\{\mathcal{O}(D_\alpha)\}_{\alpha\in \mathcal{A}}$ induce (via the tensor product) an adelic $q$-metric $(\lVert \cdot \rVert_v)$ on $\omega_X^{-1}$ (which induces a metric on $\omega_X$). Let $\omega\in \omega_X(U)$ be a non-vanishing $n$-form on an open subscheme $U$ of $X$. This defines a measure $|\omega|_v/\lVert \omega\rVert_v$ on $U(F_v)$ that is independent of the choice of the non-vanishing $n$-form $\omega$. These measures defined on some open cover of $X(F_v)$ glue to a measure on $X(F_v)$, denoted by $\tau_{X,v}$, which we call the \textit{local Tamagawa measure} on $X_v$ (see \cite[\S 2.1]{CLYTig}). 

Let $\omega \in \omega_X(G)$ be a non-vanishing translation invariant $n$-form on $G$, which exists by \Cref{cor:ex.non-van.diff}. Let $f_{\omega}$ be the canonical section of $\mathcal{O}(-\Div(\omega))\cong \omega_X^{-1}$, so that the support of $\Div(f_\omega)$ is the boundary divisor $D$. The measure 
$$d\mathbf{g}_v:= \frac{1}{\lVert f_{\omega} \rVert_v} d\tau_{X,v}$$
is well defined on $G(F_v)$ and coincides with the Haar measure $|\omega|_v$ on $G(F_v)$. As $\omega$ is a canonical section of $\omega_X$, and $\omega_X\cong \sum_\alpha -\rho_\alpha [D_\alpha]$, by Rosenlicht's lemma, there exists $c\in F^{\times}$ such that $\omega= c\prod \mathsf{s}_\alpha^{-\rho_\alpha}$. Therefore
\begin{equation}\label{eqn:dg=dtau}
	d\mathbf{g}_v= \lVert \omega \rVert \> d\tau_{X,v} =|c|_v \prod_\alpha \lVert \mathsf{s}_\alpha \rVert^{-\rho_\alpha} \> d\tau_{X,v}
\end{equation}
This defines the Haar measure $$d\mathbf{g}:=\prod_v d\mathbf{g}_v =\prod_v \prod_\alpha \lVert \mathsf{s}_\alpha \rVert^{-\rho_\alpha} \> d\tau_{X,v}$$
on $G(\ad)$, using that $\prod_v |c|_v=1$ by the product formula. This measure is clearly independent on the choice of the non-vanishing $\omega\in \omega_X(G)$. 

The \textit{Tamagawa measure} on $G(\ad)$ is defined as
\begin{equation}\label{Tau_G}
	\tau_G:= q^{-\Dim(G)(g_F-1)} d\mathbf{g}
\end{equation}
where $q$ is the cardinality of the constant field of $F$ and $g_F$ is the genus of $F$ (see  \cite[\S 2.4]{Weil2}). The Tamagawa number of $G$ is defined as
$$\tau(G):= \tau_G(G(\ad)/G(F)).$$
Therefore, we have
\begin{equation}\label{eqn:TamagawaNumberOfG}
	\int_{G(\ad)/G(F)} d\mathbf{g} = q^{\Dim(G)(g_F-1)} \tau(G).
\end{equation}
(c.f. \eqref{eqn:vol(A/F)}).

\section{Fourier analysis and the Poisson formula}\label{sec:PoissonSum}
\subsection{Harmonic analysis on $G(\ad)$}\label{subsec:HarmonicAn}

We start by recalling some facts on the harmonic analysis on the locally compact group $\ga(\ad)$ (see \cite[Theorem IV.3]{Weil1}). Let $F$ be a global function field with constant field $\mathbb{F}_q$. Let $\pi_v$ be a uniformizing parameter for $F_v$. We have an isomorphism $F_v\cong \mathbb{F}_{q_v}((\pi_v))$ so that every element $x_v\in F_v$ has a unique power series expansion of the form $$x_v=\sum_{i=-m}^\infty x_{i,v} \pi_v^i$$ for some $m \in \z_{\geq 0}$, where $x_{i,v}\in \mathbb{F}_{q_v}$ for all $i$. 
Let $F_0:=\mathbb{F}_p(t)$ and $F_{0,\infty}$ be its completion with respect to the place corresponding to $1/t$. We define the character 
\begin{equation}\label{eqn:CharOnF0Infty}
	\psi_{0,\infty}: F_{0.\infty} \rightarrow \cx, \> x_\infty \mapsto \chi(-x_{1,\infty}).
\end{equation}
By \cite[Theorem IV.3]{Weil1}, there is a unique character $\psi_0 \in \widehat{\mathbb{A}_{F_0}}$ orthogonal to $F_0$ such that its restriction to $F_{0,\infty}$ is ${\psi}_{0,\infty}$. Moreover, the character 
\begin{equation}\label{eqn:CharOnF0}
	\psi:=\psi_{0}\circ \Tr_{\ad/\mathbb{A}_{F_0}}\in \widehat{\ad}
\end{equation} is orthogonal to $F$. We can write $\psi$ as $\prod_v \psi_v$, where each $\psi_v$ is a non-trivial local character of $F_v$. 

For  $\mathbf{a}\in \ga(F_v)$, we define the character 
\begin{equation}\label{eqn:localCharacter}
	\Psi_{v,\mathbf{a}}:=\psi_v(\langle \cdot, \mathbf{a} \rangle):  \ga(\ad) \rightarrow \cx^*
\end{equation} on $\ga(F_v)$.
For  $\mathbf{a}\in \ga(\ad)$, we define the character $$\Psi_{\mathbf{a}}:=\prod_v \Psi_{v,\mathbf{a}}$$ on $\ga(\ad)$.
The map $\ga(\ad)\rightarrow {\ga(\ad)}^\wedge, \> \mathbf{a}\mapsto \Psi_\mathbf{a}$ is a topological isomorphism (see \cite[Theorem IV.3]{Weil1}). The subgroup $\ga(F)\subset \ga(\ad)$ is discrete, cocompact, and we have an induced duality
\begin{equation}\label{eqn:FGlobDuality}
	\ga(F)\rightarrow {\left( \frac{\ga(\ad)}{\ga(F)} \right)^\wedge}, \> \mathbf{a}\rightarrow \Psi_\mathbf{a}.
\end{equation}

Let us get back to the case when $G$ is an $F$-form of $\ga$. By \Cref{cor:ex.non-van.diff}, we can view $G(\ad)$ as a closed subgroup of $\mathbb{G}_a^{n+1}(\ad)$. Thus, we have an isomorphism
$$G(\ad)^\wedge\cong \frac{ \mathbb{G}_a^{n+1}(\ad)^\wedge} { (\mathbb{G}_a^{n+1}(\ad)/ G(\ad))^\wedge}.$$
Our aim is to apply a version of the Poisson summation formula for the height zeta function with respect to the subgroup $G(F)$ in $G(\ad)$. 

\begin{lem}
	The subgroup $G(F)\subset G(\ad)$ is discrete and cocompact.
\end{lem}
\begin{proof}
	By \Cref{cor:ex.non-van.diff}, $G$ is isomorphic to a subgroup of $\mathbb{G}_a^{n+1}$, where we identify $G$ with its image via this isomorphism. Thus, $G(\ad)$ is closed in $\mathbb{G}_a^{n+1}(\ad)$ and $G(F)=G(\ad)\cap \mathbb{G}_a^{n+1}(F)$. This proves that $G(F)$ is discrete in $G(\ad)$. The following commutative diagram
	% https://q.uiver.app/#q=WzAsNCxbMCwwLCJHKFxcbWF0aGJie0F9X0YpIl0sWzEsMCwiXFxtYXRoYmJ7R31fYV57bisxfShcXG1hdGhiYntBfV9GKSJdLFswLDEsIkcoRikiXSxbMSwxLCJcXG1hdGhiYntHfV9hXntuKzF9KEYpIl0sWzAsMV0sWzIsM10sWzIsMF0sWzMsMV1d
	\[\begin{tikzcd}
		{G(\mathbb{A}_F)} & {\mathbb{G}_a^{n+1}(\mathbb{A}_F)} \\
		{G(F)} & {\mathbb{G}_a^{n+1}(F)}
		\arrow[from=1-1, to=1-2]
		\arrow[from=2-1, to=2-2]
		\arrow[from=2-1, to=1-1]
		\arrow[from=2-2, to=1-2]
	\end{tikzcd}\]
	where the horizontal maps are the closed embeddings, induces a closed embedding
	$$G(\ad)/G(F) \rightarrow \mathbb{G}_a^{n+1}(\ad)/\mathbb{G}_a^{n+1}(F). $$
	As $\mathbb{G}_a^{n+1}(\ad)/\mathbb{G}_a^{n+1}(F)$ is compact, we deduce that $G(\ad)/G(F)$ is compact. 
\end{proof}

We now prove a lemma that will be used to show that only finitely many Fourier transforms of the height function are non-zero.

\begin{lem}\label{lem:Finite-Quo}
	Let $M$ be a non-empty open subgroup of $G(\ad)$. Then the quotient $G(\ad)/(M+G(F))$ is finite.
\end{lem}
\begin{proof}
	Since $M+G(F)$ induces an open subgroup of $G(\ad)/G(F)$, the cosets of $M+G(F)$ in $G(\ad)$ give an open cover. The result now follows by the compactness of $G(\ad)/G(F)$.
\end{proof}

It is a known fact (see e.g. the proof of \cite[IV, Lemma 3]{Weil1}) that the group of adelic points of $F=\mathbb{F}_q(t)$ has a decomposition
$$ \ad= F \oplus  \left( \mathfrak{m}_\infty   \times \prod_{v\neq \infty} \ov \right),$$
where $\infty$ is place of $F$ corresponding to $1/t$. One might ask whether the adelic points of forms of $\ga$ have a similar decomposition. We show that such a decomposition exists for a certain class of forms.

\begin{prop}\label{prop:SA-close-property}
	Assume that $F=\mathbb{F}_q(t)$. Let $G$ be a twist of $\ga$ defined, as a subgroup of $\mathbb{G}_a^{n+1}$, by a separable $p$-polynomial 
	$$P=\sum_{i,j}^{n+1,r} c_{i,j} x_i^{p^j}$$ 
	such that each $c_{i,j}\in \mathfrak{o}_F =\mathbb{F}_q[t]$. If $\deg (c_{i,j})< p^j $ for all $i,j$, then we have a decomposition
	$$G(\ad)= G(F) \oplus \left( \mathbb{H}_\infty \times \prod_{v\neq \infty}G(\ov) \right)$$
	where $\mathbb{H}_\infty=\mathfrak{m}_\infty^{n+1} \cap G(F_\infty)$ the intersection being taken in $\mathbb{G}_a^{n+1}(F_\infty)$.
\end{prop}
\begin{proof}
	By the decomposition $\ad= F \oplus  C$, where $C:=\left( \mathfrak{m}_\infty \times \prod_{v\neq \infty} \ov \right)$, we have that
	$$\oplus^{n+1} \ad = (\oplus^{n+1} F ) \oplus (\oplus^{n+1} C ).$$
	Let $x:=(x_i)^{n+1}_{i=1}=(r_i)_{i=1}^{n+1}+ (y_i)_{i=1}^{n+1} \in \oplus^{n+1} \ad$, where $(r_i)\in \oplus^{n+1} F$ and $ (y_i) \in \oplus^{n+1} C$. If $x\in G(\ad)$, then $$P(x)=\sum_{i,j}^{m,r} c_{i,j} (r_i+y_i)^{p^j}= \sum_{i,j}^{m,r} c_{i,j} r_i^{p^j} + \sum_{i,j}^{m,r} c_{i,j} y_i^{p^j} = 0.$$
	As $\deg (c_{i,j})< p^j $ for all $i,j$,  we have $\sum_{i,j}^{m,r} c_{i,j} y_i^{p^j}\in \oplus^{n+1} C$. Clearly, we have  $\sum_{i,j}^{m,r} c_{i,j} r_i^{p^j} \in \oplus^{n+1} F$. Therefore, $P((r_i))= 0$ and $P((y_i))=0$, which implies that $(r_i)\in G(F)$ and $(y_i)\in  \mathbb{H}_\infty \times \prod_{v\neq \infty}G(\ov) $.
\end{proof}

Let $d\mathbf{g}$ be the Haar measure on $G(\ad)$ defined in \Cref{TamHaarMeasure}. We define the \textit{Fourier transform} in the adelic component of the height pairing on $G(\ad) \times \pic(X)_\cx$ by
$$\hat{H}(\Psi; \mathbf{s})=\int_{\mathbf{x}\in G(\ad)} H(\mathbf{x};\mathbf{s})^{-1} \Psi(\mathbf{x}) \> d \mathbf{g}$$
for $\Psi\in \widehat{G(\ad)}$. Recall that $H$ is $\mathbb{K}$-invariant by \Cref{prop:height-K-inv}, where $\mathbb{K}$ was the intersection of the maximal open compact subgroups of $G(\ad)$ fixing ${(\mathcal{O}(D_\alpha),\> ||\cdot||_v)}$, for all $\alpha \in \mathcal{A}$.

\begin{prop}\label{prop:fin.many.trivial}
	The set $\{\Psi\in \widehat{G(\ad)}:\> \Psi(\mathbb{K}+ G(F))=1\}$ is finite.
\end{prop}
\begin{proof}
	The set is isomorphic to the dual of $G(\ad)/(\mathbb{K}+ G(F))$. By \Cref{lem:Finite-Quo}, $G(\ad)/(\mathbb{K}+ G(F))$ is finite, which implies that its dual is finite as well.
\end{proof}

\begin{lem}\label{lem:non-trivialOnK-FT-vanishes}
	If $\Psi\in \widehat{G(\ad)}$ is non-trivial on $\mathbb{K}$, then $\hat{H}(\Psi; \mathbf{s})=0$ for all $\mathbf{s}\in \pic(X)_\cx$ such that $H(\cdot; \mathbf{s})^{-1} $ is integrable.
\end{lem}
\begin{proof}
	By the invariance of the height function under the action of the compact subset $\mathbb{K}$, we have
	\begin{equation*}
		\begin{split}
			\hat{H}(\Psi; \mathbf{s})&= \int_{x\in G(\ad)/\mathbb{K}} \int_{\mathbf{b}\in \mathbb{K}}  H(x+\mathbf{b}; \mathbf{s})^{-1}\Psi(x+\mathbf{b})\> d\mathbf{b} \> dx
			\\ &= \int_{x\in G(\ad)/\mathbb{K}} H(x; \mathbf{s})^{-1} \Psi(x) \int_{\mathbf{b}\in \mathbb{K}}  \Psi(\mathbf{b} )\> d\mathbf{b} \> dx
			\\ &=0
		\end{split}
	\end{equation*}
	where the last equality follows since $\Psi$ is non-trivial on the compact subgroup $\mathbb{K}$.
\end{proof}

\begin{cor}\label{cor:FourTransfSummable}
	For all $\mathbf{s}\in \pic(X)_\cx$ such that $H(\cdot; \mathbf{s})^{-1} $ is integrable, the series
	$$\sum_{\Psi\in (G(\ad)/G(F))^\wedge}  \hat{H}(\Psi; \mathbf{s})$$
	converges absolutely.
\end{cor}                                                                                                                                                                                    
\begin{proof}
	Note that $$|\hat{H}(\Psi; \mathbf{s})|\leq  \int_{\mathbf{x}\in G(\ad)} | H(\mathbf{x};\mathbf{s})|^{-1}  \> d \mathbf{g}.$$ 
	The result now follows by \Cref{prop:fin.many.trivial} and  \Cref{lem:non-trivialOnK-FT-vanishes}.
\end{proof}
 
We now apply the version of Poisson summation formula in \cite[p.280]{ManinPan}.
\begin{thm}\label{prop:Z(s)-Poisson}
	For all $\mathbf{s}\in \pic(X)_\cx$ such that $H(\cdot; \mathbf{s})^{-1} $ is absolutely integrable, the following holds
	\begin{equation}\label{eqn:Poisson}
		Z(\mathbf{s})= \frac{1}{q^{\dim(G)(g_F-1)}{\tau(G)}} \sum_{\Psi\in (G(\ad)/(G(F)+\mathbb{K}))^\wedge}  \hat{H}(\Psi; \mathbf{s}).
	\end{equation}
\end{thm}
\begin{proof}
	We show that the conditions of the Poisson summation formula in \cite[p.280]{ManinPan} hold. First, we need to show that $\hat{H}(\Psi; \mathbf{s})$ is summable over $(G(\ad)/G(F))^\wedge$, which follows by \Cref{cor:FourTransfSummable}. Note that it suffices to sum over  $(G(\ad)/G(F+\mathbb{K}))^\wedge$ by \Cref{lem:non-trivialOnK-FT-vanishes}. We also need to show that $$\sum_{\mathbf{x}\in G(F)} H(\mathbf{x}+\mathbf{b}; \mathbf{s})^{-1},$$
	converges absolutely and uniformly when $\mathbf{b}$ belongs to $G(\ad)/G(F)$; this follows by an analogous proof to \cite[Lemma 5.2]{CLYT}. Finally, note that the volume of $G(\ad)/G(F)$ with respect to the measure $d\mathbf{g}$ is $q^{\dim(G)(g_F-1)}{\tau(G)}$ by \eqref{eqn:TamagawaNumberOfG}.
\end{proof}

\subsection{Characters and the Brauer group} \label{subsec:BrToChar}
We end this section by describing a correspondence between the characters of $G(\ad)$ and a subgroup of $\Br(G)$. This will allow us to compute the local Fourier transforms at characters with poles of order divisible by $p$, which is a subtle case in the positive characteristic setting (see \Cref{prop:H_v-estimate-otherChars} below). 

Let $K$ be a field of positive characteristic $p$. Set $\varphi(x):=x^p-x \in K[x]$, which defines an endomorphism of $\mathbb{G}_a(K)$. For $a\in K$, let $\chi_a:\gal(K(\alpha)/K)\rightarrow \z/p\z$ be a character, where $K(\alpha)$ is the extension given by $\varphi(\alpha)-a=0$. This defines an element in $\Ho^1(K,\mathbb{Q}/\z)\xrightarrow{\sim} H^2(K,\z)$. Let $b\in K^* = H^0(K,(K^s)^{*})$. Then the cup product of $\chi_a$ and $b$ defines an element in $\Br(K)$, which we write as $(\chi_a,b)$. For $a\in K$, $b\in K^*$, we define the symbol
$$[a,b):= (\chi_a,b) \in \Br(K)[p].$$

\begin{prop}{\cite[Prop. 11, p.215]{Serre1}} \label{prop:abproperties}
	\begin{enumerate}[label=(\roman*)]
		\item $[a+a',b)= [a,b)+[a',b) $.
		\item $[a,bb')= [a,b)+[a,b') $.
		\item $[a,b)=0$ if and only if $b$ is a norm in the extension $K(\alpha)/K$.
	\end{enumerate}
\end{prop}

\begin{cor}\label{lem:Br-xp-x}
	For all $a\in K$, $b\in K^*$, and $m\in \z_{>0}$, we have that 
	$$[a^{p^m},b)=[a,b).$$
\end{cor}
\begin{proof}
	As $\alpha=a^{p^{m-1}}$ is a root for $\varphi(x)=a^{p^{m}}-a^{p^{m-1}}$, we have that $K(\alpha)=K$ so that $b$ is trivially a norm in $K(\alpha)/K$. Therefore, by \Cref{prop:abproperties},
	$$[a^{p^{m}},b) - [a^{p^{m-1}},b) = [a^{p^{m}}-a^{p^{m-1}},b) =0.$$
	The result now follows by a successsively applying this last relation.
\end{proof}

We now restrict to the case $K=F_v$, where we fix an isomorphism ${F_v\cong \mathbb{F}_v ((\pi_v))}$ induced by fixing a unformising paramater $\pi_v$ for $F_v$. Let $\omega= f \> d \pi_v $ be a differential form of $F_v$. Then the coefficient of $\pi_v^{-1}$ in the expansion of $f$ is called the residue of $\omega$, which we denote by $\residue(\omega)$. This residue is independant of the choice of the uniformising parameter. For $b\in F_v^*$, we can write $db=(db/d\pi_v) d\pi_v$ where $db/d\pi_v$ is the formal derivative of $b$ with respect to $\pi_v$, viewed as an element of $\mathbb{F}_v((\pi_v))$. If $b\notin F_v^p$, then $db/d\pi_v$ is non-zero. 

By local class field theory, we have an isomorphsim $\inv:\Br(F_v)\xrightarrow{\sim} \q/\z$ mapping $[a,b)$ to $\inv [a,b)$ in $(1/p)\z/\z$. Following Serre \cite[chapter 14, section 5]{Serre1}, we define $[a,b)_v:=p \cdot \inv [a,b) \in \z/p\z$. The local symbol $[a,b)_v$ satifies the same properties as \Cref{prop:abproperties}.

Let $\chi: \mathbb{F}_p \rightarrow \cx$ be the additive character mapping $1$ to $\text{exp}({2\pi i}/q)$. We define the local non-trivial character 
\begin{equation}\label{eqn:varphiChar}
	\varphi_v: F_v \rightarrow \cx, \> x_v \mapsto \chi ( \text{Tr}_{\fv/\mathbb{F}_p}(x_{-1,v})),
\end{equation}
for all places $v$. Note that this character is not the same as the character $\psi_v$ defined in \eqref{eqn:localCharacter}.

\begin{lem}\label{prop:localBrEltToChar}
	Fix an isomorphism $F_v\cong \mathbb{F}_v ((\pi_v))$ and let $a\in F_v$, $b\in F_v^*$. Then
	$$\chi([a,b)_v)=\varphi_v\left(\frac{a}{b} \frac{db}{d\pi_v} \right).$$
\end{lem}
\begin{proof}
	By \cite[Prop.15, p.217]{Serre1} and the definition of $\residue$, we have
	$$\chi([a,b)_v)= \chi(\Tr_{\mathbb{F}_v/\mathbb{F}_p}(\residue(a(db/d\pi_v)/b \> d\pi_v)))= \varphi_v\left(\frac{a}{b} \frac{db}{d\pi_v} \right),$$
	where the last equality holds by \eqref{eqn:varphiChar}.
\end{proof}

We now turn our attention to $\Br(\mathbb{A}_K^1)$, which can be viewed as a subgroup of $\Br(K(x))$, where $x$ is the coordinate of $\mathbb{A}_K^1$. If $K$ is non-perfect, then $\Br(\mathbb{A}_K^1)$ is $p$-torsion and non-trivial. Indeed, we follow the construction of a non-zero element described in \cite[Theorem 5.6.1]{CTSK}. Let $b\in K^*\setminus (K^*)^p$ which corresponds to a non-zero class in $\Ho^1_{\text{fppf}}(K,\mu_p)$. The morphism $K[x]\rightarrow K[y], \> x \mapsto y^p-y$ induces an étale Artin-Schreier cover of $\mathbb{A}_K^1$, which corresponds to a non-zero element of $\Ho^1_{\text{ét}}(\mathbb{A}_K^1, \z/p\z)= \Ho^1_{\text{fppf}}(\mathbb{A}_K^1, \z/p\z)$. The cup product of these two classes gives a non-zero element in $H^2_{\text{fppf}}(\mathbb{A}_K^1,\mu_p) = \Br(\mathbb{A}_K^1)[p]$, which precisely corresponds to the class $[x,b)\in \Br(K(x))$ via the embedding of $\Br(\mathbb{A}_K^1)$ in $\Br(K(x))$. Similarly, the cup product of $b$ with the cover given by $y^p-y=ax$, for $a\in K$, corresponds to $[ax,b)\in \Br(\mathbb{A}_K^1)\subset \Br(K(x))$. 

Let $b\in K^*\setminus (K^*)^p$. We define the subgroup
$$\Br_b(\mathbb{G}_{a,K}^n):= \{ [a_1 x_1 + \cdots + a_n x_n, b) \in \Br(\mathbb{G}_{a,K}^n) \> : \> (a_1,\ldots,a_n)\in \ga(K)\}.$$ 
Let $\phi_v$ be any non-trivial local character of $F_v$. Then $F_v$ is self-dual via the isomorphism
\begin{equation}\label{eqn:F_v-SelfDual}
	F_v \rightarrow \widehat{F_v}, \>\> a \mapsto \phi_{v,a}
\end{equation} 
where $\phi_{v,a}(x):=\phi_v(ax)$.

\begin{prop}\label{prop:Ga-LocalCharBr}
	Let $F_v$ be a local field of positive characteristic and ${b\in F_v^*\setminus (F_v^*)^p}$. Let
	$$\Br_b(\mathbb{G}^n_{a,F_v}) \rightarrow \widehat{\ga(F_v)}$$
	be the map that takes an element $[a_{v,1} x_1 + \cdots + a_{v,n} x_n ,b_v)$, to the character 
	$$(x_1, \ldots , x_n) \mapsto \chi([a_{1} x_1 + \cdots + a_{n} x_n ,b)_v).$$
	Let $\mathbf{a}:=(a_{1},\ldots, a_{n}) \in \ga(F_v)$. Then
	\begin{enumerate}
		\item The image of $[a_{v,1} x_1 + \cdots + a_{v,n} x_n ,b_v)\in \Br_b(\mathbb{G}^n_{a,F_v})$ is the character $$\varphi_{v,(db/d\pi_v)/b}( \langle \cdot, \mathbf{a} \rangle).$$
		\item This map is an isomorphism of (additive) groups.
	\end{enumerate}
\end{prop}
\begin{proof}
	We first prove the result for $\mathbb{G}_a$, i.e., when $n=1$. By \Cref{prop:abproperties}, the map $x\mapsto \chi([a x,b)_v))$ is a character of $F_v$ for a fixed $a\in F_v$; by \Cref{prop:localBrEltToChar}, it coincides with the character $\varphi_{v,a(db/d\pi_v)/b}$, which prove (1) for $\mathbb{G}_a$. 
	
	To prove $(2)$, we first note that as $b\notin F_v^p$, we have $db/d\pi_v\neq 0$. By $(1)$, the image of $[ax,b)\in \Br_b(\mathbb{G}_{a,F_v})$ is the character $\varphi_{v,a(db/d\pi_v)/b}$. If this character is trivial, then by $\eqref{eqn:F_v-SelfDual}$, we see that $a(db/d\pi_v)/b=0$. Thus, as $(db/d\pi_v)/b\neq 0$, we see that $a=0$, so that $[ax,b)=0$. This shows that the map is injective. To show that the map is surjective, let $\varphi_{v,c}\in \widehat{F_v}$, for some $c\in F_v$. Then, by $(1)$, we easily see that $[(cbx/(db/d\pi_v),b)$ maps to $\varphi_{v,c}$.
	
	For the general case, observe that the natural map
	$$\prod^n \Br(\mathbb{G}_a) \rightarrow \Br(\ga)$$
	restricts to an isomorphism $\prod^n \Br_b(\mathbb{G}_a) \xrightarrow{\sim} \Br_b(\ga)$ such that
	$$[a_{v,1} x_1,b) \times \cdots \times [a_{v,n} x_n,b) \mapsto [a_{v,1} x_1+ \cdots+ a_{v,n} x_n, b).$$ The result now follows.
\end{proof}

\begin{cor}
	$\Br{\mathbb{A}_{F_v}^n}$ has a subgroup isomorphic to $\prod^n F_v$.
\end{cor}

Let $F$ be a global function field. By \eqref{eqn:FGlobDuality}, we have an isomorphism
\begin{equation}\label{eqn:FGlobDuality2}
	F\xrightarrow{\sim} (\ad/F)^\wedge, \> a\rightarrow \psi_{a}:=\prod_v \psi_{v,a}.
\end{equation}
 For a place $v$ of $F$, the pullback map $\Br(\mathbb{G}^n_{a,F}) \rightarrow \Br(\mathbb{G}^n_{a,F_v})$
induces a morphism 
$$\Br_b(\mathbb{G}^n_{a,F}) \rightarrow \widehat{\ga(F_v)},$$ for $b\in F$.
This gives a group homomorphism 
$$\Br_b(\mathbb{G}^n_{a,F}) \rightarrow \prod_v \widehat{\ga(F_v)}.$$

\begin{prop}\label{prop:Gan-LocalCharBr}
	Let $F$ be a global function field and ${b\in F^*\setminus (F^*)^p}$. Let
	$$\Br_b(\mathbb{G}^n_{a,F}) \xrightarrow{\sim} \prod_v \widehat{\ga(F_v)}$$ be the map that takes an element $[a_{1} x_1 + \cdots + a_{n} x_n ,b)$ to the character $$ (x_{1_v},\ldots, x_{n,v})_{v} \mapsto \prod_v \chi([a_{1} x_{1,v} + \cdots + a_{n} x_{n,v} ,b)_v)$$ 
{}	where $\mathbf{a}:=(a_{1},\ldots, a_{n}) \in \ga(F)$. Then
	\begin{enumerate}
		\item The image of $[a_{1} x_1 + \cdots + a_{n} x_n ,b)$ is the character $$\prod_{v} \varphi_{v,(db/d\pi_v)/b}( \langle \cdot, \mathbf{a} \rangle).$$
		\item The map is injective and its image is $(\ga(\ad)/\ga(F))^\wedge$.
		\item If $F=\mathbb{F}_q(t)$, then the image of $[t^{-3} x,t)$ is the character $\psi$ defined in \eqref{eqn:CharOnF0}.
	\end{enumerate}
\end{prop}
\begin{proof}
	First, $(1)$ clearly follows from \Cref{prop:Ga-LocalCharBr}(i). 
	We prove (2) for $\mathbb{G}_a$; the result then easily follows for $\ga$.  Let $[ax,b)\in \Br_b(\mathbb{G}_{a,F})$, where $a\in F^*$. For almost all places $v$, we have $a,b\in \ov^*$. As $b\notin F^p$ and $F_v/F$ is a separable extension, $b\notin F_v^p$ for all places $v$; hence, the formal derivative of $b$ in $F_v\cong \mathbb{F}_v((\pi_v))$, which we denoted by $db/d\pi_v$, is non-zero. Thus, $a(db/d\pi_v)/b=0$ if and only if $a=0$. By (1), this shows that the map is injective. Now, as $a,b\in \ov^*$ for almost all $v$, we see that $a(db/d\pi_v)/b$ is integral for almost all $v$. Thus, the local character $\chi([a x,b)_v)=\psi_{v,a(db/d\pi_v)/b}(x)$ is trivial on $\ov$ for almost all places $v$, implying that the image of $[a x,b)$ is indeed a character of $\mathbb{A}_F$. Moreover, by the exact sequence from class field theory 
	$$0\rightarrow \Br(F) \rightarrow \bigoplus_v \Br(F_v) \rightarrow \mathbb{Q}/\z \rightarrow 0,$$ we see that the image of $[a x,b)$ is a character that is trivial on $F$. By \eqref{eqn:FGlobDuality}, the image of $[a x,b)$ is $\psi_c$ for some $c\in F$, so that the image of $[d a x,b)$ is $\psi_{  dc}$ for all $d \in F$. By \eqref{eqn:FGlobDuality} again, this shows that the image of the map is $(\ad/F)^\wedge$, which proves $(2)$.
	
	We now prove (3). Let $F=\mathbb{F}_q(t)$ and $\pi_\infty=1/t$ be a uniformiser for $F_\infty$. Then $dt/d\pi_\infty= -\pi_\infty^{-2}=-t^2$. By (1), the image of $[t^{-3}x,t)$ is $\varphi_{\infty, -\pi_\infty^{-2}} \cdot \prod_{v\neq \infty} \varphi_{v,a(db/d\pi_v)/b}$ since $a (db/\pi_\infty) / b= t^{-3}(-t^{2})/t=-t^{-2}$.
	By the the definitions \eqref{eqn:CharOnF0} and \eqref{eqn:CharOnF0Infty}, we have $\varphi_{\infty, -\pi_\infty^{-2}}=\psi_{\infty}$. As $\psi\in (\ad/F)^\wedge$ was defined as the unique character extending $\psi_\infty$, we conclude that  $[t^{-3} x,t)$ is mapped to the character $\psi$.
\end{proof}

\begin{lem}\label{lem:psiIsVarphiUnit}
	Let $\psi=\prod_v \psi_v \in \widehat{\ad}$ be the character defined in \eqref{eqn:CharOnF0}. Then for almost all places $v$, we have $\psi_v=\varphi_{v,c_v}$ for some unit $c_v\in \ov^*$.
\end{lem}
\begin{proof}
	It suffices to prove the result for $F=\mathbb{F}_q(t)$. By \eqref{eqn:F_v-SelfDual}, for all $v\in \Omega_F$ we can write $\psi_v=\varphi_{v,c_v}$ for some $c_v\in F_v$. By the construction of $\psi$ (see the proof of \cite[Theorem IV.3]{Weil1}), we have that $\psi_v(\ov)=1$ for all $v\neq \infty$. This implies that $c_v\in \ov$ for all $v\neq \infty$. Suppose that $c_w \in \mathfrak{m}_w$ for some place $w\neq \infty$. Then the character $\psi_{\pi_w^{-1}}$ is trivial on $\mathbb{A}_\infty:= \mathfrak{m}_\infty   \times \prod_{v\neq \infty} \ov$, as $\psi_{\infty,\pi_w^{-1}}$ is trivial on $\mathfrak{m}_\infty$ (see \eqref{eqn:CharOnF0Infty}) and $\psi_{w,\pi_w^{-1}}=\varphi_{w,c_w/{\pi_w}}$ is trivial on $\mathfrak{o}_w$.
	By the decomposition ${\ad= F \oplus \mathbb{A}_\infty}$, we deduce that $\psi_{\pi_w^{-1}}$ is trivial on $\ad$. This contradicts \eqref{eqn:FGlobDuality2} since $\psi$ is non-trivial.
\end{proof}

\begin{cor}
	Let $F$ be a global function field and ${b\in F^*\setminus (F^*)^p}$. Then $db/d\pi_v$ is a unit in $\ov$ for all but finitely many places $v$; equivalently, the coefficient of $\pi_v$ in the $v$-adic expansion of $b$ in $F_v$ is non-zero for all but finitely many places $v$. 
\end{cor}
\begin{proof}
	The result follows by \Cref{prop:Gan-LocalCharBr}, \Cref{lem:psiIsVarphiUnit}, and \eqref{eqn:FGlobDuality2}.
\end{proof}

We end this section by providing a similar description for the characters of $G(F_v)$ and $G(\ad)/G(F)$ for any $F$-form $G$ of $\mathbb{G}_a^{n}$. Recall that by \Cref{cor:ex.non-van.diff}, $G$ is a subgroup $\mathbb{G}_a^{n+1}$, so that $G(F_v)$ and $G(\ad)$ are closed subgroups of the locally compact groups $\mathbb{G}_a^{n+1}(F_v)$ and $\mathbb{G}_a^{n+1}(\ad)$, respectively. Therefore, the following maps induced by the restriction of characters 
$$\mathbb{G}_a^{n+1}(F_v)^\wedge \rightarrow G(F_v)^\wedge, \quad (\mathbb{G}_a^{n+1}(\ad)/\mathbb{G}_a^{n+1}(F))^\wedge \rightarrow (G(\ad)/G(F))^\wedge $$
are surjective. 

Denote by $\iota:G\rightarrow \mathbb{G}_a^{n+1}$ the closed embedding of $G$ into $\mathbb{G}_a^{n+1}$. Let $b\in F^*\setminus (F^*)^p$ and $b_v\in F^*\setminus (F_v^*)^p$. We define
$$\Br_b(G):= \iota^* \Br_b(\mathbb{G}_{a}^{n+1}) \quad \text{and} \quad \Br_{b_v}(G_{F_v}):= \iota^* \Br_{b_v}(\mathbb{G}_{a,F_v}^{n+1}).$$
By commutativity and functionality of the Brauer group, one easily deduces the following from the case of $\ga$.

\begin{prop}
	Let $F$ be a global function field. 
	\begin{enumerate}
		\item For any place $v$ and ${b_v\in F_v^*\setminus (F_v^*)^p}$, the morphism
		$$\Br_b(G_{F_v}) \rightarrow G(F_v)^\wedge$$
		is surjective.
		\item For any ${b\in F^*\setminus (F^*)^p}$, the morphism 
		$$\Br_b(G_{F}) \rightarrow (G(\ad)/G(F))^\wedge$$
		is surjective.
	\end{enumerate}
\end{prop}

\section{Fourier transforms at the trivial character}\label{sec:FT-TrivialChar}
In this section, we compute the global Fourier transform at the trivial character. We start by bounding the local Fourier transform at a general place $v\in \Omega_F$. Subsequently, we explicitly compute the local transform at good places in the form of what is called Denef's formula (c.f. \cite[Theorem 9.1]{CLYT} and \cite[Theorem 3.1]{Denef}).

\subsection{The setting and some notation} 
We fix the following notation and assumptions for the rest of the paper. Let $G$ be an $F$-form of $\ga$ and $X$ a smooth equivariant compactification of $G$ with boundary $D$, where $\{D_\alpha\}_{\alpha\in \mathcal{A}}$ is the set of irreducible components $D$. Let $\mathcal{B}:=\{\alpha \in \mathcal{A}\>:\> D_\alpha \text{ is not geometrically reduced}\}.$ We choose and fix smooth adelic $q$-metrics on the line bundles $\mathcal{O}(D_\alpha)$ for all $\alpha\in \mathcal{A}$. We fix a $\cf$-model as in \Cref{defn:adelic-met} for $\mathcal{X}$, and let $\mathcal{D}$, $\mathcal{G}$ and $\mathcal{D}_\alpha$ be the closures of $D$, $G$, and $D_\alpha$ in $\mathcal{X}$, respectively. We assume that \Cref{assum} holds. We write $\mathcal{X}_v$ for $\mathcal{X} \times \ov$ and ${\mathcal{X}}_{\fv}$ for $\mathcal{X}_v \otimes \fv$. Let $\redv:\mathcal{X}(\ov)\rightarrow \mathcal{X}(\fv)$ be the reduction map mod $\pi_v$. 

For every $v\in \Omega_F$, we fix an embedding $F^s\subset F_v^s$, so that $\Gamma_v:=\gal(F_v^s/F_v)$ induces an action on $X^s$ and $D^s$. We denote the indexing sets for $D^s$  by $\mathcal{A}^s$ and we write $\mathcal{A}_v$ for $\mathcal{A}^s/\Gamma_v$. Then we can write $D_v=\cup_{\alpha_v \in \mathcal{A}_v} D_{\alpha_v}$, where each $D_{\alpha_v}$ is irreducible over $F_v$.  By \Cref{assum}, for all $\beta\in \mathcal{B}$, we see that $D_{\beta,v}$ is an irreducible component of $D_v$ and $D_{\beta}^s$ is an irreducible component of $D^s$. Thus, we will identify $\mathcal{B}$ as a subset of $\mathcal{A}_v$ and $\mathcal{A}^s$. For $\alpha_v \in \mathcal{A}_v$, we denote the algebraic closure of $F_v$ in the function field of $D_{\alpha_v}$ by $F_{\alpha_v}$, and set $f_{\alpha_v}:= [F_{\alpha_v}:F_v]$. 

Let $S\subset \Omega_F$ be the set of places in \Cref{defn:adelic-met}.2. We enlarge $S$ such that  $\mathcal{X}$ has good reduction away from $S$ and $\bigcup_{\alpha_v \in \mathcal{A}_v\setminus \mathcal{B}} \mathcal{D}_{\alpha_v}$ is a relative strict normal crossings divisor over $\ov$ for all $v\notin S$ (see e.g. \cite[\S 2]{Illusie} for the definition a relative strict normal crossings divisor).

\subsection{Motivation for \Cref{assum}}\label{subsection:insep-assump}
Recall from \Cref{assum} the following conditions for all $\beta\in \mathcal{B}$:
\begin{enumerate} [label=(\arabic*)]
	\item $D_\beta$ is geometrically irreducible;
	\item $D_\beta (F_v)=\emptyset$ for all $v\in \Omega_F$;
	\item  For all $ v\notin S$ and for all $ x\in X(F_v)$ the value $\lVert \mathsf{s}_\beta \rVert_v (x)$ depends only on ${\redv(x) \in {\mathcal{X}}(\fv)}$ so that $\lVert \mathsf{s}_\beta \rVert_v (\cdot)$ is well defined on $\mathcal{X}(\fv)$ (hence $\lVert \mathsf{s}_\beta \rVert_v (\cdot)$ attains finitely many values);
	
	\item For all $v \notin S$, we have
	\begin{equation*}
		\#\{\bar{x}\in {\mathcal{D}}_\beta (\fv): \> \lVert \mathsf{s}_\beta \rVert_v(\bar{x})=q_v^{-1} \}=q_v^{\dim(X)-1}+O(q_v^{\dim(X)-3/2})
	\end{equation*}
	where the implicit constant only depends on $X$.
	\item  If $\mathcal{B}\neq \emptyset$, then
	$$e_X:=\text{sup}_{\beta\in \mathcal{B}}\{m\in \z_{>0}:\>  \lVert \mathsf{s}_\beta \rVert_v(x)= q_v^{-m}, \text{ for some } v\in\Omega_F \text{ and } x\in X(F_v)\}$$ 
	is finite. If $\beta= \emptyset$, we set $e_X=1$.
\end{enumerate}

	We now motivate the conditions above.
	\begin{itemize}[itemsep=2pt,parsep=1pt,topsep=1pt]
		\item First, note that if $\mathcal{B}\neq \emptyset$, then $\pic(G) \neq 0$ by \Cref{prop:picX}. In particular, $G$ decomposes as $\mathbb{G}_a^m \times G_W$ where $G_W$ is a non-trivial $F$-wound group (see \cite[V.5]{Ost}). Thus, if $D$ is not geometrically reduced, then $G$  has to have a direct summand which an $F$-wound group. 
		\item To motivate Condition $(2)$, we recall that if $G$ is $F$-wound then $D(F_v)=\emptyset$ by \Cref{prop:BoundaryEmpty}. In the general case, we know at the very least that $D(F_v)$ is not Zariski dense in $D$ (see \cite[Proposition 2.3.26]{Poonen}).
		\item Condition $(3)$ is inspired by the fact that for an $F$-wound group $G$, the group $G(F_v)$ is compact for all $v$. Hence $G(F_v)/G(\ov)$ has finitely many elements, which in turn, implies that the local height takes finitely many values for good places $v$.
		\item Condition $(4)$, is motivated by \Cref{prop:P^p-1-valuations-insep}. For simplicity, let us suppose that $p=3$. Then, $D$ in \Cref{prop:P^p-1-valuations-insep} is given by
		$$X_0^3+tX_1^3+t^2X_2^3=0 \quad \subseteq \mathbb{P}^2.$$
		Let $a_i,b_i,c_i\in \fv$ be the coefficients of $1, \pi_v, \pi_v^2$, respectively,  in the $v$-adic expansion of $t^i$, for $i\in \{1,2\}$. Define the following subschemes of ${\mathcal{X}_{\fv}}$:
		$$D_0:=V(Y_0^3 + a_1 Y_1^3 + a_2 Y_2^3), \> D_1:=V(b_1 Y_1^3 + b_2 Y_2^3), \>D_2:=V(c_1 Y_1^3 + c_2 Y_2^3)$$
		Then for a primitive $x:=(x_0:x_1:x_2)\in \mathbb{P}^2(\ov)$, 
		$$v(x_0^3+tx_1^3+t^2x_2^3)=1 \iff \bar{x}\in D_0(\fv)\setminus D_1(\fv)$$
		and $v(x_0^3+tx_1^3+t^2x_2^3)=2 \iff \bar{x}\in D_1(\fv)\setminus D_2(\fv).$
		Thus, by the Lang-Weil estimates, we see that condition (4) is satisfied.
		
		\item Finally, asking that $e_X$ is finite implies that the values of $v(\mathsf{s}_\beta(x))$ are uniformly bounded for almost all places $v$. This, along with the previous bullet point, will ensure that convergence of the global Fourier transforms in appropriate domains.
	\end{itemize}

\subsection{Local Fourier transform at any place}
For $y\in X(F_v)$, we define $$\mathcal{A}_y=\{ \alpha_v \in \mathcal{A}_v: \> y\in D_{\alpha_v}\}.$$ 
As $D_\alpha$ is smooth for $\alpha\in \mathcal{A}\setminus \mathcal{B}$, the set of divisors $\{D_{\alpha_v} : \alpha_v\mid \alpha\}$ lying above $D_{\alpha,v}:= D_\alpha\times F_v$  do not intersect. Thus, $y$ can only lie on a unique $D_{\alpha_v}$ for $\alpha_v\mid \alpha$, for a fixed $\alpha\in \mathcal{A}$. This shows that we have a well defined injective map from $\mathcal{A}_y$ to $\mathcal{A}$ mapping $\alpha_v\in \mathcal{A}_y$ to $\alpha\in \mathcal{A}$ corresponding to the irreducible component $D_{\alpha,v}$ lying below $D_{\alpha_v}$. Thus, we can identify $\mathcal{A}_y$ with a subset of $\mathcal{A}$, which allows us to denote $\alpha_v$ by its image $\alpha$.

We now give a local analytic description for the integrals appearing in the computations of the local Fourier transforms with respect to an appropriate set of local coordinates using the strict normal crossings assumption on $D$. 

\begin{lem}\label{lem:local-coor-all-v}
	Let $y\in X_v(F_v)$. There exists an open subscheme $U_y$ containing $y$, and $\{y_{\alpha}\}_{\alpha\in \mathcal{A}_y}\cup \{y_i\}_{i=|\mathcal{A}_y|+1}^{n} \subset \mathcal{O}_X(U_y)$ such that 
	$$\phi_y:= (\{y_{\alpha}\}, \{y_i\}): U_y \longrightarrow \mathbb{A}^n$$ is étale and  restricts to a local homeomorphism $W_y\rightarrow \prod_{\alpha} \mathfrak{m}_v^{n_{\alpha}} \times \prod_i \mathfrak{m}_v^{n_i} \subset \mathbb{A}^n(F_v)$, where $n_{\alpha},n_i \in \z_{>0}$ and $W_y$ is an open analytic subset of $U_y(F_v)$; moreover, we have
	\begin{equation*}
		\int_{\mathbf{x}\in W_y} |H(x;\mathbf{s})|^{-1}  \> d \tau_{X,v} = C_y \int_{\prod_{\alpha} \mathfrak{m}_v^{n_{\alpha}} \times \prod_i \mathfrak{m}_v^{n_i}}  \prod_{\alpha \in \mathcal{A}_y}  |y_{\alpha}|_v ^{\re(s_{\alpha})-\rho_\alpha}  \> \bigwedge dy_{\alpha} \wedge \bigwedge dy_i
	\end{equation*}
	for some $C_y\in F_v$, where $\rho=(\rho_\alpha)_{\alpha\in \mathcal{A}}$ is the class of $\omega_X^{-1}$.

\end{lem}

\begin{proof}
	First, by \Cref{assum}(ii), we see that $\mathcal{B}\cap \mathcal{A}_y=\emptyset$, so that $D_{\alpha,v}$ is smooth for all $\alpha \in \mathcal{A}_y$.  As $X$ is smooth and $\cup_{\alpha\in \mathcal{A}\setminus\mathcal{B}} D_\alpha$ has strict normal crossings, there exists an open subscheme $U_y\subset X$ containing $y$ and $\{y_{\alpha}\}_{\alpha\in \mathcal{A}_y}\cup \{y_i\}_{i=|\mathcal{A}_y|+1}^{n} \subset \mathcal{O}_X(U_y)$ defining an étale morphism
	$$\phi_y: U_y \longrightarrow \mathbb{A}^n$$ such that $D_{v}|_{U_y}$ is defined by $\prod_{\alpha \in \mathcal{A}_y} y_{\alpha}=0$ in $U_y$. Let $$\omega_y=\bigwedge dy_{\alpha} \wedge \bigwedge dy_i\in \omega_X(U_y).$$ Then $\lVert \omega_y \rVert$ is non-vanishing and locally constant on $U_y(F_v)$ and the restriction of the measure $\tau_{X,v}$ to $U_y$ coincides with $|\omega_y|/ \lVert \omega_y \rVert$. By the definition of a smooth metric, for ${\alpha} \in \mathcal{A}_y$, we have that $\lVert \mathsf{s}_{\alpha} \rVert (x) = |y_{\alpha}|_v h_{{\alpha}}(x)$ on $U_y(F_v)$, where $h_{{\alpha}}:U_y(F_v) \rightarrow \mathbb{R}_+^*$ is a non-vanishing locally constant function. If $\alpha \in \mathcal{A}$ does not lie below any ${\alpha_v} \in \mathcal{A}_y$, we have $\lVert \mathsf{s}_{\alpha} \rVert (x) = h_{{\alpha}}(x)$ on $U_y(F_v)$, where $h_{{\alpha}}:U_y(F_v) \rightarrow \mathbb{R}_+^*$ is a non-vanishing local constant function. As $\phi_y$ is étale, the induced map on $U_y(F_v) \rightarrow \mathbb{A}^n(F_v)$ is a local homeomorphism for the analytic topology. Thus, there exists an analytic open $W_y \subset U_y(F_v)$ that is mapped homeomorphically to a basic open set of $\mathbb{A}^n(F_v)$ of the form $\prod_{\alpha} \mathfrak{m}_v^{n_{\alpha}} \times \prod_i \mathfrak{m}_v^{n_i}$, such that $\lVert \omega_y \rVert$ and every $h_{{\alpha}}(x)$ is constant on $W_y$.  Therefore, 
	\begin{equation*}
		\int_{\mathbf{x}\in W_{y}(F_v)} |c|_v \prod_{\alpha}  \lVert \mathsf{s}_{\alpha} \rVert (x_v)^{\text{Re}(s_{\alpha})-\rho_{\alpha}}  \> d \tau_{X,v} = C_y\int_{\prod_{\alpha} \mathfrak{m}_v^{n_{\alpha}} \times \prod_i \mathfrak{m}_v^{n_i}}  |y_{\alpha}|_v ^{\text{Re}(s_{\alpha})-\rho_{\alpha}}  \>  \bigwedge dy_{\alpha} \wedge \bigwedge dy_i
	\end{equation*}
	for some constant $C_y\in F_v$.
\end{proof}

We now bound the local Fourier transform at a general place $v$.  Define the open subset of $\pic(X)_\cx$ for $r\in \mathbb{R}:$  $$\Omega_t:=\{\mathbf{s} \in \pic(X)_\cx \> : \> \text{Re}(s_\alpha) > \rho_\alpha + r , \> \forall \alpha\in \mathcal{A} \} .$$

\begin{lem}\label{lem:trivial-bound-localTransform}
	The function $H_v(\cdot; \mathbf{s})^{-1}$ is integrable on $G(F_v)$ if $\mathbf{s}\in \Omega_{-1}$. Moreover, for all $\epsilon>0$ and any place $v$, there exists a constant $C_v(\epsilon)$ such that for all $\mathbf{s}\in \Omega_{-1+\epsilon}$ and all $\Psi\in \widehat{G(\ad)}$ we have
	$$|\hat{H}_v(\Psi;\mathbf{s})|\leq C_v(\epsilon).$$
\end{lem}

\begin{proof}
	First, we can assume that $\Psi$ is the trivial character.
	Since $X$ is projective, $X(F_v)$ is compact with respect to the analytic topology. Therefore, by \Cref{lem:local-coor-all-v}, there exists a finite set of points $y_i\in X(F_v)$, and an open cover $\{W_{y_j}\}_{j= 1}^{t}$ of $X(F_v)$, such that 
	\begin{equation*}
		\int_{\mathbf{x}\in W_{y_j}} |H(x;\mathbf{s})|^{-1}  \> d \mathbf{g}_v = C_{y_j} \int_{\prod_{\mathcal{A}_{y_j}} \mathfrak{m}_v^{n_{\alpha}} \prod_i \mathfrak{m}_v^{n_i}}  \prod_{{\alpha} \in \mathcal{A}_{y_j}}  |y_{\alpha}|_v ^{\re(s_{{\alpha}})-\rho_{\alpha}}  \> \bigwedge dy_{\alpha} \wedge \bigwedge dy_i
	\end{equation*}
	for all $j$.
	
	Again, as $X(F_v)$ is compact, there is a finite partition of unity $(f_{j})$ subordinate to this covering each with compact support; by this we mean that there exist  smooth functions $f_j:W_{y_j}\rightarrow [0,1] \subset \mathbb{R}$ for all $i\in \{1,\ldots, t\}$, such that 
	$$\sum_{j=1}^t f_j=1$$
	and the support of each $f_j$ is a compact subset of $W_{y_j}$.
	Therefore,
	\begin{equation*}
		\begin{split}
			\int_{ G(F_v)} |H_v(\mathbf{x};\mathbf{s})|^{-1} \> d \mathbf{g}_v &= \sum_{j=1}^t \int_{W_{y_j}(F_v)} |H_v(\mathbf{x};\mathbf{s})|^{-1} f_j(x) \> d \mathbf{g}_v 
			\\ &\leq \sum_{j=1}^t \int_{W_{y_j}(F_v)} |H_v(\mathbf{x};\mathbf{s})|^{-1} \> d \mathbf{g}_v
		\end{split}
	\end{equation*}
	Hence, for a fixed $j$, it suffices to prove that 
	$$ C_{y_j} \int_{\prod_{\mathcal{A}_{y_j}} \mathfrak{m}_v^{n_{\alpha}} \prod_i \mathfrak{m}_v^{n_i}}  \prod_{{\alpha} \in \mathcal{A}_{y_j}}  |y_{\alpha}|_v ^{\re(s_{{\alpha}})-\rho_{\alpha}}  \> \bigwedge dy_{\alpha} \wedge \bigwedge dy_i$$
	converges if $\mathbf{s}\in \Omega_{-1}$. This now follows by Fubini's theorem and \cite[Lemma 8.3]{CLYT}.
\end{proof}

\begin{lem}
	Suppose that $G$ is an $F$-wound group.
	The function $H_v(\cdot; \mathbf{s})^{-1}$ is integrable on $G(F_v)$ for all $\mathbf{s}\in \pic(X)_\cx$. 
\end{lem}
\begin{proof}
	The result follows from the compactness of $G(F_v)$ with respect to the analytic topology. 
\end{proof}

\subsection{Explicit computation of local Fourier transform at good places}\label{subsec:comp-trivial-char}
Let $v\notin S$. Let $\redv:\mathcal{X}_v(F_v)\rightarrow {\mathcal{X}}_v(\fv)$ be the reduction map mod $\pi_v$ (which is surjective for $v\notin S$). By \Cref{assum}(ii), for all $\beta \in \mathcal{B}$ and all $\bar{x}\in {\mathcal{X}}(\fv)$, we have that $$\lVert \mathsf{s}_\beta \rVert_v(x) = q_v^{\beta(\bar{x})}, \quad \forall x\in \redv^{-1}(\bar{x}),$$ 
for some  $\beta(\bar{x})\in \z_{>0}$ depending only on $\bar{x}$. Recall from \eqref{eqn:dg=dtau} the relation between the measures
$$ d\mathbf{g}_v =|c|_v \prod_{\alpha\in \mathcal{A}} \lVert \mathsf{s}_\alpha \rVert^{-\rho_\alpha} \> d\tau_{X,v}.$$
Since $|c|_v=1$ for almost all places, we will assume for simplicity that $|c|_v=1$. For $\bar{x} \in {\mathcal{X}}(\fv)$, we define the following subsets of $\mathcal{A}_v$:
$$\mathcal{A}_{\bar{x}}=\{ {\alpha_v} \in \mathcal{A}_v\setminus \mathcal{B}: \> \bar{x} \in \mathcal{D}_{{\alpha_v}}(\fv)\}, \quad \mathcal{B}_{\bar{x}}=\{ \beta \in \mathcal{B}\subset \mathcal{A}_v: \> \bar{x} \in \mathcal{D}_{\beta,v}(\fv)\}.$$

\begin{lem} \label{lem:local-coor-good-v}
	Let $\bar{x} \in {\mathcal{X}}(\fv)$. Then for all $\mathbf{s}\in \Omega_{-1}$ we have
	\begin{equation*}
		\int_{x \in \redv^{-1}(\bar{x})} H(x;\mathbf{s})^{-1}  \> d \mathbf{g}_v = \frac{1}{q_v^{\Dim(X)}} \prod_{\beta\in \mathcal{B}_{\bar{x}}} q_v^{\beta(\bar{x})(\rho_{\beta}-s_{\beta})}  \prod_{{\alpha_v}\in \mathcal{A}_{\bar{x}}} \frac{q_v-1}{q_v^{1+s_{\alpha}-\rho_{\alpha}}-1}.
	\end{equation*}
\end{lem}

\begin{proof}
	By assumption on $v$, the divisor $\cup_{{\alpha_v}\in \mathcal{A}_{\bar{x}}} \mathcal{D}_{{\alpha_v}}$ is relative strict normal crossings over $\ov$. Also, for all ${\alpha_v} \in \mathcal{A}_{\bar{x}}$, as $\mathcal{D}_{{\alpha_v}}(\fv)\neq \emptyset$ and $\mathcal{D}_{{\alpha_v}}$ is smooth, we have that ${D}_{{\alpha_v}}(F_v)\neq \emptyset$ by Hensel's lemma; in particular, this implies that $D_{{\alpha_v}}$ is geometrically integral over $F_v$. By following standard arguments in Arakelov geometry (see, e.g., \cite[Theorem 2.13]{Sal}, \cite[Theorem 3.1]{Denef}, \cite[Theorem 9.1]{CLYT}, or \cite[\S 6.2]{Campana} and their proofs) there exists a local set of coordinates $\{x_{\alpha_v}\}_{{\alpha_v}\in A_{\bar{x}}} \cup \{x_i\}_{i\in I}$, where $I=1,..., n-|\mathcal{A}_{\bar{x}}|$, that define an étale map of $\ov$-schemes $$\phi_{\tilde{x}}:=\left(\prod_{\mathcal{A}_{\bar{x}}} x_{\alpha_v}, \prod_{i\in I} x_i \right):  \mathcal{W} \rightarrow \mathbb{A}^n_{\ov}$$ for some open subscheme $\mathcal{W}$ of $\mathcal{X}$ containing $\bar{x}$, which restricts to an analytic map $\redv^{-1}(\bar{x})\rightarrow \prod^n {F_v}$ such that 
	\begin{enumerate}[label=(\roman*)]
		\item $\redv^{-1}(\bar{x})$ is mapped isomorphically to $\prod^n \mathfrak{m}_v$;
		\item $D_{{\alpha_v}}(F_v)$ is locally defined by $x_{\alpha_v}=0$ for ${\alpha_v} \in \mathcal{A}_{\bar{x}}$;
		\item $\redv^{-1}(\bar{x}) \cap D_{{\alpha_v}}(F_v)=\emptyset$ for all ${\alpha_v} \notin \mathcal{A}_{\bar{x}}$;
		\item and the canonical measure $\mu_v^n$ on $F^n$ pulls back to the measure $d\tau_{X,v}$.
	\end{enumerate}
	Note that for all $x\in \redv^{-1}(\bar{x})$, 
	\begin{enumerate}
		\item $\lVert \mathsf{s}_\alpha \rVert_v(x)=|x_{\alpha_v}|_v$ for the unique ${\alpha_v} \in \mathcal{A}_{\bar{x}}$ lying above $\alpha \in \mathcal{A}$;  
		\item $\lVert \mathsf{s}_\beta \rVert_v(x)=q_v^{\beta(\bar{x})}$ for all $\beta \in  \mathcal{B}_{\bar{x}}$;
		\item $\lVert \mathsf{s}_\alpha \rVert_v(x)=1$ for all other $\alpha\in \mathcal{A}$.
	\end{enumerate}
	
	Therefore, we have
	\begin{equation*}
		\begin{split}
			\int_{\redv^{-1}(\bar{x})} H(\mathbf{x};\mathbf{s})^{-1}  \> d
			\mathbf{g}_v 
			&= \int_{\redv^{-1}(\bar{x})}   \prod_{\alpha \in \mathcal{A}} \lVert \mathsf{s}_\alpha \rVert ^{\rho_\alpha-s_\alpha} d\tau_{X,v}
			\\&= \int_{\prod^n \mathfrak{m}_v}  \prod_{{\alpha_v}\in \mathcal{A}_{\bar{x}}} |x_{\alpha_v}|^{\rho_{\alpha}-s_{\alpha}} \prod_{\beta\in \mathcal{B}_{\bar{x}}} q_v^{\beta(\bar{x})(\rho_\beta-s_\beta)} \prod_{{\alpha_v} \in \mathcal{A}_{\bar{x}}} dx_{\alpha_v} \prod_{i\in I} dx_i
			\\ &=  \frac{1}{q_v^{n-|\mathcal{A}_{\bar{x}}|}} \prod_{\beta\in \mathcal{B}_{\bar{x}}} q_v^{\beta(\bar{x})(\rho_\beta-s_\beta)}  \prod_{{\alpha_v}\in \mathcal{A}_{\bar{x}}} \int_{\mathfrak{m}_v} q_v^{(s_{\alpha}-\rho_{\alpha})v(x_{\alpha_v})}  dx_{\alpha_v} 
			\\ &=  \frac{1}{q_v^{\Dim(X)}} \prod_{\beta\in \mathcal{B}_{\bar{x}}} q_v^{\beta(\bar{x})(\rho_\beta-s_\beta)}  \prod_{{\alpha_v}\in \mathcal{A}_{\bar{x}}} \frac{q_v-1}{q_v^{1+s_{\alpha}-\rho_{\alpha}}-1}
		\end{split}
	\end{equation*}
	where the last equality follows from the formula
	$$\int_{x\in\mathfrak{m}_v} q^{-sv(x)} d\mu_v = \frac{1}{q} \frac{q-1}{q^{1+s}-1}$$
	for $s\in \cx$ such that $\re(s)>-1$ 
	(see \cite[Equation (9.3)]{CLYT}).
\end{proof}

For $A\subset \mathcal{A}_v$ we define 
$$\mathcal{D}^\circ_A(\fv):=\bigcap_{{\alpha_v}\in A} \mathcal{D}_{{\alpha_v}} (\fv) \setminus \bigcup_{A\subsetneq A'\subset \mathcal{A}_v} \bigcap_{{\alpha_v}'\in A'} \mathcal{D}_{{\alpha_v}'}(\fv) .$$
If $A=\emptyset$, then we set $\mathcal{D}^\circ_A$ to be $\mathcal{X}\setminus \mathcal{D}$. Note that if $\bar{x}\in  \mathcal{D}_{{\alpha_v}} (\fv)$ for some $\alpha_v\in \mathcal{A}_v$, then $ \mathcal{D}_{{\alpha_v}} (\fv)$ is the only divisor lying above $\mathcal{D}_{{\alpha}}$, for some $\alpha \in \mathcal{A}$, that contains $\bar{x}$; in this case, we will simply denote $\alpha_v$ by $\alpha$. Therefore, by definition, if $\mathcal{D}^\circ_A(\fv)\neq \emptyset$ then for every $\alpha\in \mathcal{A}$, there is at most one $\alpha_v\in A$ lying above $\alpha$; hence, in this case, we are allowed to denote the elements of $A$ by $\alpha$

We now prove a geometric version of Denef's formula for Igusa's local zeta function over global functions fields analogous to \cite[Theorem 9.1]{CLYT}. In this version, we deal with the irreducible components of $D$ that are not geometrically reduced which do appear in positive characteristic. 

\begin{thm}\label{thm:Denef}
	For all $v\notin S$ and all $\mathbf{s}\in \Omega_{-1} \subset \pic(X)_\cx$, we have
	\begin{equation}\label{eqn:denef}
		 \hat{H}_v(\mathbf{0}; \mathbf{s})= q_v^{-\Dim(X)} \sum_{A\subset \mathcal{A}_v} \prod_{{\alpha}\in A} \frac{q_v-1}{q_v^{1+s_{\alpha}-\rho_{\alpha}} -1} \sum_{\bar{x}\in \mathcal{D}^\circ_A(\fv)} \prod_{\beta\in \mathcal{B}\cap A} q_v^{\beta(\bar{x})(\rho_\beta-s_\beta)}.
	\end{equation}

\end{thm} 
\begin{proof}
	As $X(F_v)=\mathcal{X}(\ov)$ and the measure of $X(F_v)\setminus G(F_v)$ is zero with respect to the measure $\tau_{X,v}$, we have
	\begin{equation}\label{eqn:IntOverResidueClasses}
		\begin{split}
		 	\hat{H}_v(\Psi_{0}; \mathbf{s})&=\int_{\mathbf{x}\in G(F_v)} H(\mathbf{x};\mathbf{s})^{-1}  \> d \mathbf{g}_v  \\ &= \sum_{\tilde{x}\in \mathcal{X}(\fv)}\int_{\mathbf{x}\in \redv^{-1}(\tilde{x})} H(\mathbf{x};\mathbf{s})^{-1}  \> d \mathbf{g}_v 
		 	\\ &= \sum_{\tilde{x}\in \mathcal{X}(\fv)} \frac{1}{q_v^{\Dim(X)}}  \prod_{\beta\in \mathcal{B}_{\bar{x}}} q_v^{\beta(\bar{x})(\rho_\beta-s_\beta)}  \prod_{{\alpha}\in \mathcal{A}_{\bar{x}}} \frac{q_v-1}{q_v^{1+s_{\alpha}-\rho_{\alpha}}-1}
		\end{split}
	\end{equation}
	where the last equality follows from \Cref{lem:local-coor-good-v}. The formula holds by an appropriate rearrangement of this final sum.
\end{proof}
\begin{rem}
	Note that if the boundary divisor is geometrically reduced, i.e. $\mathcal{B}=\emptyset$, then $$\sum_{\bar{x}\in \mathcal{D}^\circ_A(\fv)} \prod_{\beta\in \mathcal{B}\cap A} q_v^{\beta(\bar{x})(\rho_\beta-s_\beta)}= \#\mathcal{D}^\circ_A(\fv)$$
	in \Cref{thm:Denef}. In this case, one obtains the same formula in \cite[Theorem 9.1]{CLYT}.
\end{rem}

We now provide an estimate for $\hat{H}_v(\Psi_{0}; \mathbf{s})$ for the good places $v$ using the Lang-Weil estimates. For $A\subset \mathcal{A}_v$ we define $\mathcal{D}_A:=\cap_{\alpha_v\in A}  \mathcal{D}_{\alpha_v}$.

\begin{lem}\label{lem:Lang-Weil}
	There exists a constant $C(X)$, which only depends on $X$, such that for all $v\notin S$ and all $A\subset\mathcal{A}_v$ we have the following estimates:
	\begin{enumerate}
		\item If $\# A= 1$ and $A\cap \mathcal{B}=\emptyset$, then either $ \mathcal{D}_A(\fv) =\emptyset$, or  
		$$|\# \mathcal{D}_A(\fv)-q_v^{\Dim X -1}| \leq C(X) q_v^{\Dim X -3/2};$$
		\item If $\# A= 1$ and $A\subset \mathcal{B}$, then 
		$$|\# \mathcal{D}_A(\fv)-q_v^{\Dim X -1}| \leq C(X) q_v^{\Dim X -3/2};$$
		\item If $\# A\geq 1$, then $\# \mathcal{D}_A(\fv) \leq C(X) q_v^{\Dim X -\# A}$.
	\end{enumerate}
\end{lem}
\begin{proof}
	Suppose that $\# A=1$ and $A\cap \mathcal{B}=\emptyset$. If $ \mathcal{D}_A(\fv) \neq \emptyset$, then $\mathcal{D}_{A}\otimes \fv$ is geometrically integral. If $A:=\{\beta\}\subset \mathcal{B}$, then ${\mathcal{D}}_{v,\beta}:=\mathcal{D}_{v,\beta}\otimes \fv$ is a geometrically irreducible scheme that is not reduced. Thus the reduced structure $\overline{\mathcal{D}}^{\text{red}}_{v,\beta}\otimes \fv$ is a geometrically integral variety of dimension $n$, such that $\overline{\mathcal{D}}_{v,\beta}^{\text{red}}(\fv)=\overline{\mathcal{D}}_{v,\beta}(\fv)$.
	The result now follows by an application of the Lang-Weil estimates \cite[Theorem 1]{LangWeil}.
\end{proof}

We now give an estimate for the local Fourier transform at the trivial character. We follow a similar approach to \cite[Theorem 9.1]{CLYT} and \cite[\S 7]{Campana}. Recall from \Cref{assum}(iii) that when $\mathcal{B}\neq \emptyset$ we have defined the value
$$e_X:=\text{sup}_{\beta\in \mathcal{B}}\{m\in \z_{>0}:\>  \lVert \mathsf{s}_\beta \rVert_v(x)= q_v^{-m}, \text{ for some } v\in\Omega_F \text{ and } x\in X(F_v)\},$$ 
and when $\mathcal{B}= \emptyset$ we set $e_X=1$. By \Cref{assum}(iii), $e_X$ is finite.
\begin{prop}\label{prop:H_v-estimate}
	For all $\epsilon>0$ sufficiently small, there exists a constant $C(\epsilon)$, depending on $\epsilon$, such that for any $\mathbf{s}\in \Omega_{-1/(2e_X)+\epsilon}$  and any place $v\notin S$, 
	$$ \left| \hat{H}_v(\mathbf{0};\mathbf{s})\prod_{\alpha\in \mathcal{A}\setminus \mathcal{B}} \prod_{\{{\alpha_v} : \> \alpha_v\mid \alpha\}} (1-q_v^{-f_{\alpha_v}(1+s_\alpha-\rho_\alpha)}) \prod_{\beta \in \mathcal{B}} (1-q_v^{-(1+s_\alpha-\rho_\alpha)}) -1 \right| \leq C(\epsilon) q_v^{-1-\epsilon}$$
\end{prop}
\begin{proof}
	We study the right hand side of formula \eqref{eqn:denef}, by estimating the contributions of every possible $A\subset \mathcal{A}_v$.
	\begin{itemize} [itemsep=4pt,parsep=1pt,topsep=4pt]
		\item If $A=\emptyset$, then $\mathcal{D}_A^\circ(\fv) = (\mathcal{G}\otimes \fv)(\fv)$. As $\fv$ is a perfect field, and $\mathcal{G}\otimes \fv$ is a form of $\mathbb{G}_{a,\fv}^n$, we deduce that $\mathcal{G}\otimes \fv\cong \mathbb{G}_{a,\fv}^n$ since all forms of $\ga$ over a perfect field are trivial; this implies that $(\mathcal{G}\otimes \fv)(\fv)=q_v^n$. Therefore, this contributes to $1$ a the right hand side of \eqref{eqn:denef}.
		\item If $A=\{{\alpha_v}\}$,  ${\alpha_v} \notin \mathcal{B}$, and  $\mathcal{D}_{{\alpha_v}}(\fv)\neq \emptyset$, then by \Cref{lem:Lang-Weil} the contribution we get is
		$$q_v^{-(1+s_{\alpha} -\rho_{\alpha})}+O_X(q_v^{-3/2+1/(2e_X)-\epsilon}),$$
		when  $\re(s_\alpha)-\rho_\alpha>-1/(2e_X)+\epsilon$, where $\alpha\in \mathcal{A}$ is the index lying below $\alpha_v$. Note that as $\mathcal{D}_{{\alpha_v}}(\fv)\neq \emptyset$ and $D_{\alpha_v}$ is smooth, we have that $D_{\alpha_v}$ is geometrically integral so that $f_{\alpha_v}=1$.
		\item If $A=\{\beta\}\subset \mathcal{B}$, then by \Cref{assum}(ii),
		$$\#\{\bar{x}\in D_\beta (\fv): \> \lVert \mathsf{s}_\beta \rVert_v(\bar{x})=q_v^{-1} \}=q_v^{\dim(X)-1}+O(q_v^{\dim(X)-3/2}),$$
		and $\lVert \mathsf{s}_\beta \rVert_v(\bar{x})\leq q_v^{-e_X}$, for all $x\in X(F_v)$ and $\beta \in \mathcal{B}$.
		Therefore, by \Cref{lem:Lang-Weil}, when $\re(s_\beta)-\rho_\beta>-1/(2e_X)+\epsilon$ the contribution we get is
		$$q_v^{-(1+s_{\beta} -\rho_{\beta})}+O_X(q_v^{-1-e_X \epsilon}).$$
		
		\item Finally, if $\# A >2$, then by \Cref{lem:Lang-Weil} we get a contribution of
		$$O_X(q_v^{-3/2-e_X \epsilon}).$$
	\end{itemize}
	Therefore, we conclude that when $\re(s_\alpha)-\rho_\alpha>-1/(2e_X)+\epsilon$ for all $\alpha\in \mathcal{A}$, we have
	\begin{equation*}
		\begin{split}
			\hat{H}_v(\Psi_0;\mathbf{s}) &= 1+ \sum_{\alpha\in \mathcal{A}\setminus \mathcal{B}} \sum_{\{{\alpha_v} : \> \alpha_v\mid \alpha,\> f_{\alpha_v}=1\}} q_v^{-(1+s_\alpha-\rho_\alpha)} + \sum_{\beta\in \mathcal{B}} q_v^{-(1+s_\beta-\rho_\beta)}+O(q_v^{-1-\epsilon})
			\\ &= 1+ \sum_{\alpha\in \mathcal{A}\setminus \mathcal{B}} \sum_{\{{\alpha_v} : \> \alpha_v\mid \alpha\}} q_v^{-f_{\alpha_v}(1+s_\alpha-\rho_\alpha)} + \sum_{\beta\in \mathcal{B}} q_v^{-(1+s_\beta-\rho_\beta)}+O(q_v^{-1-\epsilon}),
		\end{split}
	\end{equation*}
	using the fact that for all $\alpha_v$ with $f_{\alpha_v}=[F_{{\alpha_v}}: F_v]>1$ we can absorb the contribution to the error term. The result now follows.
\end{proof}

For $v\in \Omega_F$, let
$\zeta_{F,v}(s)=(1-q_v^{-s})^{-1}$ be the local zeta function of $F$. We define
the zeta function of $F$ as an Euler product
$$\zeta_F(s):= \prod_{v\in \Omega_F} \zeta_{F,v}(s) = \prod_{v\in \Omega_F} (1-q_v^{-s})^{-1},$$
which converges absolutely for $\re(s)>1$ and has an analytic continuation to the whole of $\cx$; the poles of $\zeta_F(s)$ are $s=1+2m\pi i/\log(q)$ and $s=2m\pi i/\log(q)$ for $m\in \z$ (see \cite[Theorem 5.9]{Rosen}). Finally, by the Riemann hypothesis for function fields, $\zeta_F(s)$ is holomorphic and non-vanishing for $1/2<\re(s)<1$ (e.g., see \cite[Theorem 5.10]{Rosen}). Note that $\zeta_F(s)$ is periodic with respect to the imaginary axis, with period ${2\pi i}/{\log(q)}$. 

Define $F_\alpha$ to be the separable closure of $F$ in the function field of $D_\alpha$ for $\alpha \in \mathcal{A}$. Note that for $\beta \in \mathcal{B}$, we have $F_\beta=F$ as $D_\beta$ is geometrically irreducible.

\begin{cor}\label{cor:FT-exression}
	The function 
	$$\phi (\mathbf{s}):=\hat{H}(\mathbf{0};\mathbf{s})\prod_{\alpha \in \mathcal{A}} \zeta_{F_\alpha}(1+s_\alpha-\rho_\alpha )^{-1}$$
	is holomorphic on $\Omega_{-1/(2e_X)}$. 
\end{cor}

\begin{proof}
	First, since $F_\alpha \otimes_F F_v \simeq \prod_{\{\alpha_v : \> \alpha_v \mid \alpha\}} F_{{\alpha_v}}$
	for $v\in \Omega_F$, we see that
	$$\zeta_{F_\alpha,v}(s):= \prod_{\{ w\in \Omega_{F_\alpha}:{w\mid v}\}} \frac{1}{1-q_w^{-s}} = \prod_{\{\alpha_v : \> \alpha_v \mid \alpha\}} \frac{1}{1-q_v^{-f_{\alpha_v} s}}.$$
	Therefore,  
	$\zeta_{F_\alpha}(s)= \prod_{v\in \Omega_{F}} \zeta_{F_\alpha,v}(s).$
	For $v\in \Omega_F$, we can now rewrite the product
	$$\prod_{\alpha\in \mathcal{A}\setminus \mathcal{B}} \left( \prod_{\{\alpha_v : \> \alpha_v \mid \alpha\}} (1-q_v^{-f_{\alpha_v}(1+s_\alpha-\rho_\alpha)}) \right) \prod_{\beta \in \mathcal{B}} (1-q_v^{-(1+s_\beta-\rho_\beta)}),$$
	appearing in the statement of \Cref{prop:H_v-estimate}, as
	$$\prod_{\alpha\in \mathcal{A}\setminus \mathcal{B}} \zeta_{F_\alpha,v}(1+s_\alpha-\rho_\alpha )^{-1} \prod_{\beta \in \mathcal{B}} \zeta_{F,v}(1+  s_\beta-\rho_\beta)^{-1}= \prod_{\alpha\in \mathcal{A}} \zeta_{F_\alpha,v}(1+s_\alpha-\rho_\alpha )^{-1}$$
	(the equality holds as $F_\beta=F$ for all $\beta\in \mathcal{B}$.)

	For any place $v$ of $F$, we define for $\mathbf{s}\in \Omega_{0}$ the function
	$$\phi_v(\mathbf{s}):= \hat{H}_v(\mathbf{0};\mathbf{s}) \prod_{\alpha\in \mathcal{A}} \zeta_{F_\alpha,v}(1+s_\alpha-\rho_\alpha )^{-1}. $$
	By \Cref{prop:H_v-estimate}, we deduce that the Euler product
	$\prod_v \phi_v(\mathbf{s})$ converges absolutely to a holomorphic function $\phi$ on $\Omega_{-1/(2e_X)}$. Thus, for any $\mathbf{s}\in \Omega_{0}$, we have
	\begin{equation*}
		\begin{split}
			\hat{H}(\mathbf{0};\mathbf{s})&=\prod_v \phi_v(\mathbf{s})  \prod_{\alpha\in \mathcal{A}} \zeta_{F_\alpha,v}(1+s_\alpha-\rho_\alpha )
									\\&= \phi(\mathbf{s})  \prod_{\alpha \in \mathcal{A}} \zeta_{F_\alpha}(1+s_\alpha-\rho_\alpha )
		\end{split}
	\end{equation*}
	By the uniqueness of the analytic continuation, we obtain our result.
\end{proof}

Let $\mathcal{L}$ be a big line bundle on $X$ and let $\lambda=(\lambda_\alpha)\in \bigoplus_\alpha \cx [D_\alpha]$ be the corresponding vector in $\pic_\cx (X)$. By \Cref{prop:picX}, we can write $$\lambda_\alpha=\lambda'_\alpha/p'_\alpha \text{ where } \lambda'_\alpha,p'_\alpha\in \z_{\geq 0} \text{ are coprime} \text{ and } p'_\alpha=p^{k_\alpha} \text{ for some } k_\alpha\in \z_{>0} .$$ 
We define $a_\lambda:=\max\{\rho_\alpha/\lambda_\alpha\}$, $\mathcal{A}_\lambda :=  \{\alpha \in \mathcal{A} \> | \> \rho_\alpha =a_\lambda \lambda_\alpha\}$, $b_\lambda:=\# \mathcal{A}_\lambda$,  and
$$ \quad g_\lambda:=\frac{ \gcd(\{\lambda'_\alpha \}_{\alpha\in \mathcal{A}})}{\lcm(\{p'_\alpha \}_{\alpha\in \mathcal{A}})},    \quad d_\lambda:= \frac{ \gcd(\{\lambda'_\alpha\}_{\alpha\in \mathcal{A}_\lambda})}{\lcm(\{p'_\alpha\}_{\alpha\in \mathcal{A}})}.$$
Define the domain $$\Omega:=\left\{ (s_\alpha) \in \pic(X)_\cx \> : \>   0 \leq \text{Im}(s_\alpha)\leq \frac{2\pi i}{\log(q)}, \text{ for all } \alpha \right\}.$$

\begin{cor}  \label{cor:lambda-pole-order}
	Let $\lambda=(\lambda_\alpha)\in \bigoplus_\alpha \cx [D_\alpha] \cong \pic(X)_\cx$ be the vector corresponding to a big line bundle $\mathcal{L}_\lambda$ and let $H_\lambda(\cdot)= H(\cdot;\lambda)$ be the associated height function. Then
	$$ \hat{H}(\mathbf{0};s\lambda) = \int_{\mathbf{x}\in G(\ad)} H_\lambda(\mathbf{x})^{-s}  \> d \mathbf{g} $$
	converges absolutely for $\{s\in \cx \> :\>  \re(s)>a_\lambda\}$ and has a meromorphic continuation to $\re(s)>a_\lambda -\delta$ for some $\delta>0$. When $s$ is restricted to $\Omega$, the poles of largest real value are
	$$s_{m_\alpha} = a_\lambda + m_\alpha \frac{2\pi i}{\lambda_\alpha \log(q)}, \quad \>\alpha \in \mathcal{A}_\lambda, \> m_\alpha \in \{0,\ldots,\lceil \lambda_\alpha \rceil -1 \}.$$
	Moreover, the largest pole order is $b_\lambda$, and a pole has order $b_\lambda$ if and only if it is of the form $$s_j = a_\lambda + j \frac{2\pi i}{d_\lambda \log(q)},$$
	for $j\in J_\lambda:=\{0,\ldots, \lceil d_\lambda \rceil -1\}$. 
\end{cor}
\begin{proof}
	Note that $\zeta_{F_\alpha}(1+\lambda_\alpha s -\rho_\alpha)$ has simple poles of largest real value at $$s=\frac{\rho_\alpha}{\lambda_\alpha}+\frac{2\pi i n_\alpha}{\lambda_\alpha \log(q)}, \quad n_\alpha\in \z_{>0}.$$ 
	By \Cref{cor:FT-exression}, for all $\alpha$, each of these are poles of $\hat{H}(\mathbf{0};s\lambda)$ as well. By definition, the poles of largest real value have $\re(s)=a_\lambda$ which come from $\zeta_{F_\alpha}$ for $\alpha \in \mathcal{A}_\lambda$; for these poles to have order $b_\lambda$, they have to be poles of $\zeta_{F_\alpha}(1-\lambda_\alpha s -\rho_\alpha)$, for all $\alpha \in \mathcal{A}_\lambda$, which happens if and only of they are of the form
	$$s = a_\lambda + j \frac{2\pi i}{d_\lambda \log(q)},$$
	where $j\in \{0,\ldots, \lceil d_\lambda \rceil -1\}$.
\end{proof}

\section{Fourier transforms at the non-trivial characters}\label{sec:FT-NonTrivialChar}

\subsection{The Fourier transform at other characters: split case}
We follow the same notation and assumptions as in  \S\ref{sec:FT-TrivialChar}. In this section, we estimate the local Fourier transforms at the non-trivial characters of $G(\ad)$ orthogonal to $G(F)$. Recall that by \Cref{cor:ex.non-van.diff}, $G$ is a closed codimension $1$ subgroup of $\mathbb{G}_a^{n+1}$, where $n=\dim(G)$. This induces a surjective map $$(\mathbb{G}_a^{n+1}(\ad)/\mathbb{G}_a^{n+1}(F))^\wedge \rightarrow (G(\ad)/G(F))^\wedge$$ given by restricting a character: it maps $\Psi_a:= \psi(f_\mathbf{a}(\cdot))$, for $\mathbf{a}\in \mathbb{G}_a^{n+1}(F)$, to the character $\Psi_\mathbf{a}\big|_G:= \psi(f_\mathbf{a}|_G(\cdot))$, where $f_\mathbf{a}\big|_G\in F(G)$ is the restriction of the regular function $f_\mathbf{a} \in F(\mathbb{G}_a^{n+1})$. We can write  $\Psi_\mathbf{a}\big|_G=\prod_v (\Psi_{v,\mathbf{a}}\big|_G$), where $\Psi_{v,\mathbf{a}}\big|_G:= \psi_v(f_\mathbf{a}|_G(\cdot))$ (see \eqref{eqn:localCharacter}). We abuse notation and keep denoting $\Psi_{\mathbf{a}}\big|_G,\> \Psi_{v,\mathbf{a}}\big|_G$ and $f_\mathbf{a}|_G$ by $\Psi_{\mathbf{a}},\> \Psi_{v,\mathbf{a}}$ and $f_\mathbf{a}$, respectively.
We can view $f_\mathbf{a}$ as an element of $F(X)$ or $F(\mathcal{X})$ and we let $\Div(f_\mathbf{a})$ be its divisor. 

Let $\mathbf{a}\in G(F)$. Let $S_\mathbf{a}$ be a finite set of places of $F$ such that $\Div(f_\mathbf{a}) \text{ is flat over } \mathcal{C}_{F,S_\mathbf{a}}$. We define the following finite set of bad places
$$S(\mathbf{a}):=S\cup S_\mathbf{a} \cup \{v\in \Omega_F: \Psi_{\mathbf{a}} \text{ is not trivial on } G(\ov) \}.$$ 
We can write $\Div(f_a)=E-\sum_\alpha d_\alpha D_\alpha$, where $E$ is the unique irreducible component of $\{f_\mathbf{a}=0\}$ that meets $G$ and each $d_\alpha\geq 0$. Let $\mathcal{E}$ be the closure of $E$ in $\mathcal{X}$. Define the following subset of $\mathcal{A}$:
$$\mathcal{A}_0(\mathbf{a})=\{ \alpha \in \mathcal{A}: \> d_\alpha(f_\mathbf{a})=0\}.$$
Our goal is to estimate 
$$	\hat{H}(\Psi_\mathbf{a};\mathbf{s})=\prod_{v\in S(\mathbf{a})} \hat{H}_v(\Psi_{v,\mathbf{a}};\mathbf{s}) \prod_{v\notin S(\mathbf{a})} \hat{H}_v(\Psi_{v,\mathbf{a}};\mathbf{s}).$$
By \Cref{lem:trivial-bound-localTransform}, we can bound the finite product $\prod_{v\in S(\mathbf{a})} \hat{H}_v(\Psi_{v,\mathbf{0}};\mathbf{s})$ by some constant. In what follows, we estimate 
$\hat{H}_v(\Psi_{v,\mathbf{0}};\mathbf{s})$ for $v\notin S(\mathbf{a})$. We start by proving an analogue of \cite[Lemma 10.3]{CLYT} for local fields of positive characteristic, where we follow a similar proof. Recall from \eqref{eqn:varphiChar} the character $\varphi_v\in \widehat{F_v}$.
 
\begin{lem}\label{lem:int-char-units}
	Let $d\geq 0$ and $n\geq 1$ be integers. Let $u\in \ov^*$. Suppose that $p\nmid d$ and that $v$ does not lie above the place of $F_0$ corresponding to $1/t$. Then
	\begin{equation*}
		\int_{\ov^*} \varphi_v(u \pi_v^{-nd} w^d) dw =  
		\left\{
		\begin{array}{ll}
			1-1/q_v & \text{ if } d=0; \\
			-1/q_v & \text{ if } n=d=1; \\
			0 & \text{ else.}  \\
		\end{array} 
		\right.
	\end{equation*} 
\end{lem}
\begin{proof}
	If $d=1$, then 
	\begin{equation*}
		\begin{split}
			\int_{\ov^*} \varphi_v(u \pi_v^{-n} w) dw &=  \int_{\ov} \varphi_v( u \pi_v^{-n} w) dw - \frac{1}{q_v} \int_{\ov} \varphi_v(u\pi_v^{-n+1} w) dw
			=
			\left\{
			\begin{array}{ll}
				0 & \text{ if } n\geq 2; \\
				-1/q_v & \text{ if } n=1. \\
			\end{array} 
			\right.
		\end{split}
	\end{equation*}
	
	Suppose now that $d\geq 2$ and $p\nmid d$. We integrate over disks $D(a, \pi_v^e) \subset \ov^*$ for a suitable choice of $e\geq 1$. Let $w=a+\pi_v^e z$ for some $z\in \ov$. Then 
	$$w^d = a^d + \tilde{d}\pi_v^e a^{d-1} z \mod \pi_v ^{2e}$$
	where $\tilde{d}=\sum_{i=1}^d 1 \in \fv$, which is non-zero as $p\nmid d$. Therefore, if we chose $e$ such that 
	$$e-nd <0 \quad \text{and} \quad 2e-nd\geq 0, \text{ then}$$
	$$\int_{D(a,\pi_v ^e)} \varphi_v(u\pi_v^{-nd} w^d) dw = q_v^{-e} \varphi_v(u \pi_v^{-nd} a^d) \int_{\ov} \varphi_v( \tilde{d} u \pi_v^{e-nd} a^{d-1} z) dz =0.$$ 
	Any choice of an integer $e$ satisfying $nd/2 \leq e < nd$ works and such a choice exists as $nd \geq 2$.
\end{proof}
\begin{rem}\label{rem:LemmaHoldsForPSi}
	By \Cref{lem:psiIsVarphiUnit}, for almost all places $v$, we have that $\psi_v=\varphi_{v,c_v}$ for some $c_v\in \ov^*$. Therefore, for such places $v$ the conclusion of \Cref{lem:int-char-units} holds for $\psi_v$.
\end{rem}

The following is an analogue of \cite[Proposition 10.2]{CLYT}. We use the interpretation of the characters as Brauer group elements (see \S\ref{subsec:BrToChar}) to treat the case when $d_\alpha$ is a non-zero multiple of $p$. We will restrict to the case where $f_\mathbf{a}$ has no poles along the irreducible divisors that are not geometrically reduced, i.e., $d_\beta=0$ for all $\beta \in \mathcal{B}$. Note that if $G=\ga$, then $\mathcal{B}=\emptyset$, and this condition is always vacuously satisfied.  

\begin{prop}\label{prop:H_v-estimate-otherChars}
	Let $f_\mathbf{a}\in F(X)$ be the restriction of a linear form and write $\Div(f_a)=E-\sum_{\alpha\in \mathcal{A}} d_\alpha D_\alpha$. 
	Suppose that $d_\beta=0$ for all $\beta \in \mathcal{B}$. Then
	for all $\epsilon>0$ sufficiently small, there exists a constant $C(\epsilon)$, depending on $\epsilon$, such that for any $\mathbf{s}\in \Omega_{-1/(2e_X)+\epsilon}$  and any place $v\notin S(\mathbf{a})$, we have
	$$ \left| \hat{H}_v(\Psi_{v,\mathbf{a}};\mathbf{s})\prod_{\alpha\in \mathcal{A}_0(\mathbf{a})} \prod_{\{\alpha_v:\> \alpha_v \mid \alpha\}} (1-q_v^{-f_{\alpha_v}(1+s_\alpha-\rho_\alpha)}) -1 \right| \leq C(\epsilon) q_v^{-1-\epsilon} $$
\end{prop}
\begin{proof}
	We follow the proof strategy of \cite[Proposition 10.2]{CLYT} and \cite[\S 8.1.1]{Campana}. However, there is a new case to be considered when $p$ divides $d_\alpha$ for some $\alpha\in \mathcal{A}$. To overcome this we relate the character to an element of $\Br(G)$ using \S\ref{subsec:BrToChar}. 
	
	We proceed by computing the integrals over the different residue classes. Let $\tilde{x}\in \mathcal{X}(\fv)$ and $A=\{\alpha \in \mathcal{A} \> : \> \tilde{x}\in \mathcal{D}_\alpha\}.$
	\begin{itemize} [leftmargin=20pt,itemsep=5pt,parsep=1pt,topsep=4pt]
		\item If $A=\emptyset$, then $\redv(\tilde{x})^{-1} \subset G(\ov)$. As $\Psi_{v,\mathbf{a}}$ is trivial on $G(\ov)$, we get
		\begin{equation*}
			\begin{split}
				\int_{\mathbf{x}\in \redv^{-1}(\tilde{x})} H_v(\mathbf{x};\mathbf{s})^{-1} \Psi_{v,\mathbf{a}}(\mathbf{x})  \> d \mathbf{g}_v = \int_{\mathbf{x}\in \redv^{-1}(\tilde{x})}  \> d \mathbf{g}_v =q_v^{- \dim(X)}.
			\end{split}
		\end{equation*}
		As $G(\fv)\cong \ga(\fv) \cong \prod^n \fv$, this contributes to a $1$.
		
		\item Suppose that $A=\{\alpha\}$, $\tilde{x}\notin \mathcal{E}$. By \cite[Prop 10.2, case 2]{CLYT} or \cite[\S 8.1.1]{Campana}, we can introduce local analytic coordinates $x_\alpha$ and $x_1,\ldots , x_{n-1}$ such that $f_{\mathbf{a}}= u x_\alpha ^{-d_\alpha}$, for $u\in \ov^*$. We have two cases. 
		\begin{enumerate}[label=(\roman*)]
			\item Assume that $d_\alpha$ is coprime to $p$. Then
			\begin{equation*}
				\begin{split}
					\int_{\redv^{-1}(\tilde{x})} H_v(\mathbf{x};\mathbf{s})^{-1} \Psi_{v,\mathbf{a}}(\mathbf{x}) \> d \mathbf{g}_v &= \int_{\mathfrak{m}_v \times \mathfrak{m}_v ^{n-1}}  q_v^{-(s_\alpha-\rho_\alpha)v(x_\alpha)} \psi_{v}(u x_\alpha^{-d_\alpha})  \> d x_\alpha \prod_{i\in I} dx_i 
					\\&= \frac{1}{q_v^{n-1}} \sum_{n_\alpha \geq 1} q_v^{-(1+s_\alpha-\rho_\alpha)n_\alpha} \int_{\ov^*} \psi_v(u \pi_v^{-n_\alpha d_\alpha} t_\alpha^{-d_\alpha})  dt_\alpha
					\\&= \left\{
					\begin{array}{ll}
						\frac{q_v-1}{q_v^n} \frac{1}{q_v^{1+s_\alpha-\rho_\alpha}-1} & \text{ if } d_\alpha=0; \\
						-\frac{1}{q_v^n} q_v^{-(1+s_\alpha-\rho_\alpha)} & \text{ if } d_\alpha=1; \\
						0 & \text{ if } d_\alpha\geq 2;
					\end{array} 
					\right.
				\end{split}
			\end{equation*}
			where the last equality follows by \Cref{lem:int-char-units} and \Cref{rem:LemmaHoldsForPSi}. Thus, the contribution to the main term comes from $\alpha \in \mathcal{A}$ with $d_\alpha=0$.
			
			\item Assume that $d_\alpha=mp^k$ for $m,k\in \z_{>0}$ and $p\nmid m$. We will use the results in \S\ref{subsec:BrToChar} to deal with this case. 
			Note that by \Cref{lem:psiIsVarphiUnit}, we have $\psi_v=\varphi_{v,c_v}$ for some $c_v^*$. Let $u_0:= c_v u$. Then,
			by \Cref{prop:Ga-LocalCharBr},
			$$\Psi_\mathbf{v,a}(x)=\psi_v(u x_\alpha ^{-d_\alpha}) = \varphi_v(u_0 x_\alpha ^{-d_\alpha}) = \chi( [x_\alpha ^{-d_\alpha},b)_v ),$$
			for some $b\in F_v$.
			Now, by \Cref{lem:Br-xp-x}, 
			$$[x_\alpha^{-d_\alpha},b)=[x_\alpha^{-mp^k},b) = [x_\alpha^{-m},b) \in \Br(F_v(\mathbb{G}_a)).$$
			Therefore,
			as $m$ is non-zero and coprime to $p$, this reduces the computation to the previous case with $d_\alpha=m>0$. Hence, we get no contribution to the main term.
		\end{enumerate}

		\item If $|A|\geq 2$, or $|A|=1$ and $\tilde{x}\in \mathcal{E}$, then we don't get a contribution to the main term. This follows by \Cref{lem:Lang-Weil} and proceeding as in \Cref{prop:H_v-estimate}.
	\end{itemize}
	The result now follows by summing all the contributions using \Cref{lem:Lang-Weil}.
\end{proof}

\begin{cor}\label{cor:FT-otherChar-exression}
	Let $f_\mathbf{a}\in F(X)$ be the restriction of a linear form and write $\Div(f_a)=E-\sum_{\alpha\in \mathcal{A}} d_\alpha D_\alpha$. 
	Suppose that $d_\beta=0$ for all $\beta \in \mathcal{B}$.
	Then the function 
	$$\phi (\mathbf{s}):=\hat{H}(\Psi_\mathbf{a};\mathbf{s})\prod_{\alpha \in \mathcal{A}_0(\mathbf{a})} \zeta_{F_\alpha}(1+s_\alpha-\rho_\alpha )^{-1}$$
	is holomorphic on $\Omega_{-1/(2e_X)}$. 
\end{cor}
\begin{proof}
	The proof is analogous to \Cref{cor:FT-exression}.
\end{proof}

\section{Proofs of the main theorems}\label{sec:Asump-LeadCons}
In this section we prove \Cref{prop:HeightZetaPoles}. Subsequently, we will use an appropriate Tauberian theorem to conclude the proof of \Cref{thm:Manin}. Finally, we will compute the leading constant to prove \Cref{thm:Leading-constant}.

In \Cref{prop:HeightZetaPoles}, to ensure the existence of a meromorphic continuation of $Z_\lambda(s)$ to the left of the line $\{\re(s)=a_\lambda\}$, we restrict to the situation where
\Cref{cor:FT-otherChar-exression} can be applied to all the relevant characters. To that end, we ask that the following condition holds:
\begin{condition}\label{condition}
	For all non-trivial characters $\Psi \in (G(\ad)/(G(F)+\mathbb{K}))^\wedge$, there exists a representative $\Psi_a= \psi\circ f_\mathbf{a}$ of $\Psi$ such that if  $\Div(f_a)=E-\sum_{\alpha\in \mathcal{A}} d_\alpha D_\alpha$, then
	$d_\beta=0$ for all $\beta \in \mathcal{B}$.
\end{condition}
\begin{rem}
	Note that \Cref{condition} is satisfied vacuously when
	\begin{itemize}[itemsep=2pt,parsep=2pt,topsep=2pt]
		\item $\pic(G)$ is trivial (by \Cref{Cor:PicG-D-g.reduced});
		\item or $G(\ad)=G(F)+\mathbb{K}$. 
	\end{itemize}
\end{rem}

\subsection{Proof of \Cref{prop:HeightZetaPoles}}
Let $\lambda\in \pic(X)_\cx$ be the vector corresponding to the big line bundle $\mathcal{L}_\lambda$, and let $H_\lambda(\cdot)=H(\cdot; s(\lambda_\alpha))$ be the associated height function. Write
$Z_\lambda(s):=Z(s\lambda)$
for the height zeta function relative to $H_\lambda$, for $s\in \cx$.
By \Cref{cor:lambda-pole-order}, $\hat{H}(\mathbf{0};s\lambda)$ converges absolutely for $\{s\in \cx \> :\>  \re(s)>a_\lambda\}$ and has a meromorphic continuation to $\re(s)>a_\lambda -\delta$ for some $\delta>0$; moreover, the poles of $\hat{H}(\mathbf{0};s\lambda)$ in $\Omega$ of order $b_\lambda$ with largest real value are $$s_j = a_\lambda + j \frac{2\pi i}{d_\lambda \log(q)},\quad j\in J_\lambda:=\{0,\ldots, \lceil d_\lambda \rceil-1\}.$$
Now, by \Cref{lem:Finite-Quo}, there exists a set of representatives $\{y_1,\ldots, y_m\}$ for the cosets of $G(F)+ \mathbb{K}$ in $G(\ad)$. Recall that we are assuming \Cref{condition} is satisfied. By \Cref{cor:FT-otherChar-exression}, for all  $\Psi \in (G(\ad)/(G(F)+\mathbb{K}))^\wedge$, we have that $\hat{H}(\Psi;s\lambda)$ converges absolutely for $\{s\in \cx \> :\>  \re(s)>a_\lambda\}$ and has a meromorphic continuation to ${\re(s)>a_\lambda -\delta_0}$ for some $\delta_0>0$; moreover, the set of poles of $\hat{H}(\Psi;s\lambda)$ in $\Omega$ of order $b_\lambda$ with largest real value is a subset of $\{s_j\}_{j\in J_\lambda}$. Therefore, by the Poisson summation formula \eqref{eqn:Poisson} the series $$Z_{\lambda}(s)=\sum_{x\in G(F)} H_{\lambda}^{-s}(x)$$
converges absolutely and uniformly for $\re(s)>a_\lambda$, and has a meromorphic continuation to $\re(s)>a_\lambda-\delta$ for some $\delta>0$. It remains to show that $\{s_j\}_{j\in J_\lambda}$ is the set of poles of $Z_\lambda(s)$ of order $b_\lambda$ with largest real value.

By the Poisson summation formula \eqref{eqn:Poisson}, we have 
	\begin{align*}\label{eqn:zeta_char_orthog}
			q^{\dim(G)(g_F-1)}{\tau(G)}\cdot Z_{\lambda}(s)&= \sum_{\Psi \in (G(\ad)/(G(F)+\mathbb{K}))^\wedge}  \hat{H}(\Psi;s\lambda) \\&=\sum_{\Psi \in (G(\ad)/(G(F)+\mathbb{K}))^\wedge}  \int_{\mathbf{x}\in G(\ad)} H(\mathbf{x};s\lambda)^{-1} \Psi(\mathbf{x}) \> d \mathbf{g} 
			\\ &= \int_{\mathbf{x}\in G(\ad)} H(\mathbf{x};s\lambda)^{-1} \sum_{\Psi \in (G(\ad)/(G(F)+\mathbb{K}))^\wedge} \Psi(\mathbf{x}) \> d \mathbf{g} 
			\\&=\sum_{i=1}^m \underset{\mathbf{x}\in G(\ad)+\mathbb{K}}{\int} H(y_i+\mathbf{x};s\lambda)^{-1} \sum_{\Psi\in \left( \frac{G(\ad)}{G(F)+\mathbb{K}}\right)^\wedge}  \Psi (y_i+ \mathbf{x}) \> d \mathbf{g}
			\\&= \sum_{i=1}^m \left(\sum_{\Psi\in \left( \frac{G(\ad)}{G(F)+\mathbb{K}}\right)^\wedge}  \Psi(y_i) \right)\underset{\mathbf{x}\in G(\ad)+\mathbb{K}}{\int} H(y_i+\mathbf{x};s\lambda)^{-1}  \> d \mathbf{g}
			\\&= m \int_{\mathbf{x}\in G(F)+\mathbb{K}} H(\mathbf{x};s\lambda)^{-1}  \> d \mathbf{g}, 
	\end{align*}
	where the last equality follows by character orthogonality. By the proof of \cite[Lemma 5.2]{CLYT}, for a fixed $\lambda \in \pic(X)_\cx$, there exists $C_\lambda\in \mathbb{R}_{>0}$ such that for all $y_i$ and all $x\in G(\ad)$, we have
	\begin{equation}\label{eqn:height-bound-translation}
		C_\lambda^{-1} |H(x;s \lambda )|\leq |H(y_i + x;s \lambda )| \leq C_\lambda |H(x;s \lambda )|
	\end{equation}
	for all $s\in \cx$ as $\re(s) \rightarrow a_\lambda^+$. Note that we can write
	$$ \hat{H}(\mathbf{0};s\lambda)= \sum_{i=1}^m \int_{\mathbf{x}\in G(F)+\mathbb{K}} H(y_i+\mathbf{x};s\lambda)^{-1} \> d \mathbf{g} .$$
	Therefore, by \eqref{eqn:height-bound-translation}, we conclude that
	$$m C_\lambda^{-1}  {|Z_\lambda(s)|}  \leq \frac{1}{q^{\dim(G)(g-1)}{\tau(G)}} |\hat{H}(\mathbf{0};s\lambda)|   \leq m C_\lambda |Z_\lambda(s)| ,$$
	for all $s\in \cx$ as $\re(s) \rightarrow a_\lambda^+$.
	Hence, by \Cref{cor:lambda-pole-order}, we deduce that
	$$\lim_{s\rightarrow s_j} (s-s_j)^{b_\lambda} Z_\lambda(s) \neq 0,$$
	for all $j=0,\ldots \lceil d_\lambda \rceil-1$, as desired. \hfill \qed

\subsection{Proof of \Cref{thm:Manin}}
By \Cref{prop:HeightZetaPoles}, we have that $Z_\lambda(s)$ converges absolutely and is holomorphic in the region $\{s\in \Omega \>|\> \re(s)>a_\lambda-\delta\}$ except at the poles $\{s_j\}_{J_\lambda}$ of order $b_\lambda$, and possibly at some other poles of order less than $b_\lambda$. For all $j\in J_\lambda$, we define
$$r_j:= \lim_{s\rightarrow s_j} (s-s_j)^{b_\lambda} Z_\lambda(s).$$ 
By definition, $c_\lambda:=\lim_{s\rightarrow a_\lambda} (s-a_\lambda)^{b_\lambda} Z_\lambda(s) = r_0$.
\begin{prop} \label{thm:Tauberian}
	Let $\lambda=(\lambda_\alpha) \in \bigoplus_{\alpha\in \mathcal{A}}\z [D_\alpha]$. Then for all $M\in \z_{>0}$
	$$N(\lambda , M) = \frac{\log(q)^{b_\lambda}}{(b_\lambda-1)!} M^{b_\lambda-1} \sum_{j\in J_\lambda} q^{s_j M} r_j + o(q^{a_\lambda M} M^{b_\lambda-1}).$$
\end{prop}
\begin{proof}
	 As $\lambda=(\lambda_\alpha) \in \bigoplus_{\alpha\in \mathcal{A}}\z [D_\alpha]$, we see that $g_\lambda:=\gcd(\{\lambda_\alpha \}) \in \z$. By \Cref{lem:Period}, we deduce that $2\pi i/\log(q)$ is an imaginary period for $Z_\lambda(s)$. Therefore, by applying a straightforward generalisation of the Tauberian theorem \cite[Theorem 17.4]{Rosen} (see, e.g., \cite[Theorem 2.1]{Herrero}), the result follows.
\end{proof}

We apply \Cref{thm:Tauberian} to compute the weighted average counting function:
\begin{equation*}
	\begin{split}
		N_{\av} (\mathcal{L}_\lambda, M) &:= \frac{1}{d_\lambda} \sum_{k=0}^{d_\lambda-1}  q^{-a_\lambda k} N(\mathcal{L}_\lambda, M+k)
		\\&= \frac{\log(q)^{b_\lambda}}{(b_\lambda-1)!} M^{b_\lambda-1} \frac{1}{d_\lambda} \sum_{k=0}^{d_\lambda-1}  q^{-a_\lambda k} \sum_{j=0}^{d_\lambda-1} r_j q^{(a_\lambda + j \frac{2\pi i}{d_\lambda \log(q)}) (M+k)}  + o(q^{a_\lambda M} M^{b_\lambda-1})
		\\&= \frac{\log(q)^{b_\lambda}}{(b_\lambda-1)!} M^{b_\lambda-1} \frac{1}{d_\lambda} \sum_{j=0}^{d_\lambda-1} r_j  q^{(a_\lambda + j \frac{2\pi i}{d_\lambda \log(q)})M} \sum_{k=0}^{d_\lambda-1} e^{j k \frac{2\pi i}{d_\lambda}}  + o(q^{a_\lambda M} M^{b_\lambda-1})
		\\&= c_\lambda \frac{\log(q)^{b_\lambda}}{(b_\lambda-1)!} M^{b_\lambda-1}  q^{a_\lambda M} + o(q^{a_\lambda M} M^{b_\lambda-1}),
	\end{split}
\end{equation*}
where the last equality follows from the fact that $r_0=c_\lambda$, $\sum_{k=0}^{d_\lambda-1} e^{j k \frac{2\pi i}{d_\lambda}}=d_\lambda$ for $j=0$, and $\sum_{k=0}^{d_\lambda-1} e^{j k \frac{2\pi i}{d_\lambda}}=0$ for $j\in \{1,\ldots,d_\lambda-1\}.$ This proves the first statement of \Cref{thm:Manin}.

Suppose now that $d_\lambda \mid g_\lambda$. Then $\frac{2 \pi i}{d_\lambda \log(q)}$ is an imaginary period for $Z_\lambda(s)$, so that
$$\lim_{s\rightarrow s_j} (s-s_j)^{b_\lambda} Z_\lambda(s)= \lim_{z\rightarrow a_\lambda} (z-a_\lambda)^{b_\lambda} Z_\lambda(z+j\frac{2 \pi i}{d_\lambda \log(q)})= \lim_{z\rightarrow a_\lambda} (z-a_\lambda)^{b_\lambda} Z_\lambda(z) = c_\lambda,$$
i.e., $r_j=c_\lambda$ for all $j\in J_\lambda$. Therefore, by \Cref{thm:Tauberian},
\begin{equation}\label{eqn:N-count}
	N(\lambda , M) = \frac{\log(q)^{b_\lambda}}{(b_\lambda-1)!} M^{b_\lambda-1} c_\lambda q^{a_\lambda M} \sum_{j\in J_\lambda} q^{j\frac{2\pi i}{d_\lambda \log(q)} M} + o(q^{a_\lambda M} M^{b_\lambda-1})
\end{equation}
Suppose that $d_\lambda \mid M$ and set $M'=M/d_\lambda$. Then
$$\sum_{j=0}^{d_\lambda-1} q^{j \frac{2\pi i}{d_\lambda\log(q)} M}= \sum_{j=0}^{d_\lambda-1} e^{ (2\pi i) jM'}= d_\lambda,$$
which implies that
$$N(\lambda , M) = d_\lambda c_\lambda \frac{\log(q)^{b_\lambda}}{(b_\lambda-1)!} q^{a_\lambda M} M^{b_\lambda-1}+o(q^{a_\lambda M} M^{b_\lambda-1}).$$
This proves (2) in \Cref{thm:Manin}.

We now prove $(1)$.
Suppose that $d_\lambda \mid  g_\lambda$ and $d_\lambda \nmid M$. Recall from \eqref{eqn:defnHeight} the definition of the height function
$$H_\lambda(x)= \prod_v \prod_\alpha \lVert \mathsf{s}_\alpha \rVert_v (x)^{-\lambda_\alpha }.$$  
For a fixed place $v$, we can write $$\prod_\alpha \lVert \mathsf{s}_\alpha \rVert_v (x)^{-\lambda_\alpha } = q^{\sum_\alpha v_{x,\alpha} \lambda_\alpha} $$
for some integers $v_{x,\alpha}\in \z$. Since $d_\lambda\mid g_\lambda:=\gcd(\{\lambda_\alpha\}_{\lambda\in \mathcal{A}})$, one easily sees that if $H_\lambda(x)=q^M$ for some $x\in G(F)$, then $d_\lambda$ has to divide $M$. This shows that $N(\lambda,M)=0$ when $d_\lambda \nmid M$, which proves (1). \hfill \qed

\subsection{Proof \Cref{thm:Leading-constant}}
We compute the leading constant for $\rho=(\rho_\alpha)\in \pic(X)_\cx$, the vector corresponding to the isomorphism class of $\omega_X^{-1}$. Note that $a_\rho=1$ and $b_\rho=\rank(\pic(X))$.
We can rewrite the relation \eqref{eqn:dg=dtau} between the measures $\tau_{X,v}$ and $d\mathbf{g}_v$ as
$$d\tau_{X,v}= |c|_v^{-1} H_v(x_v;\rho)^{-1} \> d\mathbf{g}_v.$$ 
Therefore, as $\tau_{X,v}(D(F_v))=0$, we have 
\begin{equation}\label{eqn:TauXH}
	\tau_{X,v}(X(F_v))=\tau_{X,v}(G(F_v))=|c|_v^{-1} \int_{G(F_v)} H_v(x_v;\rho)^{-1} \> d\mathbf{g}_v = |c|_v^{-1} \hat{H}_v(\mathbf{0};\rho).
\end{equation}
Thus, by \Cref{cor:FT-exression}, the Euler product
$$\prod_v \left(\tau_{X,v}(X(F_v)) \prod_{\alpha\in \mathcal{A}} \zeta_{F_\alpha,v}(1)^{-1}\right) $$
converges absolutely to a strictly positive number.
\begin{defn} \label{defn:Peyre's-constant}
	The Tamagawa measure on $X(\ad)$ defined by Peyre \cite[Definition 2.6.1]{Peyre} is 
	$$\tau_X:=q^{-\dim(G)(g_F-1)}\prod_{\alpha\in \mathcal{A}} \residue_{s=1} \zeta_{F_\alpha}(s) \times \prod_v \left( \prod_{\alpha\in \mathcal{A}}  \zeta_{F_\alpha,v}(1)^{-1}\right) \cdot \tau_{X,v}.$$
\end{defn}
Note that $\Br(X)=\Br(F)$ by \Cref{prop:picX}, so that $X(\ad)^{\Br}=X(\ad)$.

\begin{prop}\label{prop:tamagawa}
	We have that
	$$\lim_{s\rightarrow 1^+} {(s-1)^{b_\rho}} \hat{H}(\mathbf{0};s\rho)=q^{\dim(G)(g-1)} \tau_X(X(\ad))  \prod_\alpha \rho_\alpha^{-1}.$$
\end{prop}
\begin{proof}
	By \Cref{cor:FT-exression}, we have that $\lim_{s\rightarrow 1^+} {(s-1)^{b_\rho}} \hat{H}(\mathbf{0};s\rho)$ is equal to
	\begin{align*}
			& \lim_{s\rightarrow 1^+} {(s-1)^{b_\rho}} \phi(s (\rho_\alpha)) \prod_{\alpha \in \mathcal{A}} \zeta_{F_\alpha}(1+\rho_\alpha(s-1))
			\\&= \left(\prod_\alpha \rho_\alpha^{-1}\right) \left( \lim_{s\rightarrow 1^+}  \prod_{\alpha\in \mathcal{A}} \rho_\alpha(s-1) \zeta_{F_\alpha}(1+\rho_\alpha(s-1))   \right) \left( \lim_{s\rightarrow 1^+} \prod_v \phi_v(s (\rho_\alpha)) \right) 
			\\&= \left(\prod_\alpha \rho_\alpha^{-1}\right) \left(   \prod_{\alpha\in \mathcal{A}} \residue_{s=1} \zeta_{F_\alpha}(s)  \right) \left( \prod_v  \hat{H}_v(0;\rho)  \prod_{\alpha\in \mathcal{A}} \zeta_{F_\alpha,v}(1)^{-1} \right)
			\\&= \left(\prod_\alpha \rho_\alpha^{-1}\right) \left(   \prod_{\alpha\in \mathcal{A}} \residue_{s=1} \zeta_{F_\alpha}(s)  \right) \cdot \left(  \prod_v |c|_v \tau_{X,v}(X(F_v))  \prod_{\alpha\in \mathcal{A}} \zeta_{F_\alpha,v}(1)^{-1} \right)
			\\&= \left(\prod_\alpha \rho_\alpha^{-1}\right) \tau_X(X(\ad)) \cdot q^{\dim(G)(g_F-1)},
	\end{align*}
	as desired.
\end{proof}

By \Cref{cor:lambda-pole-order}, \Cref{cor:FT-otherChar-exression}, and the Poisson summation formula \eqref{eqn:Poisson},  we deduce that
$$\lim_{s\rightarrow 1^+} {(s-1)^{b_\rho}} Z_\rho(s)= \lim_{s\rightarrow 1^+} {(s-1)^{b_\rho}} \frac{1}{q^{\dim(G)(g_F-1)}\tau(G)} \hat{H}(\mathbf{0}; \rho s).$$  
By \Cref{prop:tamagawa}, we get 
\begin{equation}
	\begin{split}
		c_\rho:=\lim_{s\rightarrow 1^+} (s-1)^{b_\rho} Z_{\rho}(s)
		= \frac{\tau_X(X(\ad))}{\tau(G)} \prod_\alpha \rho_\alpha^{-1}
	\end{split}
\end{equation}
Now, by \Cref{prop:picX}, we see that 
$$\alpha^*(X)= \frac{\prod_{\alpha \in \mathcal{A}} \rho_\alpha^{-1}}{|\pic(G)|}.$$
Also, by \cite[Theorem 1.1]{ZRos}, we can express the Tamagawa number of $G$ as
$$\tau(G)= \frac{\# \text{Ext}^1(G,\mathbb{G}_m)}{\# \Sh(G)}.$$
By \cite[Corollary 4.14]{Achet}
$$ \text{Ext}^1(G,\mathbb{G}_m)\cong \pic(G),$$
and by \Cref{thm:ShaTrivial} we have $\Sh(G)=0$.
Therefore, we can rewrite $\alpha^*(X)$ as
$$\alpha^*(X)= \frac{\prod_{\alpha \in \mathcal{A}} \rho_\alpha^{-1}}{\tau(G)}.$$
Finally, we conclude that
$c_\rho=\alpha^*(X) \tau_X(X(\ad)) \hfill\qed$

\section{Purely inseparable example} \label{section:Proj-example}
In this final section, we start by showing that \Cref{assum} holds in the setting of \Cref{thm:resGm}, which follows by \Cref{prop:P^p-1-valuations-insep} below. Subsequently, we show  that \Cref{condition} holds as well, which is a consequence of \Cref{prop:trivialOnF+K}. Finally, we compute the leading constant explicitly in \Cref{prop:tau(X(ad))}.
 
Let $F=\mathbf{F}_q(t)$ and $F^{1/p}=\mathbf{F}_q(t^{1/p})$. Let $G=\WeilRes_{F^{1/p}/F} \mathbb{G}_m / \mathbb{G}_m$, which is a non-trivial twist of $\mathbb{G}_a^{p-1}$ that splits over $F^{1/p}$ (see \Cref{example:ResGm}). By \Cref{prop:ResGm-SmComp}, $X:= \mathbb{P}^{p-1}$ is a smooth compactification of $G$, where the boundary divisor $D$ is given by the equation
$$\sum_{i=0}^{p-1} t^i X_i^{p}= X_0^p+tX_1^p+t^2X_2^p+\cdot\cdot\cdot + t^{p-1} X_{p-1}^p =0,$$
where $X_0,X_1,...,X_{p-1}$ are the projective coordinates of $\mathbb{P}^{p-1}$; we denote the defining polynomial by $f$. Let $\mathcal{G}\subset \pp_{\mathbb{F}_q[t]}$ be the $\mathbb{F}_q[t]$-model of $G$ defined by $f$. Note that $D$ is a geometrically irreducible, reduced, and regular divisor that is not geometrically reduced. To follow the notation in the previous sections, we will denote $D$ by $D_\beta$, so that $\mathcal{A}=\mathcal{B}=\{\beta\}$. It is clear that $\mathcal{O}(D_\beta)\cong \mathcal{O}(p) \cong \omega_X^{-1}$. Therefore, to follow our setting, we choose the class of $\mathcal{O}(D_\beta)$ as a basis for $\pic(X)_\cx\cong \cx [\mathcal{O}(D_\beta)]$. Note that the coordinate of  $\omega_X^{-1}$ with respect to this basis is $\rho=(\rho_\beta)=(1)$. 
We endow $\mathcal{O}(D_\beta)$ with the adelic $q$-metric $(\lVert \cdot \rVert_v)_{v\in \Omega_F}$ defined by 

$$\lVert \mathsf{s}(x_v) \rVert_v := \max \left\{ \frac{|X_0(x_v)|_v^p}{|\mathsf{s}(x_v)|_v} ,  \frac{|X_1(x_v)|_v^p}{|\mathsf{s}(x_v)|_v}, \cdot\cdot\cdot, \frac{|X_{p-1}(x_v)|_v^p}{|\mathsf{s}(x_v)|_v}, \frac{|f(x_v)|_v}{|\mathsf{s}(x_v)|_v} \right\}^{-1}$$
for any $\mathsf{s}\in \mathcal{O}(D_\beta)(U)$, any $U$ open in $X$, and any $x_v\in U(F_v)$. In particular, for the canonical section $f$ of $\mathcal{O}(D_\beta)$, we have
$$\lVert f(x_v) \rVert_v := \max \left\{ \frac{|X_0(x_v)|_v^p}{|f(x_v)|_v} ,  \frac{|X_1(x_v)|_v^p}{|f(x_v)|_v}, \cdot\cdot\cdot, \frac{|X_{p-1}(x_v)|_v^p}{|f(x_v)|_v}, 1 \right\}^{-1}.$$

Set $\mathfrak{o}_F:=\mathbb{F}_q[t]$ and identify $\spec(\mathfrak{o}_F)$ with the complement of the point at infinity in $\mathcal{C}_F=\mathbb{P}^1_{\mathbb{F}_q}$. The place corresponding to $1/t\in F$, which we denote by $v=\infty$, is identified with the point at infinity in $\mathbb{P}^1_{\mathbb{F}_q}$.  Let $\mathcal{X}:=\mathbb{P}^{p-1}_{\mathfrak{o}_F}$ be the $\mathfrak{o}_F$-model of $X$,  $\mathcal{D}_\beta$ be the divisor defined by $f$, and $\mathcal{G}:=\mathcal{X}\setminus \mathcal{D}_\beta$.

\begin{lem}
	The adelic $q$-metric $(\mathcal{O}(D_\beta), \lVert \cdot \rVert_v)_{v\in \Omega_F}$ is given by the integral model at all places $v\neq \infty$. For $v=\infty$, we have $\lVert f(x_\infty) \rVert_\infty=1$ for all $x_\infty \in X(F_\infty)$.
\end{lem} 
\begin{proof}
	Suppose that $v\neq\infty$. For any primitive point ${x_v=(x_0:\ldots:x_{p-1})\in \mathcal{X}(\ov)}$, we have $|X_i(x_v)|_v=1$ for some $i\in \{0,\ldots, p-1\}$; this implies that $\lVert \mathsf{s}(x_v) \rVert_v =|\mathsf{s}(x_v)|_v$ for any  $\mathsf{s}\in \mathcal{O}(D_\beta)(U)$ and any $U$ open in $X$  (in particular,
	$\lVert f(x_v) \rVert_v = |f(x_v)|_v$). This shows that the $v$-adic metric is given by the integral model for all places $v\neq \infty$.
	
	Let $v= \infty$ and  $x_\infty=(x_0:\ldots:x_{p-1})\in \mathcal{X}(\ov)$ be a primitive point. We have
	$$|f(x_\infty)|_\infty=|t|_\infty^{p-1} |\sum_{i=0}^{p-1} (1/t)^{p-1-i} X_i^{p} |_\infty= q^{p-1} |\sum_{i=0}^{p-1} (1/t)^{p-1-i} X_i^{p} |_\infty.$$
	Let $j\in \{0,\ldots p-1\}$ be the maximal index such that $x_j\in \ov^*$. Then we have $|\sum_{i=0}^{p-1} (1/t)^{p-1-i} X_i^{p} |_v= q^{-(p-1-j)}$, which implies that
	\begin{equation}\label{eqn:f_infty}
		|f(x_\infty)|= q^{p-1} q^{-(p-1-j)}= q^j.
	\end{equation}
	This shows that for all primitive points $x_\infty=(x_0:\ldots:x_{p-1})\in \mathcal{X}(\ov)$, we have that
	$\lVert f(x_\infty) \rVert_\infty=1,$ as desired.
\end{proof}

We now prove that \Cref{assum} is satisfied for $D_\beta$.

\begin{prop} \label{prop:P^p-1-valuations-insep}
	The following statements hold for $D_\beta$:
	\begin{enumerate}[label=(\roman*)]
		\item $D_\beta (F_v)=\emptyset$ for all $v\in \Omega_F$.
		\item For all places $v$ and for all $x\in X(F_v)$, we have $$\lVert f(x) \rVert_v \in \{1,q_v^{-1},\cdot \cdot \cdot, q_v^{-(p-1)}\};$$ 
		moreover, the value $\lVert f(x) \rVert_v$  depends only on $\redv(x)\in {\mathcal{X}}(\fv)$. 
		\item For any place $v\neq \infty$,
		$$\# \{\bar{x}\in {\mathcal{X}}(\fv):\> \lVert f (\redv^{-1}(\bar{x})) \rVert_v=q_v^{m} \} = q_v^{p-1-m}$$
		for $0\leq m \leq p-1$. In particular, we have that ${e_X=p-1}$.
	\end{enumerate}
\end{prop}
\begin{proof}
	$(i)$ As $G$ is $F$-wound, the statement follows by \Cref{prop:BoundaryEmpty}. 
	
	$(ii)$
	Note that we have a one-to-one correspondence between the finite place of $F$ and the irreducible polynomials of $\mathfrak{o}_F=\mathbb{F}_q[t]$; for a finite place $v$, we denote the corresponding irreducible polynomial by $\pi_v$, which is a uniformizing parameter for $F_v$. Let $v$ be any place different from  $\infty$. We wish to study $v(f(x_v))$ for any primitive point $x_v\in X(F_v)$. Let $S_v$ be the set of representatives of $\ov/\pi_v$ consisting of all polynomials in $\mathbb{F}_q[t]$ of degree less than $d_v:=\deg(\pi_v)$. Note that $S_v$ is an additive set of representatives, but not multiplicative. Every $y\in F_v$ has a unique representation
	$y=\sum_{j=-n}^{\infty} y_j \pi_v^j,$
	for some $n\in \z$ and $y_j\in S_v$, for all $j$. Hence, if $y\in \ov$, we can write $y=y_0+\pi_v y'$, for some $y_0\in S_v$ and $y'\in \ov$. Thus, $y^p = y_0^p + \pi_v^p y'^p$, which  implies that 
	$$y^p \equiv y_0^p \mod \pi_v^p,$$ 
	Let $(x_0:x_1:\cdot\cdot\cdot: x_{p-1})\in \mathbb{P}^{p-1}(\ov)$ be a primitive representative for the $x$ (i.e., $x_i\in \ov$ for all $i$, and there exists $j$ such that $x_j\in \ov^*$.) Then for all $i\in \{0,\ldots,p-1\}$ we have a unique representation 
	$$x_i:=\sum_{j=0}^{\infty} x_{i,j} \pi_v^j$$
	where $x_{i,j}\in S_v$ for all $j$. Therefore, as $t\in \ov$ for all $v\neq \infty$, we have
	$$f(x)= \sum_{i=0}^{p-1} t^i x_i^{p}\equiv  \sum_{i=0}^{p-1} t^i x_{i,0}^{p} \mod \pi_v^p.$$
	Thus, if  $v(\sum_{i=0}^{p-1} t^i x_{i,0}^{p} )<p$, then $v(f(x))=v(\sum_{i=0}^{p-1} t^i x_{i,0}^{p} )$. Hence, it suffices to prove that $v(\sum_{i=0}^{p-1} t^i x_{i,0}^{p} )<p$.  As $x_{i,0}\in S_v \subset \mathbb{F}_q[t] \subset F_v$, and $t \in \mathbb{F}_q[t] \subset F_v $,  we have $\sum_{i=0}^{p-1} t^i x_{i,0}^{p}\in \mathbb{F}_q[t] \subset{F_v}$ (by the inclusion $\mathbb{F}_q[t]\subset F_v$, we mean the the restriction of the embedding of $F$ inside $F_v$). Since the degree of the polynomial $x_{i,0}\in S_v$ is $\leq d_v-1$ for all $i$, the polynomial $\sum_{i=0}^{p-1} t^i x_{i,0}^{p}$ has degree at most $pd_v-1$; since the degree of the polynomial $\pi_v$ is $d_v$, this implies that $v(\sum_{i=0}^{p-1} t^i x_{i,0}^{p})\leq p-1$. Therefore, we have shown that  $$v(f(x))=v(\sum_{i=0}^{p-1} t^i x_{i,0}^{p} )< p,$$
	which proves $(ii)$.
	
	$(iii)$ Observe the following:
	\begin{itemize} [itemsep=4pt,parsep=1pt,topsep=4pt]
		\item We have an isomorphism of (additive) groups, given by our choice of representatives $S_v$ for $\fv$,
		$$\iota: \prod^{p} S_v \xrightarrow{\sim} \mathbb{G}_a^{p}(\fv), \quad (y_{0}, \cdot\cdot\cdot, y_{p-1}) \mapsto (\bar{y}_{0}, \cdot\cdot\cdot, \bar{y}_{p-1})$$
		where $\bar{y}_{i}$ is the image of $y_{i}$ under $\mathbb{F}_q[t] \subset \ov\rightarrow \ov/ \pi_v \cong \fv$. 
		\item Let $P_r\subset \mathbb{F}_q[t]$ be the additive subgroup of polynomials of degree at most $r$, where we set $P_{-1}:=\{0\}$. Note that the cardinality of $P_{pd_v-1}$ is $q^{pd_v}$. The map
		$$f: \prod^{p} S_v \rightarrow P_{pd_v-1}, \quad (y_0, \cdot\cdot\cdot, y_{p-1}) \mapsto \sum_{i=0}^{p-1} t^i y_{i}^{p}$$
		is a homomorphism of additive groups.
		Since $D_\beta (F)=\emptyset$, the homomorphism is injective. As the cardinality of $\prod^{p} S_v $ is equal to $q_v^{p}=q^{pd_v}$, the map is an isomorphism.
		\item Let $0\leq m \leq p-1$ be an integer. Then one sees that the cardinality of the subset
		$$P_{pd_v-1,m}:=\{ h(t)\in P_{pd_v-1}: \> \pi_v^m \mid h(t),\> \pi_v^{m+1} \nmid h(t)\}$$
		is equal to $$\# P_{d_v(p-m)-1} - \# P_{d_v(p-m-1)-1} = q^{d_v(p-m)} - q^{d_v(p-m-1)}= q_v^{p-m}- q_v^{p-m-1}.$$
		
		\item Let $\bar{y}:=(\bar{y}_0,\ldots, \bar{y}_{p-1})\in \mathbb{G}_a^p(\fv)$. For $\bar{a}\in \fv^*$, define $\bar{a} \bar{y}:=(\bar{a} \bar{y}_0,\ldots,\bar{a} \bar{y}_{p-1}) $. Then $\iota^{-1}(\bar{a}\bar{y})\equiv a \iota^{-1}(\bar{y}) \mod \prod^p \pi_v$ in $\prod^{p} F_v$ for any lift $a\in \ov^*$ of $\bar{a}$. Since $\sum_{i=0}^{p-1} t^i a^p y_{i}^{p}=a^p \sum_{i=0}^{p-1} t^i y_{i}^{p}$ where $(y_0,\ldots ,y_{p-1})=\iota^{-1}(\bar{y})$, we see that
		$$v( f(\iota^{-1}(\bar{a}\bar{y}) )= v( a^p f (\iota^{-1}(\bar{y}) ))= v( f_ \beta(\iota^{-1}(\bar{y}) ).$$
		
	\end{itemize}
	
	In conclusion, we have an isomorphism $$\tilde{f}:=f\circ \iota^{-1}: \mathbb{G}_a^{p}(\fv)\xrightarrow{\sim} P_{pd_v-1}$$
	with the property that $v(\tilde{f}((\bar{y}_0:\ldots: \bar{y}_{p-1})))$ is well defined for any projective point $(\bar{y}_0:\ldots: \bar{y}_{p-1})\in \mathbb{P}^{p-1}(\fv)\cong \mathbb{G}_a^{p}(\fv)/\fv^*$; also, if $y\in \redv^{-1}(\bar{y})$ then ${v(f(y))= v(\tilde{f}(\bar{y}))}$.
	Therefore, for $1\leq m\leq p-1$, we have
	\begin{equation*}
		\begin{split}
			\# \{\bar{x}\in \mathbb{P}^{p-1}(\fv) :\> v(\tilde{f}(\bar{x}))=m\}  &=\frac{1}{q_v-1} \# \{\bar{x}\in \mathbb{G}_a^{p}(\fv) :\>  v(\tilde{f}(\bar{x})))=m \} 
			\\&=\frac{1}{q_v-1} \# \{\bar{x}\in \mathbb{G}_a^{p}(\fv) \>:\> \tilde{f}(\bar{x}) \in P_{pd_v-1,m} \}
			\\&= \frac{1}{q_v-1} \#P_{pd_v-1,m} 
			\\&=  \frac{1}{q_v-1} (q_v^{p-m}-q_v^{p-m-1})
			\\&= q_v^{(p-1)-m}
		\end{split}
	\end{equation*}
	This completes the proof of $(iii)$.
\end{proof}

Let $U_j\subset \pp_F$, for $j\in \{0,\ldots, p-1\}$, be the affine patch given by ${X_j\neq 0}$ with coordinates $\{x_{i/j} \}_{i\neq j}$, where $x_{i/j}=X_i/X_j$. We can view ${f_j:=X_j^{-(p-1)}f(X_0,\ldots, X_{p-1})}$ as a polynomial in the variables $\{x_{i/j} \}_{i\neq j}$. The measures 
${|f_j|_v^{-1} \bigwedge_{i\neq j} dx_{i/j}}$ glue to a Haar measure $d\mathbf{g}_v$ on $G(F_v)$; this induces the Haar measure $d\mathbf{g}:=\prod_v d\mathbf{g}_v$ on $G(\ad)$. Hence, the local Tamagawa measure $\tau_{X,v}$ on $X(F_v)$ is equal to the measure induced by $\lVert f \rVert_v d\mathbf{g}_v$ (c.f. \Cref{TamHaarMeasure}).

We now show that \Cref{condition} holds. First, we prove the following lemma.
\begin{lem}\label{lem:volFinfty}
	We have that $\vol(G(F_\infty):d\mathbf{g}_\infty)=p\cdot q^{-(p-1)}$.
\end{lem}
\begin{proof}
	First, note that we can identify $G(F_v)$ with the set of primitive points of $\pp(\ov)$. We define the subset $$B_m:=\{(x_0:\ldots:x_{p-1})\in \pp(\ov) : x_m \in \ov^* \text{ and } x_j\notin \ov^*, \forall j>m \},$$
	which is contained in the affine patch $U_m$.
	Then $\{B_i\}_{i=0}^{p-1}$ gives a partition for $\pp(\ov)$.
	Now for $x\in B_m$, we see that $|f(x)|_\infty= q^m$ by \eqref{eqn:f_infty}. Therefore,
	\begin{equation}
		\begin{split}
			\int_{g\in G(F_\infty)} d\mathbf{g}_\infty &= \sum_{m=0}^{p-1} \int_{x\in B_m\subset U_m} |f_m(x)|^{-1}_\infty \bigwedge_{i\neq m} dx_{i/m}
			\\&= \sum_{m=0}^{p-1} q^{-m} \int_{x\in B_m\subset U_m} \bigwedge_{i\neq m} dx_{i/m}
			\\&= \sum_{m=0}^{p-1} q^{-m} q^{m-(p-1)}=pq^{-(p-1)} 
		\end{split}
	\end{equation}
	as desired.
\end{proof}

Let $\mathbb{K}\subset G(\ad)$ be the maximal compact open subgroup stabilizing the compatible system of heights
$$H: G(\ad) \times \pic(X)_\mathbb{C} \rightarrow \mathbb{C}, \>\>\> ( x ;  s ) \mapsto \prod_v \lVert f \rVert (x_v)^{-s}.$$
By our choice of the adelic metrics and \Cref{prop:K-stabilizer}, we see that
$$\mathbb{K}= G(F_\infty)\times \prod_{v\neq \infty} \mathcal{G}(\ov).$$

\begin{prop}\label{prop:trivialOnF+K}
	We have the following:
	\begin{enumerate}
		\item $\mathbb{K}\cap G(F)=\{e_G\}$, where $e_G$ is the identity of $G(F)$.
		\item $\tau_G(\mathbb{K})=p$.
		\item $G(\ad)=G(F)\oplus \mathbb{K}$.
	\end{enumerate}
	In particular, $\{\Psi\in \widehat{G(\ad)}:\> \Psi(\mathbb{K}+G(F))=1\}$ consists of the trivial character only.
\end{prop}
\begin{proof}
	(1) Let $G'$ be the closed subgroup of $\mathbb{G}_a^p$ defined by
	$$y_0^p+ty_1^p+t^2 y_2^p + \cdots+ t^{p-1}y_{p-1}= y_{p-1}.$$
	We fix an $\mathfrak{o}_F$-model $\mathcal{G}'$ of $G'$ which is given by the same equation.
	By \cite[Proposition VI.5.3]{Ost}, there is an explicit $F$-groups isomorphism $$\phi: G \rightarrow G', \> (x_0:\cdots : x_{p-1})\rightarrow (F_0(x_0\cdots, x_{p-1}), \cdots, F_{p-1}(x_0,\cdots, x_{p-1})),$$
	where each $F_i$ is a quotient of two homogenous polynomials of degree $p$ having denominator equal to $x_0^p +tx_1^p + \cdots t^{p-1} x_{p-1}^p$. This shows that $\phi(\mathcal{G}(\ov))\subset \mathcal{G}'(\ov)$ for all $v\neq \infty$. Now, by \Cref{prop:SA-close-property}, 
	\begin{equation}\label{eqn:adelesDecom}
		G'(\ad)= G'(F) \oplus \left( \mathbb{H}_\infty \times \prod_{v\neq \infty}\mathcal{G}'(\ov) \right)
	\end{equation}
	where $\mathbb{H}_\infty=\mathfrak{m}_\infty^{p} \cap G'(F_\infty)$, the intersection being taken in $\mathbb{G}_a^{p}(F_\infty)$. In fact, by \eqref{eqn:f_infty}, we have that the valuation of $x_0^p +tx_1^p + \cdots t^{p-1} x_{p-1}^p$ at $v=\infty$ is strictly negative so that $v(F_i(x))$ is strictly positive for all $i$ and all $x\in G(F_\infty)$. Thus, $\mathbb{H}_\infty= G'(F_\infty)$. Therefore, $\phi(\mathbb{K})\subset  G(F_\infty) \times \prod_{v\neq \infty}\mathcal{G}'(\ov)$; by \eqref{eqn:adelesDecom}, this implies that $\mathbb{K}\cap G(F)=\{e_G\}$.
	
	$(2)$ By formula \eqref{Tau_G}, we have
	$$\tau_G(\mathbb{K})= q^{-\Dim(G)(g_F-1)} \left( \int_{G(F_v)} d\mathbf{g}_v \right) \prod_{v\neq \infty } \int_{\mathcal{G}(\ov)} d\mathbf{g}_v = q^{(p-1)}\cdot (p q^{-(p-1)}) = p,$$
	using that $\int_{\mathcal{G}(\ov)} d\mathbf{g}_v=1$, for all $v\neq \infty$, and that $\int_{G(F_v)} d\mathbf{g}_v=p q^{-(p-1)}$ by \Cref{lem:volFinfty}.
	
	(3) Note that $\tau_G(G(\ad)/G(F))=p$ by \cite[Corollary VI.5.4]{Ost}. Now, since $\tau_G((G(F)\oplus\mathbb{K})/G(F))=\tau_G(\mathbb{K})=p$ and $G(F)\oplus\mathbb{K}$ is cocompact in $G(\ad)$, we deduce that $G(\ad)=G(F)\oplus\mathbb{K}$.
\end{proof}

Therefore, we have shown that $X$ satisfies the conditions of \Cref{prop:HeightZetaPoles}. As $\pic(G)\cong \z/p\z$ and $\rho_\beta=1$, we have
$$\alpha^*(X)= \frac{ \rho_\beta^{-1}}{|\pic(G)|}=1/p.$$
Therefore, by \Cref{thm:Manin}, \Cref{thm:Leading-constant}, and the fact that $d_\rho=\rho_\beta=1$, we see that
$$ N_{\av}(\omega_X^{-1}, M) = N(\omega_X^{-1}, M)  \sim c_\rho { \log(q)} q^{M}, \quad \text{ as } M\rightarrow \infty,$$
where 
$c_\rho=  \tau_X(X(\ad))/p.$ We now  complete the proof of \Cref{thm:resGm}

\begin{prop}\label{prop:tau(X(ad))}
	We have
	$$\tau_X(X(\ad)) =\frac{p}{\log q} \prod_{v\neq \infty} \left(1-\frac{1}{q_v}\right) \left(1+ \frac{1}{q_v} + \frac{1}{q_v^2} +\cdots + \frac{1}{q_v^{p-1}} \right).$$
\end{prop}
\begin{proof}
	By \Cref{defn:Peyre's-constant}, we have
	$$\tau_X(X(\ad))=q^{(p-1)} \residue_{s=1} \zeta_{F}(s) \times \prod_v (1-{q_v}^{-1}) \cdot \tau_{X,v}(X(F_v)) .$$
	By \cite[Theorem 5.9]{Rosen}, we see that $\residue_{s=1} \zeta_{F}(s)= q/((q-1)\log q)$. As $d\tau_{X,v}=\lVert f \rVert_v d\mathbf{g}_v$ and $\lVert f \rVert_\infty=1$, we have $\vol(G(F_\infty):d\tau_{X,\infty})=p q^{-(p-1)}$ by \Cref{lem:volFinfty}. Recall from \eqref{eqn:TauXH} that $$\tau_{X,v}(X(F_v))=\hat{H}_v(\mathbf{0};\rho).$$
	Let $v\neq \infty$. By applying Denef's formula (\Cref{thm:Denef}) and using \Cref{prop:P^p-1-valuations-insep}, we have 
	\begin{equation*}
		\begin{split}
			\hat{H}_v(\mathbf{0}; s)&= 1+ q_v^{-\Dim(X)} \sum_{\bar{x}\in \mathcal{D}_\beta (\fv)} q_v^{\beta(\bar{x})(1-s)}. 
			\\ &=  1+ q_v^{-\Dim(X)} \sum_{m=1}^{p-1} q_v^{m(1-s)} \# \{\bar{x}\in \mathcal{X}(\fv):\> \lVert f \rVert_v (x)=q_v^{m}, \> \forall x\in \redv^{-1}(\bar{x}) \} 
			\\&= 1+  \sum_{m=1}^{p-1} q_v^{m(-s)} = 1+ q_v^{-s} +q_v^{-2s} +\cdots + q_v^{-(p-1)s}.
		\end{split}
	\end{equation*}
	for all $s\in \cx$ with $\re(s)>0$. As $\rho=(1) \in \cx [\mathcal{O}(D_\beta)] \cong \pic(X)_\cx$, the result follows by a straightforward computation. 
\end{proof}


\begin{thebibliography}{9}
	

	
	\bibitem{Achet} R. Achet, \textit{Picard group of the forms of the affine line and of the additive group}, Journal of Pure and Applied Algebra Volume \textbf{221}(11)  (2017), 2838-2860.
	
	\bibitem{Achet2} R. Achet, \textit{Unirational Algebraic Groups}, 2019, available at https://hal.archives-ouvertes.fr/hal-02358528/document.
	
	\bibitem{Bastos} G. Bastos, \textit{Some results on the degree of imperfection of complete valued fields}, Manuscripta Math \textbf{25} (1978), 315–322.
	
	\bibitem{BatManin} V. Batyrev and Y. Manin, \textit{Sur le nombre des points rationnels de hauteur bornée des variétés algébriques}, Math. Ann. \textbf{286} (1990), 27-43.
	
	\bibitem{BeckerMac}	M. Becker and S. MacLane, \textit{The minimum number of generators for in- separable extensions}, Bull. Amer. Math. Soc.46 (1940), 182-186.
	
	\bibitem{Bor1} A. Borel, \textit{Linear algebraic groups}, Second enlarged edition, Gra. Texts Math. \textbf{126}, Springer, 1991.
	
	\bibitem{Neron} S. Bosch, W. Lutkebohmert, M. Raynaud, \textit{Néron Models}, Ergebnisse der Math. Springer Heidelberg, 21, 1990.
	
	\bibitem{BGM20} S. Boucksom, W. Gubler and F. Martin, 
		\textit{Non-Archimedean volumes of metrized nef line bundles},
		{{\'E}pijournal de G{\'e}om{\'e}trie Alg{\'e}brique},
		2020.
	
	\bibitem{BrowningDP} T. Browning, \textit{An overview of Manin’s conjecture for del Pezzo surfaces}, \textit{In Analytic number theory}, volume 7 of Clay Math. Proc., pages 39–55. Amer. Math. Soc., Providence, RI, 2007.
	
	\bibitem{Bri} M. Brion, \textit{On Linearization of Line Bundles}, J. Math. Sci. Univ. Tokyo \textbf{22} (2015), 113–147.
	
	
	\bibitem{CLYTig} A. Chambert-Loir, Y. Tschinkel, \textit{Igusa integrals and volume asymptotics in analytic and adelic geometry}, Confluences
	Math. \textbf{2}(3) (2010), 351–429.
	
	\bibitem{CLYT} A. Chambert-Loir, Y. Tschinkel, \textit{On the distribution of points of bounded height on equivariant compactifications of vector groups}, Invent. math. \textbf{148} (2002), 421–452.
	
	\bibitem{CLSurvey} A. Chambert-Loir, \textit{Lectures on height zeta functions: At the confluence of algebraic geometry,	algebraic number theory, and analysis}, MSJ Memoirs \textbf{21} (2010), 17-49.
	
	\bibitem{CTSK} J. Colliot-Thélène, A. Skorobogatov, \textit{The Brauer–Grothendieck Group}, Ergebnisse der Mathematik und ihrer Grenzgebiete. 3. Folge / A Series of Modern Surveys in Mathematics, Springer, 2021.
	
	\bibitem{CGP} B. Conrad, O. Gabber, G. Prasad, Pseudo-reductive Groups, Cambridge Univ. Press (2nd edition), 2015.
	
	\bibitem{Conrad1}  B. Conrad, \textit{Finiteness theorems for algebraic groups over function fields,} Compositio Math. \textbf{148} (2012), 555-639.
	
	\bibitem{Denef} J. Denef. \textit{On the degree of Igusa’s local zeta function}, Amer. J. Math. \textbf{109} (1987), 991–1008.
	
	\bibitem{Azur} Azur Đonlagić \textit{Brauer-Manin obstructions for homogeneous spaces of commutative affine algebraic groups over global fields}, arXiv:2410.12127 (2024).
	
	\bibitem{FMT} J. Franke, Y. Manin, Y. Tschinkel,\textit{ Rational points of bounded height on Fano varieties}. Invent. Math. \textbf{95} (1989), 421–435.
	
	\bibitem{Herrero} S. Herrero, T. Martínez, P. Montero, \textit{Counting rational points on Hirzebruch-Kleinschmidt varieties over global function fields}, arXiv:2408.07631 (2024).
	
	 \bibitem{Illusie} L. Illusie and M. Temkin, Exposé X. Gabber’s modification theorem (log smooth case), Astérisque
	363-364 (2014), 167–212. Travaux de Gabber sur l’uniformisation locale et la cohomologie étale des schémas quasi-excellents.
	
	\bibitem{KMT} T. Kambayashi, M. Miyanishi, M. Takeuchi, \textit{Unipotent Algebraic Groups}, Lecture Notes in Mathematics, vol 414. Springer, Berlin, Heidelberg (1974).
	
	\bibitem{LangWeil} S. Lang and A. Weil, \textit{Number of points of varieties in finite fields}, Amer. J. Math. \textbf{76} (1954), 819–827.
	
	\bibitem{Liu} Qing Liu, \textit{Algebraic Geometry and Arithmetic Curves}, Oxford University Press, 2006.
	
	MacLane, The minimum number of generators for in- separable extensions, Bull. Amer. Math. Soc. 46 (1940), 182-186
	
	\bibitem{ManinPan} Y. Manin and A. Panchishkin, \textit{Number theory I. Introduction to number theory}, Springer,
	Berlin, 1995.
	
	\bibitem{Mumford}  D. Mumford, J. Fogarty, F. Kirwan,\textit{ Geometric invariant theory}. Third edition, Ergeb. Math. Grenzgeb.  \textbf{34}(2), Springer-Verlag, Berlin, 1994.
	
	\bibitem{Peyre1} E. Peyre, \textit{Hauteurs et mesures de Tamagawa sur les variétés de Fano}, Duke Math. J. \textbf{79}(1) (1995), 101–218.
	
	\bibitem{Peyre} E. Peyre, \textit{Points de hauteur bornée sur les variétés de drapeaux en caractéristique finie}, Acta Arithmetica \textbf{152}(2) (1995), 185–216.
	
	\bibitem{Campana} M. Pieropan, A. Smeets, S. Tanimoto, A. Várilly-Alvarado, \textit{Campana points of bounded height on vector group compactifications}, Proc. London Math. Soc. \textbf{123}(3)  (2021), 57–101.
	
	\bibitem{Poonen}
	B. Poonen, \textit{Rational points on varieties}, Graduate Studies in Mathematics
	186, American Mathematical Society, Providence, 2017.
	
	\bibitem{Rosen} M. Rosen, \textit{Number Theory in Function Fields}, Springer-Verlag, New York, 2002.
	
	\bibitem{Rosenlicht} M. Rosenlicht, 
		\textit{Toroidal algebraic groups},
		Proceedings of the American Mathematical Society
		 \textbf{12} (6)  (1961), 984-988.
	
	\bibitem{ZRos1} Z. Rosengarten, \textit{Pathological Behavior of Arithmetic Invariants of Unipotent Groups}, Algebra and Number Theory \textbf{15}(7) (2021), 1593-1626.
	
	\bibitem{ZRos} Z. Rosengarten, \textit{Tamagawa Numbers And Other Invariants of Pseudo-reductive Groups Over Global
	Function Fields}, Algebra and Number Theory \textbf{15}(8) (2021), 1865–1920.
	
	\bibitem{RosNN} Z. Rosengarten, N. Tân, N. Thang, \textit{On The Galois And Flat Cohomology Of Unipotent Algebraic Groups Over Local And Global Function Fields II}, Michigan Mathematical Journal Advance Publication (2023).
	
	
	\bibitem{Serre1}  J.-P. Serre, \textit{Local Fields}, Graduate Texts in Mathematics, vol 67, 1979.
	
	
	\bibitem{Sal} P. Salberger, \textit{Tamagawa measures on universal torsors and points of bounded height on Fano varieties}, Astérisque \textbf{251} (1998), 91–258. Nombre et répartition de points de hauteur bornée
	(Paris, 1996).
	
	\bibitem{Ti1} J. Tits,\textit{ Lectures on algebraic groups}, notes by P. André and D. Winter. Yale University, 
	1968.
	
	\bibitem{StacksEx} The Stacks Project Authors, \textit{Stacks Project}, 2024. 
	
	\bibitem{Ost} J. Oesterlé,\textit{ Nombres de Tamagawa et groupes unipotents en caractéristique p}, Inventiones mathematicae \textbf{78}(1) (1984), 13–88.
	
	\bibitem{Weil1} A. Weil, \textit{Basic Number Theory}, Springer Berlin Heidelberg, 1973.
	
	\bibitem{Weil2} A. Weil, \textit{Adeles and algebraic groups}, Progr. Math., no. 23, Birkhäuser, 1982.
	

	
	
\end{thebibliography}
\end{document}